\documentclass[reqno, 12pt]{amsart}
\pdfoutput=1
\makeatletter
\let\origsection=\section \def\section{\@ifstar{\origsection*}{\mysection}} 
\def\mysection{\@startsection{section}{1}\z@{.7\linespacing\@plus\linespacing}{.5\linespacing}{\normalfont\scshape\centering\S}}
\makeatother        

\usepackage{amsmath,amssymb,amsthm}
\usepackage{mathrsfs}
\usepackage{mathabx}\changenotsign
\usepackage{dsfont}

\usepackage{todonotes}
\usepackage{verbatim}
 
\usepackage{graphicx}
\usepackage{tikz}
\usepackage{xcolor}

\usepackage[backref]{hyperref}
\hypersetup{
    colorlinks,
    linkcolor={red!60!black},
    citecolor={green!60!black},
    urlcolor={blue!60!black},
    pdftitle={Characterising k-connected sets in infinite graphs},
    pdfauthor={J. Pascal Gollin and Karl Heuer}
}

\usepackage[open,openlevel=2,atend]{bookmark}

\usepackage[abbrev,msc-links,backrefs]{amsrefs} 
\usepackage{doi}

\renewcommand{\PrintDOI}[1]{\doi{#1}}

\usepackage[T1]{fontenc}
\usepackage{lmodern}
\usepackage[babel]{microtype}
\usepackage[english]{babel}

\usepackage{geometry}
\numberwithin{equation}{section}
\numberwithin{figure}{section}

\usepackage{enumitem}

\let\polishlcross=\l
\def\l{\ifmmode\ell\else\polishlcross\fi}

\def\paragraph#1{%
  \noindent\textbf{#1.}\enspace}

\let\emptyset=\varnothing
\let\setminus=\smallsetminus

\makeatletter
\def\moverlay{\mathpalette\mov@rlay}
\def\mov@rlay#1#2{\leavevmode\vtop{   \baselineskip\z@skip \lineskiplimit-\maxdimen
   \ialign{\hfil$\m@th#1##$\hfil\cr#2\crcr}}}
\newcommand{\charfusion}[3][\mathord]{
    #1{\ifx#1\mathop\vphantom{#2}\fi
        \mathpalette\mov@rlay{#2\cr#3}
      }
    \ifx#1\mathop\expandafter\displaylimits\fi}
\makeatother

\DeclareFontFamily{U}  {MnSymbolC}{}
\DeclareSymbolFont{MnSyC}         {U}  {MnSymbolC}{m}{n}
\DeclareFontShape{U}{MnSymbolC}{m}{n}{
    <-6>  MnSymbolC5
   <6-7>  MnSymbolC6
   <7-8>  MnSymbolC7
   <8-9>  MnSymbolC8
   <9-10> MnSymbolC9
  <10-12> MnSymbolC10
  <12->   MnSymbolC12}{}
\DeclareMathSymbol{\powerset}{\mathord}{MnSyC}{180}

\let\epsilon=\varepsilon

\let\rho=\varrho
\let\theta=\vartheta

\theoremstyle{plain}
\newtheorem{thm}{Theorem}[section]
\newtheorem{theorem}[thm]{Theorem}
\newtheorem{lemma}[thm]{Lemma}
\newtheorem{corollary}[thm]{Corollary}

\newtheorem{thm-intro}{Theorem}[]

\theoremstyle{definition}
\newtheorem{remark}[thm]{Remark}
\newtheorem{remarks}[thm]{Remarks}
\newtheorem{obs}[thm]{Observation}
\newtheorem{situation}[thm]{Situation}

\usepackage{accents}

\let\phi=\varphi

\DeclareMathOperator{\dom}{dom}
\DeclareMathOperator{\Dom}{Dom}
\DeclareMathOperator{\cf}{cf}

\newcommand{\mcm}[3]{\newcommand{#1}[#2]{{\ensuremath{#3}}}} 
\mcm{\abs}{1}{\left\lvert #1 \right\rvert}
\mcm{\lFK}{0}{\ell\text{--}FK_{k,\kappa}}
\mcm{\lK}{1}{\ell\text{--}K(k,#1)}
\mcm{\restricted}{0}{{\upharpoonright}}

\begin{document}

\author{J. Pascal Gollin}
\address[Gollin]{Discrete Mathematics Group, Institute for Basic Science (IBS), 55 Expo-ro, Yuseong-gu, Daejeon, Korea, 34126}
\email{\tt pascal.gollin@uni-hamburg.de}

\author{Karl Heuer}
\address[Heuer]{Institut f\"{u}r Softwaretechnik und Theoretische Informatik, Technische Universit\"{a}t Berlin, Ernst-Reuter-Platz 7, 10587 Berlin, Germany}
\email{\tt karl.heuer@tu-berlin.de}

\title{Characterising $k$-connected sets in infinite graphs}

%\date\today



\begin{abstract}
    A \emph{$k$-connected set} in an infinite graph, where~${k > 0}$ is an integer, is a set of vertices such that any two of its subsets of the same size~${\ell \leq k}$ can be connected by~$\ell$ disjoint paths in the whole graph.
    
    We characterise the existence of $k$-connected sets of arbitrary but fixed infinite cardinality via the existence of certain minors and topological minors. 
    We also prove a duality theorem for the existence of such $k$-connected sets:
    if a graph contains no such $k$-connected set, then it has a tree-decomposition which, whenever it exists, 
    precludes the existence of such a $k$-connected set.
\end{abstract}

\maketitle

\setcounter{footnote}{1}

\section{Introduction}
\label{sec:intro}

A common aspect of structural graph theory is the study of the duality between connectivity and tree structure. 
Such type of duality theorems assert that if a graph contains no `highly connected part', then there is some kind of tree structure with certain properties, usually a tree-decomposition of the graph, that, if it exists, clearly precludes the existence of such a `highly connected part'. 
Some of the more well-known examples include the duality between brambles and tree-width, as well as tangles and branch-width. 

One of the notions of connectivity, which has been studied in finite graphs, is the one of a so called $k$-connected set. 
For~${k \in \mathbb{N}}$, a set~$X$ of at least~$k$ vertices of a graph~$G$ is called {\emph{$k$-connected in~$G$}}, 
if for all~${Z_1, Z_2 \subseteq X}$ with~${\abs{Z_1} = \abs{Z_2} \leq k}$ 
there are~$\abs{Z_1}$ many vertex disjoint paths from~$Z_1$ to~$Z_2$ in~$G$.
We often omit stating the graph in which~$X$ is $k$-connected if it is clear from the context. 

In finite graphs, $k$-connected sets have also been studied in connection to tree-width.
This connection was first observed by Robertson, Seymour and Thomas \cite{RST:planar}, and later improved by Diestel, Gorbunov, Jensen and Thomassen \cite{DJGT:k-con&grid}*{Prop.~3}, who showed 
that for any finite graph~$G$ and~${k \in \mathbb{N}}$, 
if~$G$ contains a $(k+1)$-connected set of size at least $3k$, 
then~$G$ has tree-width at least~$k$, 
and conversely if~$G$ has no $(k+1)$-connected set 
of size at least~$3k$, then~$G$ has tree-width less than~$4k$.

Recently, Geelen and Joeris~\cites{Joeris:phd-thesis, GJ:grid} studied the duality between $k$-connected sets and \emph{$k$-tree-width}, that is the analogue of tree-width when only considering tree-decompositions of adhesion less than~$k$: 
They showed that the maximum size of a $k$-connected set is bounded from below by the $k$-tree-width~$w$ and from above by~${\binom{w+1}{k-1} (k-1)}$ \cite{GJ:grid}*{Thm.~1.2}. 

As our first main theorem, we generalise their theorem to a duality theorem for infinite $k$-connected sets and $k$-tree-width, for which we say a graph~$G$ has \emph{$k$-tree-width~$\kappa$} if~$\kappa$ is the smallest cardinal such that there is a tree-decomposition of~$G$ of adhesion less than~$k$ for which each part has size less than~$\kappa$.

\begin{thm-intro}\label{main-thm-simple-1}
    Let~$G$ be an infinite graph, 
    let~${k \in \mathbb{N}}$ 
    and let~${\kappa \leq \abs{V(G)}}$ be an infinite cardinal.
    Then the following statements are equivalent.
    \begin{enumerate}[label=(\alph*)]
        \item\label{item:t1-set} ${V(G)}$ contains a subset of size~$\kappa$ that is $k$-connected in~$G$.
        \item\label{item:t1-treeset} The $k$-tree-width of~$G$ is greater than~$\kappa$.
    \end{enumerate}
\end{thm-intro}

\vspace{0.2cm}

Our second main result, Theorem~\ref{main-thm-simple-2}, describes how $k$-connected sets ``look like'' by characterising their existence with the existence of certain unavoidable (topological) minors\footnote{Since the non-existence of a large $k$-connected set in a graph is a property which is closed under both the minor and topological minor relation, the existence of $k$-connected sets in a graph is a property which is well-suited to be characterised via the existence of certain (topological) minors.}.

It is a well-known and easy-to-prove fact that every large connected finite graph contains a long path or a vertex of high degree. 
More precisely, for every~${m \in \mathbb{N}}$ there is an~${n \in \mathbb{N}}$ such that each connected graph with at least~$n$ vertices 
either contains a path~$P_m$ of length~$m$ or a star~$K_{1, m}$ with~$m$ leaves as a subgraph (cf.~\cite{Diestel:GT5}*{Prop.~9.4.1}).
In a sense which can be made precise \cite{Diestel:GT5}*{Thm.~9.4.5}, the existence of these `unavoidable' subgraphs characterises 
connectedness with respect to the subgraph relation in a minimal way:
in every infinite collection consisting of graphs with arbitrarily large connected subgraphs we find arbitrarily long paths or arbitrarily large stars as subgraphs; 
but with only paths or only stars this would not be true, 
since long paths and large stars do not contain each other.

For $2$-connected graphs there is an analogous result, which also is folklore: 
For every~${m \in \mathbb{N}}$ there is an~${n \in \mathbb{N}}$ such that every $2$-connected finite graph with at least~$n$ vertices either contains 
a subdivision of a cycle~$C_m$ of length~$m$
or a subdivision of a complete bipartite graph~$K_{2, m}$ 
\cite{Diestel:GT5}*{Prop.~9.4.2}.
As before, cycles and~$K_{2,m}$'s as unavoidable topological minors characterise $2$-connectedness with respect to the topological minor relation in a minimal way (cf. \cite{Diestel:GT5}*{Thm.~9.4.5}).

In 2016, 
Geelen and Joeris~\cites{Joeris:phd-thesis, GJ:grid} 
generalised these results to arbitrary~${k \in \mathbb{N}}$.
For this, they relaxed `$k$-connectedness' to containing a large $k$-connected set.
They introduced certain graphs called generalised wheels (depending on~$k$ and~$m$), which together with the complete bipartite graph~$K_{k, m}$ are the unavoidable minors: 
they contain large $k$-connected sets themselves and they characterise graphs that contain large $k$-connected sets with respect to the minor relation.

This result encompasses the characterisation for $k=2$ as mentioned above, 
as well as earlier results from Oporowski, Oxley and Thomas~\cite{OOT:3&4-conn}, 
who proved similar results for ${k = 3}$ and ${k = 4}$ (albeit with different notions of `$k$-connectedness'). 

\vspace{0.2cm}

Now let us consider infinite graphs.
Again there is a well-known and easy-to-prove fact that each infinite connected graph contains either a \emph{ray}, that is a one-way infinite path, or a vertex of infinite degree.
This can also be seen as a characterisation of infinite connected graphs via these two unavoidable subgraphs: a ray and the complete bipartite graph $K_{1,\aleph_0}$.
There is also a more localised version of this result: the Star-Comb Lemma
(cf.~Lemma~\ref{star-comb}).
In essence this lemma relates these subgraphs to a given vertex set.

For $2$-connected infinite graphs one can easily construct an analogous result. 
We say a vertex~$d$ \emph{dominates} a ray~$R$ if they cannot be separated by deleting a finite set of vertices not containing~$d$. 
Given a ray~$R$ and the complete graph on two vertices~$K_2$, the \emph{one-way infinite ladder} is the graph ${R \times K_2}$. 

Now it is a common exercise to prove that 
the unavoidable (topological) minors for $2$-connectivity are 
the one-way infinite ladder, the union of a ray~$R$ with a complete bipartite graph between a single vertex and $V(R)$\footnote{Note that with the advent of topological infinite graph theory, those results became an even more meaningful extension of the finite result as these unavoidable minors correspond to infinite cycles in locally finite or finitely separable graphs (cf.~\cite{Diestel:GT5}*{Section~8.6}) and~\cite{Diestel:topGTsurvey}*{Section~5}).} 
as well as the complete bipartite graph $K_{2, \aleph_0}$.

In 1978, Halin \cite{Halin:simplicial-decomb} studied such a problem for arbitrary~${k \in \mathbb{N}}$. 
He showed that every $k$-connected graph whose set of vertices has size at least~$\kappa$ for some uncountable regular cardinal~$\kappa$ contains a subdivision of~$K_{k, \kappa}$.
Hence for all those cardinals,~$K_{k, \kappa}$ is the unique unavoidable topological minor characterising graphs with a subdivision of a $k$-connected graph of size~$\kappa$.
In a way, this characterisation result is stronger than the results previously discussed, since it obtains a direct equality between the size of a $k$-connected set and the size of the minors which was not possible for finite graphs.
The unavoidable (topological) minors for graphs whose set of vertices has singular cardinality remained undiscovered.

Oporowski, Oxley and Thomas \cite{OOT:3&4-conn} also studied countably infinite graphs for arbitrary~${k \in \mathbb{N}}$, but again for a different notion of `$k$-connectedness'.
Together with the~$K_{k, \aleph_0}$, the unavoidable minors for countably infinite essentially $k$-connected\footnote{A graph is \emph{essentially $k$-connected} if there is a constant $c \in  \mathbb{N}$ such that for each separation $(A,B)$ of order less than $k$ one of $A$ or $B$ has size less than~$c$. 
As before, we will not use this notion in this paper.} 
graphs have the following structure.
For~${\ell, d \in \mathbb{N}}$ with~${\ell + d = k}$, they consist of a set of~$\ell$ disjoint rays,~$d$ vertices that dominate one of the rays (or equivalently all of those rays) 
and infinitely many edges connecting pairs of them in a tree-like way.

This leads to our second main result, Theorem~\ref{main-thm-simple-2}.
For~${k \in \mathbb{N}}$ and an infinite cardinal~$\kappa$ we will define certain graphs with a $k$-connected set of size~$\kappa$ in Section~\ref{sec:typical}, the so called \emph{$k$-typical graphs}.
These graphs will encompass complete bipartite graphs~$K_{k, \kappa}$ as well as the graphs described by Oporowski, Oxley and Thomas \cite{OOT:3&4-conn} for~${\kappa = \aleph_0}$.
We will moreover introduce such graphs even for singular cardinals~$\kappa$.
It will turn out that for fixed~$k$ and~$\kappa$ there are only finitely many $k$-typical graphs up to isomorphisms.
We shall characterise graphs with a $k$-connected set of size~$\kappa$ via the existence of a minor of such a $k$-typical graph with a $k$-connected set of size~$\kappa$. 
In contrast to the finite case, the minimality of the list of these graphs in the characterisation is implied by the fact that it really is a finite list of graphs (for which even if it were not minimal we could pick a minimal sublist), and not a finite list of `classes of graphs', like `paths' and `stars' in the finite case for~${k=1}$.

Moreover we will extend the definition of $k$-typical graphs to so called \emph{generalised $k$-typical graphs}. 
As before for fixed~$k$ and~$\kappa$ there are only finitely many generalised $k$-typical graphs up to isomorphisms, and we shall extend the characterisation from before to be with respect to the topological minor relation.

\begin{thm-intro}\label{main-thm-simple-2}
    Let~$G$ be an infinite graph, 
    let~${k \in \mathbb{N}}$ 
    and let~${\kappa \leq \abs{V(G)}}$ be an infinite cardinal.
    Then the following statements are equivalent.
    \begin{enumerate}[label=(\alph*)]
        \item\label{item:t2-set} ${V(G)}$ contains a subset of size~$\kappa$ that is $k$-connected in~$G$.
        \item\label{item:t2-minor} $G$ contains a $k$-typical graph of size~$\kappa$ as a minor with finite branch sets.
        \item\label{item:t2-subdivision} $G$ contains a subdivision of a generalised $k$-typical graph of size~$\kappa$.
    \end{enumerate}
\end{thm-intro}

In fact, we will prove a slightly stronger result which will require some more notation, Theorem~\ref{main-thm} in Subsection~\ref{subsec:main-thm}.
In the same vein as the Star-Comb Lemma, that result will relate the minors (or subdivisions) with a specific $k$-connected set in the graph. 

\vspace{0.2cm}

After fixing some notation and recalling some basic definitions and simple facts in Section~\ref{sec:prelims}, we will define the $k$-typical graphs and generalised $k$-typical graphs in Section~\ref{sec:typical}.
In Section~\ref{sec:k-connSets} we will collect some basic facts about $k$-connected sets and their behaviour with respect to minors or topological minors.
Section~\ref{sec:end-structure} deals with the structure of ends in graphs.
Subsection~\ref{subsec:def-seq} is dedicated to extend a well-known connection between minimal separators and the degree of an end from locally finite graphs to arbitrary graphs. Afterwards, Subsection~\ref{subsec:T-conn-rays} gives a construction of how to find disjoint rays in some end with additional structure between them.
Sections~\ref{sec:regular-minors} and~\ref{sec:singular-minors} are dedicated to prove 
Theorem~\ref{main-thm-simple-2}. 
The case of~$\kappa$ being a regular cardinal is covered in Section~\ref{sec:regular-minors}, and the case of~$\kappa$ being a singular cardinals is covered in Section~\ref{sec:singular-minors}.
In Section~\ref{sec:applications} we will talk about some applications of the characterisation via minors, and in Section~\ref{sec:duality} we shall prove Theorem~\ref{main-thm-simple-1}.

\section{Preliminaries}
\label{sec:prelims}

%In this paper, we will work in ZFC.
For general notation about graph theory that we do not specifically introduce here we refer the reader to~\cite{Diestel:GT5}.

In this paper we consider both finite and infinite cardinals.
As usual, for an infinite cardinal~$\kappa$ we define its \emph{cofinality}, denoted by~$\cf\kappa$, 
as the smallest infinite cardinal~$\lambda$ such that there is a 
set~${X \subseteq \{ Y \subseteq \kappa\ |\ \abs{Y} < \kappa\}}$ 
such that~${\abs{X} = \lambda}$ and~${\bigcup X = \kappa}$.
We distinguish infinite cardinals~$\kappa$ to 
\emph{regular cardinals}, i.e.~cardinals where ${\cf\kappa = \kappa}$, and 
\emph{singular cardinals}, i.e.~cardinals where ${\cf\kappa < \kappa}$.
Note that~$\cf\kappa$ is always a regular cardinal.
For more information on infinite cardinals and ordinals, we refer the reader to~\cite{Kunen:set-theory1980}.

Throughout this paper, let $G$ denote an arbitrary simple and undirected graph with vertex set~$V(G)$ and edge set~$E(G)$. 
We call~$G$ \emph{locally finite} if each vertex of~$G$ has finite degree.

\vspace{0.2cm}

Let~$G$ and~$H$ be two graphs.
The \emph{union}~${G \cup H}$ of~$G$ and~$H$ is the graph with vertex set ${V(G) \cup V(H)}$ and edge set ${E(G) \cup E(H)}$.
The \emph{Cartesian product}~${G \times H}$ 
of~$G$ and~$H$ is the graph with vertex set ${V(G) \times V(H)}$ 
such that two vertices ${(g_1, h_1), (g_2, h_2) \in V(G\times H)}$ are adjacent 
if and only if either ${h_1 = h_2}$ and ${g_1g_2 \in E(G)}$ or ${g_1 = g_2}$ and ${h_1h_2 \in E(H)}$ holds.

Given two sets $A$ and $B$, we denote by $K(A,B)$ the complete bipartite graph between the classes $A$ and $B$.
We also write $K_{\kappa, \lambda}$ for $K(A,B)$ if $\abs{A} = \kappa$ and $\abs{B} = \lambda$ for two cardinals~$\kappa$ and~$\lambda$.

\vspace{0.2cm}

Unless otherwise specified, a path in this paper is a finite graph.
The \emph{length} of a path is the size of its edge set.
A path is \emph{trivial}, if it contains only one vertex, which we will call its \emph{endvertex}.
Otherwise, the two vertices of degree~$1$ in the path are its \emph{endvertices}.
The other vertices are called the \emph{inner vertices} of the path.

Let~${A, B \subseteq V(G)}$ be two (not necessarily disjoint) vertex sets.
An \emph{$A$\,--\,$B$ path} is a path whose inner vertices are disjoint from ${A \cup B}$ such that one of its end vertices lies in~$A$ and the other lies in~$B$.
In particular, a trivial path whose endvertex is in ${A \cap B}$ is also an $A$\,--\,$B$ path.
An \emph{$A$\,--\,$B$ separator} is a set~$S$ of vertices such that~${A \setminus S}$ and~${B \setminus S}$ lie in different components of~${G - S}$.
We also say $S$ \emph{separates}~$A$ and~$B$.
For convenience, by a slight abuse of notation, if~${A = \{a\}}$ (or~${B = \{b\}}$) is a singleton we will replace~$A$ by~$a$ (or~$B$ by~$b$ respectively) for these terms.

We shall need the following version of Menger's Theorem for finite parameter~$k$ in infinite graphs, which is an easy corollary of Menger's Theorem for finite graphs.

\begin{theorem}\label{finite-parameter-menger}
    \cite{Diestel:GT5}*{Thm.~8.4.1}
    Let $k \in \mathbb{N}$ and let $A$, $B \subseteq V(G)$.
    If~$A$ and~$B$ cannot be separated by less than~$k$~vertices, 
    then~$G$ contains~$k$ disjoint $A$\,--\,$B$ paths.
\end{theorem}    

We shall also need a trivial cardinality version of Menger's theorem, which is easily obtained from Theorem~\ref{finite-parameter-menger} by noting that the union of less than~$\kappa$ many disjoint $A$\,--\,$B$ paths for an infinite cardinal~$\kappa$ has size less than~$\kappa$ (cf.~\cite{Diestel:GT5}*{Section~8.4}).

\begin{theorem}\label{cardinality-menger}
    Let $\kappa$ be a cardinal and let $A$, $B \subseteq V(G)$.
    If~$A$ and~$B$ cannot be separated by less than~$\kappa$~vertices, 
    then~$G$ contains~$\kappa$ disjoint $A$\,--\,$B$ paths. \qed
\end{theorem}    

\vspace{0.2cm}

Recall that a one-way infinite path~$R$ is called a \emph{ray} and a two-way infinite path~$D$ is called a \emph{double ray}.
The unique vertex of degree~$1$ of~$R$ is its \emph{start vertex}.
A subgraph of~$R$ (or~$D$) that is a ray itself is called a \emph{tail} of~$R$ (or~$D$ respectively).
Given~${v \in R}$, we write~$vR$ for the tail of~$R$ with start vertex~$v$.
A finite path~${P \subseteq R}$ (or~${P \subseteq D}$) is a \emph{segment} of~$R$ (or~$D$ respectively).
If~$v$ and~$w$ are the endvertices of~$P$, then we denote~$P$ also by~${vRw}$ (or~${vDw}$ respectively).
If~$v$ is the end vertex of~$vRw$ whose distance is closer to the start vertex of~$R$, then $v$ is called the \emph{bottom vertex of~$vRw$} and $w$ is called the \emph{top vertex of~$vRw$}.
If additionally~$v$ is the start vertex of~$R$, then we call~$vRw$ an \emph{initial segment} of~$R$ and denote it by~${Rw}$. 
As a similar notation, given a tree~$T$ and nodes~$v$ and~$w$ of~$T$, we write~${vTw}$ for the unique $v$\,--\,$w$ path in~$T$. 

An \emph{end} of~$G$ is an equivalence class of rays, where two rays are \emph{equivalent} if they cannot be separated by deleting finitely many vertices of~$G$.
We denote the set of ends of~$G$ by~$\Omega(G)$.
A ray being an element of an end ${\omega \in \Omega(G)}$ is called an \emph{$\omega$-ray}.
A double ray all whose tails are elements of $\omega$ is called an \emph{$\omega$-double ray}.

For an end~${\omega \in \Omega(G)}$ let~$\deg(\omega)$ 
denote the \emph{degree} of~$\omega$, that is the supremum of the set 
${\{ \abs{\mathcal{R}}\ |\ \mathcal{R}\text{ is a set of disjoint } \omega\text{-rays} \}}$.
Note for each end $\omega$ there is in fact a set $\mathcal{R}$ of vertex disjoint $\omega$-rays with $\abs{\mathcal{R}} = \deg(\omega)$ \cite{Halin:ends}*{Satz~1}.

Recall that a vertex~${d \in V(G)}$ \emph{dominates} a ray~$R$ if~$d$ and some tail of~$R$ lie in the same component of~${G-S}$ for every finite set ${S \subseteq V(G) \setminus \{d\}}$.
By Theorem~\ref{cardinality-menger} this is equivalent to the existence of infinitely many $d$\,--\,$R$ paths in~$G$ which are disjoint but for~$d$ itself.
Note that if~$d$ dominates an~$\omega$-ray, then it also dominates every other $\omega$-ray.
Hence we also write that~$d$ \emph{dominates} an end~${\omega \in \Omega(G)}$ if $d$ dominates some $\omega$-ray.
Let~$\Dom(\omega)$ denote the set of vertices dominating $\omega$ and let~${\dom(\omega) = \abs{\Dom(\omega)}}$.
If ${\dom(\omega) > 0}$, we call $\omega$ \emph{dominated}, and if ${\dom(\omega) = 0}$, we call $\omega$ \emph{undominated}.

For an end~${\omega \in \Omega(G)}$, let~$\Delta(\omega)$ denote~${\deg(\omega) + \dom(\omega)}$, which we call the \emph{combined degree} of~$\omega$.
Note that the sum of an infinite cardinal with some other cardinal is just the maximum of the two cardinals.

The following lemma due to K\"onig about the existence of a ray is a weak version of the compactness principle in combinatorics.
\begin{lemma}[K\"onig's Infinity Lemma]\label{koenig}\cite{Diestel:GT5}*{Lemma~8.1.2}
Let $(V_i)_{i \in \mathbb{N}}$ be a sequence of disjoint non-empty finite sets, and let $G$ be a graph on their union.
Assume that for every $n > 0$ each vertex in $V_n$ has a neighbour in $V_{n-1}$.
Then $G$ contains a ray $v_0v_1 \ldots$ with $v_n \in V_n$ for all $n \in \mathbb{N}$.
\end{lemma}

In Section~\ref{sec:duality} we shall also use a stronger version of the compactness principle in combinatorics, the Generalised Infinity Lemma.
In order to state that lemma, we need the following definitions.

A partially ordered set $(P, \leq)$ is \emph{directed} if any two elements have a common upper bound, i.e.~for any $p,q \in P$ there is an $r \in P$ with $p \leq r$ and $q \leq r$.
A \emph{directed inverse system} consists of 
a directed poset $P$, 
a family of sets $( X_p \colon p \in P )$, and
for all $p, q \in P$ with $p < q$ a map $f_{q,p}: X_q \to X_p$
such that the maps are \emph{compatible}, i.e.~$f_{q, p} \circ f_{r, q} = f_{r, p}$ for all $p, q, r \in P$ with $p < q < r$.
The \emph{inverse limit} of such a directed inverse system is the set 
\[
    \lim \limits_{\longleftarrow}\, (X_p \colon p \in P) = \left\{ ( x_p \colon p \in P) \in \prod \limits_{p \in P} X_p \colon f_{q,p}(x_q) = x_p \right\}.
\]
\begin{lemma}[Generalised Infinity Lemma]
    \label{gen-inf-lemma}\cite{Diestel:GT5}*{Appendix~A}
    The inverse limit of any directed inverse system of non-empty finite sets is non-empty.
\end{lemma} 

A \emph{comb}~$C$ is the union of a ray $R$ together with infinitely many disjoint finite paths each of which has precisely one vertex in common with $R$, which has to be an endvertex of that path.
The ray $R$ is the \emph{spine} of $C$ and the end vertices of the finite paths that are not on $R$ together with the end vertices of the trivial paths are the \emph{teeth} of~$C$.
A comb whose spine is in $\omega$ is also called an $\omega$-comb.
A \emph{star} is the complete bipartite graph $K_{1, \kappa}$ for some cardinal $\kappa$, where the vertices of degree~$1$ are its \emph{leaves} and the vertex of degree $\kappa$ is its \emph{centre}.

Next we state a version of the Star-Comb lemma in a slightly stronger way than elsewhere in the literature (e.g.~\cite{Diestel:GT5}*{Lemma 8.2.2}).
We also give a proof for the sake of completeness.
\begin{lemma}[Star-Comb Lemma]\label{star-comb}
    Let ${U \subseteq V(G)}$ be infinite 
    and let $\kappa \leq \abs{U}$ be a regular cardinal.
    Then the following statements are equivalent.
    \begin{enumerate}[label=(\alph*)]
        \item \label{item:star-comb-1} 
            There is a subset ${U_1 \subseteq U}$ with ${\abs{U_1} = \kappa}$ such that~$U_1$ is $1$-connected in~$G$.
        \item \label{item:star-comb-2}
            There is a subset ${U_2 \subseteq U}$ with ${\abs{U_2} = \kappa}$ such that~$G$ either contains a subdivided star whose set of leaves is~$U_2$ or a comb whose set of teeth is~$U_2$.
        
        (Note that if~$\kappa$ is uncountable, only the former can exist.)
    \end{enumerate}
    Moreover, if these statements hold, we can choose~${U_1 = U_2}$.
\end{lemma}

\begin{proof}
    Note that a set of vertices is $1$-connected, if and only if it belongs to the same component of~$G$.
    Hence if~\ref{item:star-comb-2} holds, then~$U_2$ is $1$-connected and we can set ${U_1 := U_2}$ to satisfy~\ref{item:star-comb-1}.
    
    If~\ref{item:star-comb-1} holds, then we take a tree ${T \subseteq G}$ containing~$U_1$ such that each edge of~$T$ lies on a path between two vertices of~$U_1$. 
    Such a tree exists by Zorn's Lemma since~$U_1$ is $1$-connected in~$G$.
    We distinguish two cases.
    
    If~$T$ has a vertex~$c$ of degree~$\kappa$, then this yields a subdivided star with centre~$c$ and a set~${U_2 \subseteq U_1}$ of leaves with~${\abs{U_2} = \kappa}$ by extending each incident edge of~$c$ to a $c$\,--\,$U_1$ path.
    
    Hence we assume~$T$ does not contain a vertex of degree~$\kappa$.
    Given some vertex~${v_0 \in V}$ and~${n \in \mathbb{N}}$, let~$D_n$ denote the vertices of~$T$ of distance~$n$ to~$v_0$.
    Since~$T$ is connected, the union $\bigcup \{ D_n\ |\ n \in \mathbb{N} \}$ equals~$V(T)$.
    And because~$\kappa$ is regular, it follows that~${\kappa = \aleph_0}$, and therefore that~$T$ is locally finite.
    Hence each~$D_n$ is finite and, since~$T$ is still infinite, each~$D_n$ is non-empty.
    Thus~$T$ contains a ray~$R$ by Lemma~\ref{koenig}.
    If~$R$ does not already contain infinitely many vertices of~$U_1$, then by the property of~$T$ there are infinitely many edges of~$T$ between~$V(R)$ and~${V(T - R)}$.
    We can extend infinitely many of these edges to a set of disjoint $R$\,--\,$U_1$ paths, ending in an infinite subset~${U_2 \subseteq U_1}$, yielding the desired comb.
    
    In both cases, $U_2$ is still $1$-connected, and hence serves as a candidate for $U_1$ as well, yielding the ``moreover'' part of the claim.
\end{proof}

The following immediate remark helps to identify when we can obtain stars by an application of the Star-Comb lemma.

\begin{remark}\label{rem:star-from-comb}
    If there is an~$\omega$-comb with teeth~$U$ and if~$v$ dominates~$\omega$, then there is also a set ${U' \subseteq U}$ with~${\abs{U'} = \abs{U} = \aleph_0}$ such that~$G$ contains a subdivided star with leaves~$U'$ and centre~$v$. \qed
\end{remark}

We say that an end $\omega$ is in the \emph{closure} of a set $U \subseteq V(G)$, if there is an $\omega$-comb whose teeth are in $U$.
Note that this combinatorial definition of closure coincides with the topological closure when considering the topological setting of locally finite graphs mentioned in the introduction \citelist{\cite{Diestel:GT5}*{Section~8.6}\cite{Diestel:topGTsurvey}}. 

\vspace{0.2cm}

For an end~$\omega$ of~$G$ and an induced subgraph~$G'$ of~$G$ we write~${\omega{\upharpoonright}G'}$ for the set of rays~${R \in \omega}$ which are also rays of~$G'$.
The following remarks are immediate.

\begin{remarks}\label{rem:end_subgraph}
    Let~${G' = G - S}$ for some finite~${S \subseteq V(G)}$.
    \begin{enumerate}
        \item $\omega{\upharpoonright}G'$ is an end of~$G'$ for every end~${\omega \in \Omega(G)}$.
        \item For every end $\omega' \in \Omega(G')$ there is an end $\omega \in \Omega(G)$ such that $\omega{\upharpoonright}G' = \omega'$.
        \item The degree of ${\omega \in \omega(G)}$ in $G$ is equal to the degree of $\omega{\upharpoonright}G'$ in $G'$.
        \item $\Dom(\omega) = \Dom(\omega{\upharpoonright}G') \cup (\Dom(\omega) \cap S)$ for every end~${\omega \in \Omega(G)}$.
        \qed
    \end{enumerate}
\end{remarks}

Given an end~${\omega \in \Omega(G)}$, we say that an $\omega$-ray~$R$ is \emph{$\omega$-devouring} if no $\omega$-ray is disjoint from~$R$.
We need the following lemma about the existence of a single $\omega$-devouring ray for an end $\omega$ of at most countable degree, which is a special case of~\cite{GH:end-devouring}*{Thm.~1}.

\begin{lemma}
    \label{lem:end_devouring}
    If $\deg(\omega) \leq \aleph_0$ for $\omega \in \Omega(G)$, then $G$ contains an $\omega$-devouring ray.
\end{lemma}

A different way to prove this lemma arises from the construction of normal spanning trees, cf.~\cite{Diestel:GT5}*{Prop.~8.2.4}.
Imitating this proof according to an enumeration of the vertices of a maximal set of disjoint $\omega$-rays yields that the normal ray constructed this way is $\omega$-devouring.

\vspace{0.2cm}

Let us fix some notations regarding minors. 
Let~$G$ and~$M$ be graphs. 
We say~$M$ is a \emph{minor} of~$G$ if~$G$ contains an \emph{inflated subgraph~${H \subseteq G}$} witnessing this, i.e.~for each $v \in V(M)$ 
\begin{itemize}
    \item there is a non-empty \emph{branch set~${\mathfrak{B}(v) \subseteq V(H)}$};
    \item ${H[\mathfrak{B}(v)]}$ is connected;
    \item ${\{ \mathfrak{B}(v)\ |\ v \in V(M) \}}$ is a partition of~${V(H)}$; and
    \item there is an edge between~${v, w \in V(M)}$ in~$M$ if and only if there is an edge between some vertex in~${\mathfrak{B}(v)}$ and a vertex in~${\mathfrak{B}(w)}$ in~$H$.
\end{itemize}
We call~$M$ a \emph{finite-branch-set minor} or \emph{fbs-minor} of~$G$ if each branch set is finite.

Without loss of generality we may assume that such an inflated subgraph~$H$ witnessing that~$M$ is a minor of~$G$ is minimal with respect to the subgraph relation.
Then~$H$ has the following properties for all ${v, w \in V(M)}$:
\begin{itemize}
    \item $H[\mathfrak{B}(v)]$ is a finite tree~$T_v$;
    \item for each ${v, w \in V(M)}$ there is a unique edge~$e_{vw}$ in~$E(H)$ between~${\mathfrak{B}(v)}$ and~${\mathfrak{B}(w)}$ if ${vw \in E(M)}$, and no such edge if ${vw \notin E(M)}$;
    \item each leaf of $T_v$ is an endvertex of such an edge between two branch sets.
\end{itemize}

Given a subset ${C \subseteq V(M)}$ and a subset ${A \subseteq V(G)}$, we say that \emph{$M$ is an fbs-minor of~$G$ with~$A$ along~$C$}, 
if~$M$ is an fbs-minor of~$G$ 
such that the map mapping each vertex of the inflated subgraph to the branch set it is contained in induces a bijection between~$A$ and the branch sets of~$C$.
As before, we assume without loss of generality that an inflated subgraph~$H$ witnessing that~$M$ is an fbs-minor of~$G$ is minimal with respect to the subgraph relation.
We obtain the properties as above, but a leaf of $T_v$ could be the unique vertex of~$A$ in~${\mathfrak{B}(v)}$ instead.

\vspace{0.2cm}

For~${\ell, k \in \mathbb{N}}$, we write~${[\ell, k]}$ for the \emph{closed integer interval} ${\{ i \in \mathbb{N}\ |\ \ell \leq i \leq k \}}$ 
as well as~${[k, \ell)}$ for the \emph{half open integer interval}~${\{ i \in \mathbb{N}\ |\ \ell \leq i < k \}}$.

Given some set~$I$, a \emph{family}~$\mathcal{F}$ indexed by~$I$ is a sequence of the form~${( F_i\ |\ i \in I)}$, where the \emph{members}~$F_i$ are some not necessarily different sets.
For convenience we sometimes use a family and the set of its members with a slight abuse of notation interchangeably, for example with common set operations like~${\bigcup \mathcal{F}}$. 
Given some~${J \subseteq I}$, we denote by~${\mathcal{F}\restricted{J}}$ the \emph{subfamily}~${( F_j\ |\ j \in J )}$.
A set $T$ is a \emph{transversal} of~$\mathcal{F}$, if ${\abs{T \cap F_i} = 1}$ for all ${i \in I}$.
For a family~${( F_i\ |\ i \in \mathbb{N})}$ with index set $\mathbb{N}$ we say some property holds for \emph{eventually all} members, if there is some $N \in \mathbb{N}$ such that the property holds for~$F_i$ for all $i \in \mathbb{N}$ with $i \geq N$.

The following lemma is a special case of the famous Delta-Systems Lemma, a common tool of infinite combinatorics.
\begin{lemma}\label{delta-system-lemma}
    \cite{Kunen:set-theory1980}*{Thm.~II.1.6}
    Let~$\kappa$ be a regular cardinal,~$U$ be a set and \linebreak 
    ${\mathcal{F} = (F_\alpha \subseteq U\ |\ \alpha \in \kappa)}$ a family of finite subsets of~$U$.
    Then there is a finite set ${D \subseteq U}$ and a set~${I \subseteq \kappa}$ with~${\abs{I} = \kappa}$ such that~${F_\alpha \cap F_\beta = D}$ for all~${\alpha, \beta \in I}$ with ${\alpha \neq \beta}$.
\end{lemma}

\vspace{0.2cm}

A pair~${(T,\mathcal{V})}$ is a \emph{tree-decomposition of~$G$} if~$T$ is a (possibly infinite) tree and~$\mathcal{V}$ is a family~${( V_t\ |\ t \in V(T))}$ of sets of vertices of~$G$ such that
\begin{enumerate}
    \item[(T1)] ${V(G) = \bigcup \mathcal{V}}$; 
    \item[(T2)] for every edge~${e \in E(G)}$ there is a~${t \in V(T)}$ such that~${e \in E(G[V_t])}$; and
    \item[(T3)] ${V_{t_1} \cap V_{t_3} \subseteq V_{t_2}}$ whenever $t_1, t_2, t_3 \in V(T)$ satisfy~$t_2 \in V(t_1 T t_3)$.
\end{enumerate}

We call the sets in~$\mathcal{V}$ the \emph{parts} of~$(T,\mathcal{V})$. 
We call the sets~$V_{t_1} \cap V_{t_2}$ for each edge~${t_1 t_2 \in E(T)}$ the \emph{adhesion sets} of~$(T,\mathcal{V})$. 
If every adhesion set has cardinality less than~$\kappa$, then we say~$(T,\mathcal{V})$ has \emph{adhesion} less than~$\kappa$. 

In Section~\ref{sec:duality} we will make use of the existence of $k$-lean tree-decompositions for finite graphs to prove our desired duality theorem. 
Given~${k \in \mathbb{N}}$, a tree-decomposition of adhesion less than~$k$ is called \emph{$k$-lean} 
if for any two (not necessarily distinct) parts~$V_{t_1}$, $V_{t_2}$ of the tree-decomposition and vertex sets ${Z_1 \subseteq V_{t_1}}$, ${Z_2 \subseteq V_{t_2}}$ with ${\abs{Z_1} = \abs{Z_2} = \ell \leq k}$ 
there are either $\ell$ disjoint $Z_1$\,--\,$Z_2$ paths in~$G$ or there is an edge~${t t'}$ on the $t_1$\,--\,$t_2$ path in the decomposition tree inducing a separation of order less than $\ell$.
In particular, given a $k$-lean tree-decomposition, each part~$V_t$ is~${\min\{k, \abs{V_t}\}}$-connected in~$G$.

In \cite{CDHH:k-blocks}, the authors noted that the proof given in~\cite{DB:lean-td} of a theorem of Thomas \cite[]{Thomas:lean-finite}*{Thm.~5} about the existence of lean tree-decompositions witnessing the tree-width of a finite graph can be adapted to prove the existence of a $k$-lean tree-decomposition of that graph.

\begin{theorem}\label{thm:fin-k-lean}
    \cite{CDHH:k-blocks}*{Thm.~2.3}
    Every finite graph has a $k$-lean tree-decomposition for any~${k \in \mathbb{N}}$.
\end{theorem}

\vspace{0.2cm}

A \emph{separation} of~$G$ is a tuple~${(A,B)}$ of vertex sets such that ${A \cup B = V(G)}$ and such that there is no edge of~$G$ between~${A \setminus B}$ and~${B \setminus A}$.
The set~${A \cap B}$ is the \emph{separator of $(A,B)$} and the cardinality ${\abs{A \cap B}}$ is called the \emph{order} of~$(A,B)$.
Given~${k \in \mathbb{N}}$, let~${S_k(G)}$ denote the set of all separations of $G$ of order less than~$k$.

For two separations~${(A,B)}$ and~${(C,D)}$ of~$G$ we define a partial order~$\leq$ between them as~${(A,B) \leq (C,D)}$ if~${A \subseteq C}$ and~${D \subseteq B}$. 

Two separations~${(A,B)}$ and~${(C,D)}$ are \emph{nested} if 
they are comparable with respect to this partial order. 
A set~$N$ of separations of~$G$ is called a \emph{nested separation system of~$G$} if it is \emph{symmetric}, i.e.~${(B,A) \in N}$ for each ${(A,B) \in N}$ and \emph{nested}, i.e.~the separations in~$N$ are pairwise nested. 

A \emph{chain} of separations is a set of separations which is linearly ordered by~$\leq$. 
In this paper, we will consider chains of ordinal \emph{order type}, i.e.~order isomorphic to an ordinal. 

An \emph{orientation}~$O$ of a nested separation system~$N$ is a subset of~$N$ that contains precisely one of~${(A,B)}$ and~${(B,A)}$ for all~${(A,B) \in N}$.
An orientation~$O$ of~$N$ is \emph{consistent} if whenever~${(A,B) \in O}$ and~${(C,D) \in N}$ 
with~${(C,D) \leq (A,B)}$, then~${(C,D) \in O}$.
For each consistent orientation~$O$ of~$N$ we define a \emph{part~$V_O$} of~$N$ as the vertex set ${\bigcap \{ B\ |\ (A,B) \in O \}}$. 
We say a consistent orientation~$O_2$ is \emph{between} two consistent orientations~$O_1$ and~$O_3$ if~${O_1 \cap O_3 \subseteq O_2}$. 
In some sense, the separations in the symmetric difference of~$O_1$ and~$O_3$ represent `edges' of the unique `path' between~$O_1$ and~$O_3$ and the consistent orientations between them correspond to the `nodes' of that `path', see~\cites{Diestel:tree-sets,GK:infinite-tree-sets} for more details. 

Now it is easy to verify analogues of the axioms for tree-decompositions: 
the union of all parts cover the vertex set of~$G$, 
every edge of~$G$ is contained in~${G[V_O]}$ for some consistent orientation~$O$ of~$N$, 
and the intersection of two parts~$V(O_1)$ and~$V(O_3)$ is contained in any part corresponding to a consistent orientation~$O_2$ between~$O_1$ and~$O_3$. 

Note that we allow the empty set~$\emptyset$ as a nested separation system.
In this case, we say that~${V(G)}$ is a part of~$\emptyset$ (this can be viewed as the empty intersection of vertex sets of the empty set as an orientation of $\emptyset$).

A nested separation system~$N$ has \emph{adhesion less than~$k$} if all separations it contains have order less than~$k$, i.e.~${N \subseteq S_k(G)}$.

\vspace{0.2cm}

Each oriented edge of the tree of a tree-decomposition of~$G$ induces a separation ${(A,B)}$ where~$B$ is the union of the parts corresponding to the side where the edge is oriented towards, while~$A$ is the union of the parts on the other side.
It is easy to check that the set of separations induced by all those edges is a nested separation system.
Moreover, properties like adhesion and the size of parts are transferred by this process.

On the other hand, given a nested separation system~$N$ that contains no chains of order type~${\omega+1}$ and no separations of the form~${(A,V(G))}$ then it is possible to construct a tree that `captures' the tree-like structure of~$N$: 
the nodes of the tree will be the consistent orientations of~$N$ that contain a maximal element, and there is an edge between nodes precisely if the orientations orient a unique element of~$N$ oppositely. 
For more details on this construction, see~\cite{GK:infinite-tree-sets}*{Section~3}. 
Now it is easy to see that given two nodes of the tree, the set of parts between the parts corresponding to the two nodes defines us the unique path in the tree between those two nodes. 
Together with the observations above we conclude that this defines us a tree-decomposition of~$G$.

\begin{theorem}
    \cite{GK:infinite-tree-sets}*{Section~3}
    \label{thm:ts2td}
    A nested separation system of~$G$ which does not contain any separations of the form~${(A,V(G))}$ and which does not contain any chains of order type~${\omega+1}$ defines a tree-decomposition with the same parts and the same adhesion. 
\end{theorem}

This relationship between tree-decompositions and nested separation systems is well-known for finite graphs. 
For more information on nested separation systems and their connection to tree-decompositions we refer the interested reader to \cites{Diestel:tree-sets, GK:infinite-tree-sets}. 

The definition of $k$-lean tree-decomposition can easily be lifted to nested separation systems.
A nested separation system ${N \subseteq S_k(G)}$ is called \emph{$k$-lean} if given any two (not necessarily distinct) parts~$V_{O_1}$,~$V_{O_2}$ of~$N$ and vertex sets ${Z_1 \subseteq V_{O_1}}$, ${Z_2 \subseteq V_{O_2}}$ with ${\abs{Z_1} = \abs{Z_2} = \ell \leq k}$ 
there are either $\ell$ disjoint $Z_1$\,--\,$Z_2$ paths in~$G$ or there is a separation~$(A,B)$ in~$N$ with ${V_{O_1} \subseteq A}$ and ${V_{O_2} \subseteq B}$ of order less than~$\ell$.
Here, we specifically allow the empty set as a nested separation system to be $k$-lean if its part, the whole vertex set of~$G$, is~${\min\{k, \abs{V(G)}\}}$-connected.
Again, we obtain that each part~$V_O$ of a $k$-lean nested separation system is~${\min\{k, \abs{P}\}}$-connected in~$G$.
Moreover, note that the nested separation system that a $k$-lean tree-decomposition induces is $k$-lean as well. 

In Section~\ref{sec:duality} we will lift Theorem~\ref{thm:fin-k-lean} to nested separation systems. 
This then allows us to construct a $k$-lean tree-decomposition for any graph and any~${k \in \mathbb{N}}$.

\section{Typical graphs with \emph{k}-connected sets}
\label{sec:typical}

Throughout this section, let $k \in \mathbb{N}$ be fixed.
Let $\kappa$ denote an infinite cardinal.

In Subsection~\ref{subsec:k-typ} we shall describe 
an up to isomorphism finite class of graphs each of which contains a designated $k$-connected set of size~$\kappa$.
We call such a graph a \emph{$k$-typical graph} and the designated $k$-connected set its \emph{core}.
These graphs will appear as the minors of Theorem~\ref{main-thm-simple-2}\ref{item:t2-minor}.

In Subsection~\ref{subsec:gen-k-typ} we shall describe 
based on these $k$-typical graphs a more general but still finite class of graphs each of which again contains a designated $k$-connected set of size~$\kappa$.
We call such a graph a \emph{generalised $k$-typical graph} and the designated $k$-connected set its \emph{core}.
These graphs will appear as the topological minors of Theorem~\ref{main-thm-simple-2}\ref{item:t2-subdivision}.

\subsection{\emph{k}-typical graphs}
\label{subsec:k-typ}
\ \newline \indent
The most basic graph with a $k$-connected set of size~$\kappa$ is a complete bipartite graph ${K_{k, \kappa} = K([0,k), Z)}$ for any infinite cardinal~$\kappa$ and a set~$Z$ of size~$\kappa$ disjoint from~$[0,k)$.
Although in this graph the whole vertex set is $k$-connected, we only want to consider the infinite side~$Z$ as the \emph{core}~$C(K_{k,\kappa})$ of $K_{k, \kappa}$, cf.~Figure~\ref{fig:k4kappa}.
This is the first instance of a \emph{$k$-typical graph} with a core of size~$\kappa$.
For uncountable regular cardinals~$\kappa$, this is the only possibility for a $k$-typical graph with a core of size~$\kappa$.

\begin{figure}[htbp]
    \includegraphics[scale=1.5]{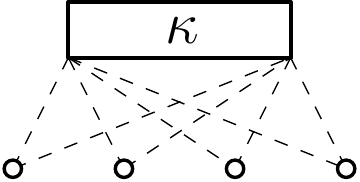}
    \caption{A stylised version of a $K_{4, \kappa}$, where the large box stands for the core of $\kappa$ many vertices and the dashed lines from a vertex to the corners of the box represent that this vertex is connected to all vertices in the box.}
    \label{fig:k4kappa}
\end{figure}

\vspace{0.2cm}

A \emph{$k$-blueprint}~$\mathcal{B}$ is a tuple~$(B,D)$ 
such that
\begin{itemize}
    \item $B$ is a tree of order $k$; and 
    \item $D$ is a set of leaves of $B$ with $\abs{D} < \abs{V(B)}$.
\end{itemize}

\noindent
Take the ray ${\mathfrak{N} := \big( \mathbb{N}, \{n (n+1)\ |\ n \in \mathbb{N} \} \big)}$ and the Cartesian product $B \times \mathfrak{N}$.
For a node $b \in V(B)$ and $n \in \mathbb{N}$ let 
\begin{itemize}
    \item $b_n$ denote the vertex $(b,n)$;
    \item $\mathfrak{N}_b$ denote the ray $(\{b\}, \emptyset) \times \mathfrak{N} \subseteq B \times \mathfrak{N}$; and
    \item $B_n$ denote the subgraph $B \times (\{n\}, \emptyset) \subseteq B \times \mathfrak{N}$.
\end{itemize}
Then let~${\mathfrak{N}(B/D) := (B \times \mathfrak{N}) / \{ \mathfrak{N}_d\ |\ d \in D \}}$ denote the contraction minor of~${B \times \mathfrak{N}}$ obtained by contracting each ray~$\mathfrak{N}_d$ for each~${d \in D}$ to a single vertex.
We denote the vertex of~$\mathfrak{N}(B/D)$ corresponding to the contracted ray $\mathfrak{N}_d$ by~$d$ for~${d \in D}$ and call such a vertex \emph{dominating}.
Using this abbreviated notation, we call the tree~${B_n - D}$ the \emph{$n$-th layer of~$\mathfrak{N}(B/D)$}.

A triple $\mathcal{B} = (B,D,c)$ is called a \emph{regular $k$-blueprint} if $(B,D)$ is a $k$-blueprint and ${c \in V(B) \setminus D}$.
We denote by~${T_k(\mathcal{B})}$ the graph~${\mathfrak{N}(B/D)}$ and by~${C(T_k(\mathcal{B}))}$ the vertex set~${V(\mathfrak{N}_c)}$, which we call the 
\emph{core} of~$T_k(\mathcal{B})$, 
see Figure~\ref{fig:regBP-small} for an example.

\begin{figure}[htbp]
    \includegraphics[scale=1]{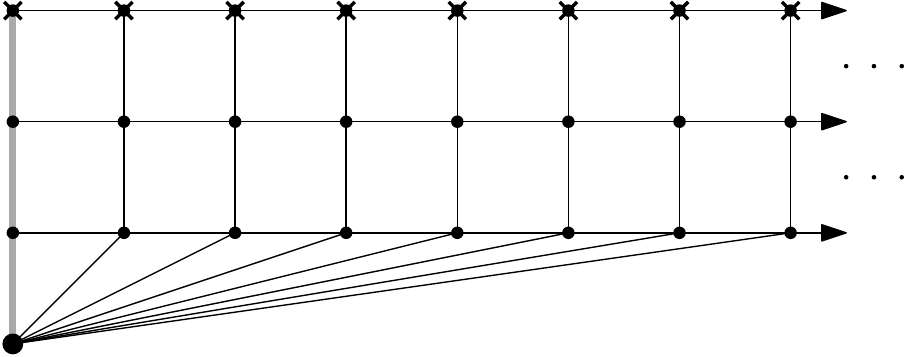}
    \caption{Image of $T_4(P,\{d\},c)$ where $P = c a b d$ is a path of length $3$ between nodes $c$ and $d$.
    $P_0$ is represented in gray.
    The crosses represent its core.}
    \label{fig:regBP-small}
\end{figure}

\begin{lemma}\label{lem:typ-countable}
    For a regular $k$-blueprint $\mathcal{B}$ the core of~$T_k(\mathcal{B})$ is $k$-connected in~$T_k(\mathcal{B})$.
\end{lemma}

\begin{proof}
    Let $\mathcal{B} = (B, D, c)$ and let~$C = C(T_k(\mathcal{B}))$ denote the core of~$T_k(\mathcal{B})$. 
    Let~${U, W \subseteq C}$ with~${\abs{U} = \abs{W} = k' \leq k}$.
    Suppose for a contradiction there is a vertex set~$S$ of size less than~$k'$ separating~$U$ and~$W$.
    Then there are~${m, n \in \mathbb{N}}$ with~${c_m \in U \setminus S}$, ${c_n \in W \setminus S}$ such that the $n$-th and $m$-th layer are both disjoint from~$S$. 
    Moreover there is a~${b \in {B}}$ such that~$\mathfrak{N}_b$ (or~$\{b\}$ if~${b \in D}$) are disjoint from~$U$ and~$W$.
    Hence we can connect~$c_m$ and~$c_n$ with the path consisting of the concatenation of the unique $c_m$\,--$b_m$ path in~$B_m$, the unique $b_m$\,--\,$b_n$ path in~$\mathfrak{N}_b$ and the unique $b_n$\,--\,$c_n$ path in~$B_n$.
    This path avoids~$S$, contradicting that~$S$ is a separator.
    By Theorem~\ref{finite-parameter-menger} there are~$k'$ disjoint $U$\,--\,$W$ paths, and hence~$C$ is $k$-connected in~$T_k(\mathcal{B})$.
\end{proof}

For any regular $k$-blueprint~${\mathcal{B} = (B,D,c)}$ the graph~$T_k(\mathcal{B})$ is a \emph{$k$-typical graph} with a countable core.
Such graphs are besides the complete bipartite graph~${K_{k, \aleph_0}}$ the only other $k$-typical graphs with a core of size~$\aleph_0$.

Note that given two regular $k$-blueprints~${\mathcal{B}_1 = (B_1, D_1, c_1})$ and~${\mathcal{B}_2 = (B_2, D_2, c_2)}$ such that there is an isomorphism~$\varphi$ between~$B_1$ and~$B_2$ that maps~$D_1$ to~$D_2$, then~${T_k(\mathcal{B}_1)}$ and~${T_k(\mathcal{B}_2)}$ are isomorphic.
Moreover, if~$\varphi$ maps~$c_1$ to~$c_2$, then there is an isomorphism between~${T_k(\mathcal{B}_1)}$ and~$T_k(\mathcal{B}_2)$ that maps the core of~${T_k(\mathcal{B}_1)}$ to the core of~${T_k(\mathcal{B}_2)}$.
Hence up to isomorphism there are only finitely many $k$-typical graphs with a core of size~$\aleph_0$.

\vspace{0.2cm}

Given a singular cardinal~$\kappa$ we have more possibilities for typical graphs with $k$-connected sets of size $\kappa$.
We call a sequence ${\mathcal{K} = ( \kappa_\alpha < \kappa\ |\ \alpha \in \cf\kappa )}$ of infinite cardinals a \emph{good $\kappa$-sequence}, if 
\begin{itemize}
    \item it is cofinal, i.e.\ $\bigcup \mathcal{K} = \kappa$;
    \item it is strictly ascending, i.e.\ $\kappa_\alpha < \kappa_\beta$ for all $\alpha < \beta$ with $\alpha, \beta \in \cf\kappa$;
    \item $\cf\kappa < \kappa_\alpha < \kappa$ for all $\alpha \in \cf\kappa$; and
    \item $\kappa_\alpha$ is regular for all $\alpha \in \cf\kappa$.
\end{itemize}
Note that given any~${I \subseteq \cf\kappa}$ with~${\abs{I} = \cf\kappa}$ there is a unique order-preserving bijection between~$\cf\kappa$ and~$I$.
Hence we can relabel any cofinal subsequence~${\mathcal{K}\restricted I}$ of a good $\kappa$-sequence~$\mathcal{K}$ to a good $\kappa$-sequence ${\overline{\mathcal{K}\restricted I}}$.
Moreover, note that any cofinal sequence can be made into a good $\kappa$-sequence by looking at a strictly ascending subsequence starting above the cofinality of~$\kappa$, then replacing each element in the sequence by its successor cardinal and relabeling as above.
Here we use the fact that each successor cardinal is regular.
Hence for every singular cardinal~$\kappa$ there is a good $\kappa$-sequence.

\vspace{0.2cm}

Let ${\mathcal{K} = ( \kappa_\alpha < \kappa\ |\ \alpha \in \cf\kappa )}$ be a good $\kappa$-sequence and 
let~${\ell \leq k}$ be a non-negative integer.
As a generalisation of the graph~$K_{k,\kappa}$ we first consider the disjoint union of the complete bipartite graphs~$K_{k,\kappa_\alpha}$.
Then we identify $\ell$ sets of vertices each consisting of a vertex of the finite side of each graph, and connect the other~${k - \ell}$ vertices of each with disjoint stars~$K_{1, \cf\kappa}$.
More formally, 
let~${X = [\ell,k) \times \{0\}}$,  
and for each~${\alpha \in \cf\kappa}$ let~${Y^\alpha = \{\alpha\} \times [0,k) \times \{1\}}$ 
and let~${Z^\alpha = \{\alpha\} \times \kappa_\alpha \times \{2\}}$. 
We denote the family~${( Y^\alpha\ |\ \alpha \in \cf\kappa )}$ with~$\mathcal{Y}$ and the family~${( Z^\alpha\ |\ \alpha \in \cf\kappa )}$ with~$\mathcal{Z}$.
Then consider the union~${\bigcup \{ K(Y^\alpha, Z^\alpha)\ |\ \alpha \in \cf\kappa \}}$
of the complete bipartite graphs 
and let~${\lK{\mathcal{K}}}$ denote the graph where for each ${i \in [0,\ell)}$ we identify the set ${\cf\kappa \times \{i\} \times \{1\}}$ to one vertex in that union.
For this graph we fix some further notation. Let 
\begin{itemize}
    \item $x_i$ denote $(i, 0) \in X$ for ${i \in [\ell,k)}$;
    \item $y_i = y_i^\alpha$ for all~${\alpha \in \cf\kappa}$ denote the vertex corresponding to ${\cf\kappa \times \{i\} \times \{1\}}$ for $i \in [0,\ell)$; we call such a vertex a \emph{degenerate vertex of~$\lK{\mathcal{K}}$};
    \item $y_i^\alpha$ denote $(\alpha, i, 1)$ for~${i \in [\ell,k)}$; and 
    \item $\mathcal{Y}_i$ denote $( y_i^\alpha\ |\ \alpha \in I)$ for~${i \in [\ell,k)}$.
\end{itemize}

Note that while the definition of~$\lK{\mathcal{K}}$ formally depends on the choice of a good $\kappa$-sequence, the structure of the graph is independent of that choice.

\begin{remark}\label{rem:cof1}
    ${\lK{\mathcal{K}_0}}$ is isomorphic to a subgraph of~${\lK{\mathcal{K}_1}}$, and vice versa, for any two good $\kappa$-sequences $\mathcal{K}_0$, $\mathcal{K}_1$.
    \qed
\end{remark}

Given $\lK{\mathcal{K}}$ as above, 
let $S_i$ denote the star~${K(\{ x_i \}, \bigcup \mathcal{Y}_i )}$ for all~$i \in [\ell, k)$.
Consider the union of~${\lK{\mathcal{K}}}$ with~${\bigcup_{i \in [\ell, k)} S_i}$.
We call this graph~$\lFK(\mathcal{K})$, or an \emph{$\ell$-degenerate frayed $K_{k,\kappa}$} (with respect to~$\mathcal{K}$).
As before, any vertex~$y_i$ for~${i \in [0, \ell)}$ is called a \emph{degenerate vertex of~$\lFK(\mathcal{K})$}, and any~$x_i$ for~${i \in [\ell, k)}$ is called a \emph{frayed centre of~$\lFK(\mathcal{K})$}.
The \emph{core}~${C(\lFK(\mathcal{K}))}$ of~${\lFK(\mathcal{K})}$ is the vertex set~${\bigcup \mathcal{Z}}$. 
As with $K_{k,\kappa}$ it is easy to see that~$C(\lFK(\mathcal{K}))$ is $k$-connected in~$\lFK(\mathcal{K})$ and of size~$\kappa$.

Note that Remark~\ref{rem:cof1} naturally extends to~$\lFK(\mathcal{K})$.
Hence for each $\kappa$ we now fix a specific good $\kappa$-sequence and write just $\lFK$ when talking about an $\ell$-degenerate frayed $K_{k,\kappa}$ regarding that sequence, 
see Figure~\ref{fig:2FK4kappa} for an example.
Further note that $k$--$FK_{k,\kappa}$ is isomorphic to~$K_{k,\kappa}$.
We also call a $0$-degenerate frayed~$K_{k,\kappa}$ just a \emph{frayed~$K_{k,\kappa}$} or~$FK_{k,\kappa}$ for short.

For a singular cardinal $\kappa$ and for any~${\ell \in [0,k]}$ the graph~$\lFK$ is a \emph{$k$-typical graph} with a core of size~$\kappa$.
These are besides the complete bipartite graph~$K_{k, \kappa}$ the only other $k$-typical graphs with a core of size~$\kappa$ if $\kappa$ has uncountable cofinality. 

\begin{figure}[htbp]
    \includegraphics[scale=1]{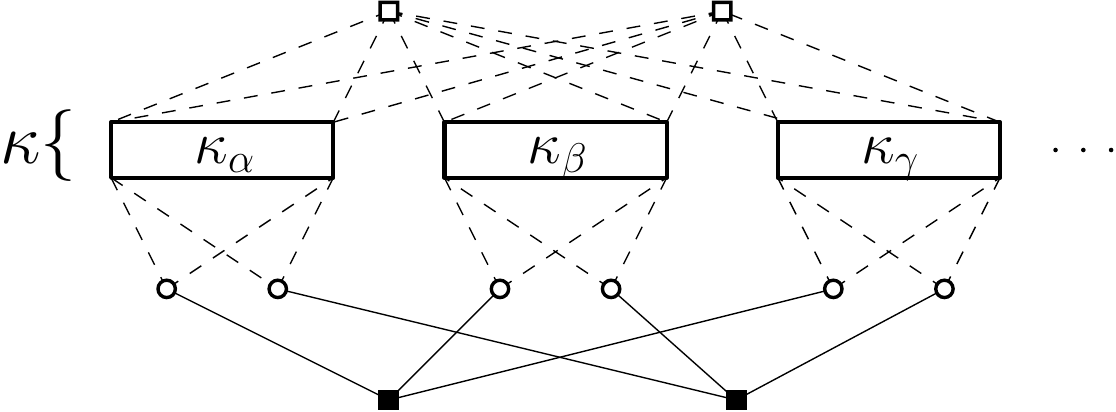}
    \caption{Image of $2$\,--\,$FK_{4,\kappa}$.
    The black squares represent the frayed centres and the white squares the degenerate vertices.
    Its core is represented by the union of the boxes (labelled according to the fixed good $\kappa$-sequence) and has size $\kappa$ as illustrated by the bracket.}
    \label{fig:2FK4kappa}
\end{figure}

\vspace{0.2cm}

Next we will describe the other possibilities of $k$-typical graphs for singular cardinals with countable cofinality.

A \emph{singular $k$-blueprint} $\mathcal{B}$ 
is a $5$-tuple $(\ell, f , B, D, \sigma)$ such that 
\begin{itemize}
    \item $0 \leq \ell + f < k$;
    \item $(B,D)$ is a~${(k - \ell - f)}$-blueprint with~${2\cdot\abs{D} \leq \abs{V(B)}}$; and 
    \item $\sigma: [\ell+f, k) \to V(B-D) \times \{0,1\}$ is an injective map.
\end{itemize}

Let~${\mathcal{B} = (\ell,f,B,D,\sigma)}$ be a singular $k$-blueprint and let~${\mathcal{K} = ( \kappa_\alpha < \kappa\ |\ \alpha \in \aleph_0 )}$ be a good $\kappa$-sequence.
We construct our desired graph~${T_k(\mathcal{B})(\mathcal{K})}$ as follows.
We start with $\ell$-$FK_{k, \kappa}(\mathcal{K})$
with the same notation as above.
We remove the set ${\{ x_i\ |\ i \in [\ell + f, k) \}}$ from the graph we constructed so far.
Moreover, we take the disjoint union with~${\mathfrak{N}(B/D)}$ as above.
We identify the vertices~${\{y_i^\alpha\ |\ i \in [\ell + f, k) \}}$ with distinct vertices of the~${(2 \alpha + \abs{V(B)})}$-th and~${(2 \alpha + 1 + \abs{V(B)})}$-th layer 
for every~${\alpha \in \aleph_0}$
as given by the map~$\sigma$, that is
\[
    y_{i}^{\alpha} \sim \pi_0(\sigma(i))_{2 \alpha + \pi_1(\sigma(i)) + \abs{V(B)}}
\]
where~$\pi_0$ and~$\pi_1$ denote the projection maps for the tuples in the image of~$\sigma$.
For convenience we denote a vertex originated via such an identification by any of its previous names.
The \emph{core} of $T_k(\mathcal{B})(\mathcal{K})$ is $C(T_k(\mathcal{B})(\mathcal{K})) := \bigcup \mathcal{Z}$.
For an example we refer to Figure~\ref{fig:singBP-small}.

\begin{figure}[htbp]
    \includegraphics[scale=1]{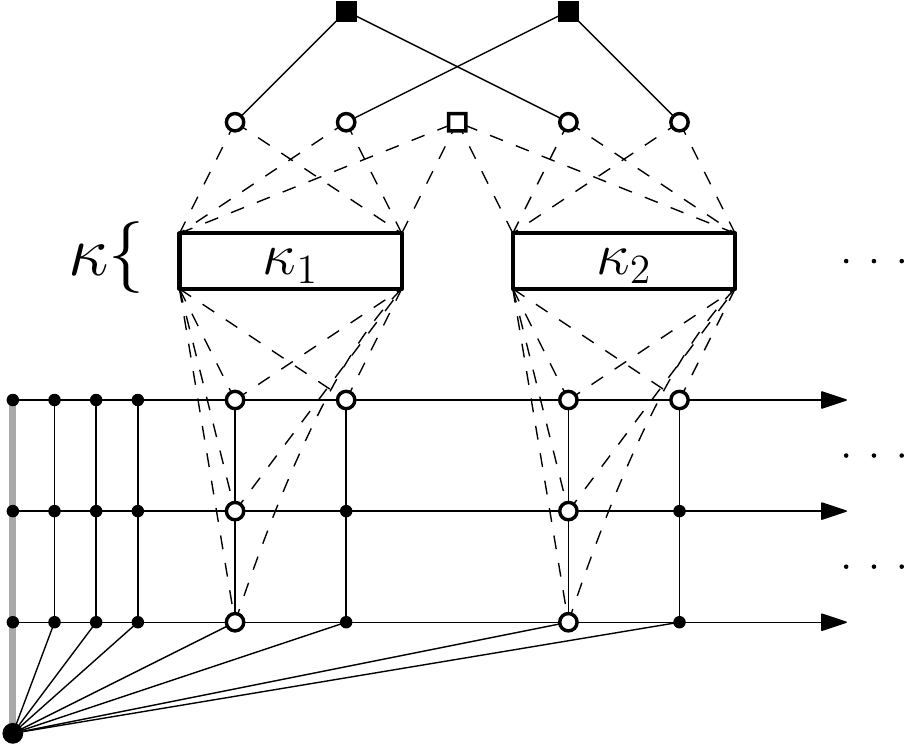}
    \caption{Image of $T_7(1,2,P,\{d\},\sigma)$ for $P$ and $\{d\}$ as in Figure~\ref{fig:regBP-small}.}
    \label{fig:singBP-small}
\end{figure}

As before, the information given by a specific good $\kappa$-sequence does not matter for the structure of the graph.
Similarly, we get with Remark~\ref{rem:cof1} that two graphs ${T_k(\mathcal{B})(\mathcal{K}_0)}$ and ${T_k(\mathcal{B})(\mathcal{K}_1)}$ obtained by different good $\kappa$-sequences $\mathcal{K}_0$, $\mathcal{K}_1$ are isomorphic to fbs-minors of each other.
Hence when we use the fixed good $\kappa$-sequence as before, we call the graph just~$T_k(\mathcal{B})$.

\begin{lemma}\label{lem:typ-singular}
    For a singular $k$-blueprint~$\mathcal{B}$, the core of~$T_k(\mathcal{B})$ is $k$-connected in~$T_k(\mathcal{B})$.
\end{lemma}

\begin{proof}
    Let~${\mathcal{B} = (\ell, f, B, D, \sigma)}$ and let~${C}$ denote the core of~${T_k(\mathcal{B})}$.
    Let ${U, W \subseteq C}$ with ${\abs{U} = \abs{W} = k' \leq k}$.
    Suppose for a contradiction there is a vertex set~$S$ of size less than~$k'$ separating~$U$ and~$W$.
    This separator needs to contain all degenerate vertices as well as block all paths via the frayed centres.
    Hence there are less than~${k' - \ell - f}$ many vertices of~$S$ on~${\mathfrak{N}(B/D)}$, and therefore there is either a $b \in V(B) \setminus D$ such that either $\mathfrak{N}_b$ does not contain a vertex of $S$ or a $d \in D \setminus S$.
    Moreover, there are~${m, n \in \mathbb{N}}$ such that~${u^m \in (U \cap Z^m) \setminus S}$ and~${w^n \in (W \cap Z^n) \setminus S}$.
    Now~${n \neq m}$ since~$S$ cannot separate two vertices of~${Z^n \setminus S}$ in~${K(Y^n,Z^n) \subseteq T_k(\mathcal{B})}$.
    Since the vertices of ${Y^n \cap \mathfrak{N}(B/D)}$ lie on at least ${(k - \ell - f) / 2}$ different rays of the form ${\mathfrak{N}_x}$ for ${x \in V(B)\setminus D}$,
    there is a vertex ${v^n \in (Y^n \cap \mathfrak{N}(B/D)) \setminus S}$ such that the ray~$\mathfrak{N}_{x}$ that contains $v^n$ either has no vertices of~$S$ on its tail starting at~$v^n$ or on its initial segment up to~$v^n$.
    Also, there is an~${N \in \mathbb{N}}$ with $N \geq n$ in the first case and $N \leq n$ in the second case (since $n \geq \abs{V(B)}$) such that~$B_N$ does not contain a vertex of~$S$.
    Hence we can find a path avoiding $S$ starting at $w^n$ and ending on the ray $\mathfrak{N}_b$ or the dominating vertex~$d$.
    Analogously, we get~${v^m \in (Y^m \cap \mathfrak{N}(B/D)) \!\setminus\! S}$,~$B_M$ and a respective path avoiding~$S$. 
    Hence we can connect~$u^m$ and~$w^n$ via a path avoiding~$S$, contradicting the assumption.
\end{proof}

For a singular cardinal~$\kappa$ with countable cofinality and for any singular $k$-blueprint~$\mathcal{B}$ the graph~${T_k(\mathcal{B})}$ is a \emph{$k$-typical graph} with core of size~$\kappa$.
These are the only remaining $k$-typical graphs.

Note that as before there are up to isomorphism only finitely many $k$-typical graphs with a core of size~$\kappa$.

\vspace{0.2cm}

In summary we get for each~${k \in \mathbb{N}}$ and each infinite cardinal~$\kappa$ a finite list of $k$-typical graphs with a core of size~$\kappa$:

\begin{center}
    \begin{tabular}[c]{ c | c | c}
        $\kappa$ & $k$-typical graph $T$ & core $C(T)$\\ \hline
        $\kappa = \cf\kappa > \aleph_0$ & $K_{k,\kappa}$ & $Z$\\
        $\kappa = \cf\kappa = \aleph_0$ & $K_{k,\kappa}$ & $Z$ \\
                                                           & $T_k(B,D,c)$ & $V(\mathfrak{N}_c)$\\
        $\kappa > \cf\kappa > \aleph_0$ & $K_{k,\kappa}$ &  $Z$\\
                                                           & \lFK & $\bigcup \mathcal{Z}$\\
        $\kappa > \cf\kappa = \aleph_0$ & $K_{k,\kappa}$ &  $Z$\\
                                                           & \lFK & $\bigcup \mathcal{Z}$\\
                                                           & $T_k(\ell,f,B,D,\sigma)$ & $\bigcup \mathcal{Z}$
    \end{tabular}
\end{center}

Note that for the finiteness of this list we need the fixed good $\kappa$-sequence for the singular cardinal~$\kappa$.

\begin{lemma}\label{lem:typical-k-conn}
    The core of a $k$-typical graph is $k$-connected in that graph. \qed
\end{lemma}

\subsection{Generalised \emph{k}-typical graphs}
\label{subsec:gen-k-typ}
\ \newline \indent
The $k$-typical graphs cannot serve for a characterisation for the existence of $k$-connected sets as in Theorem~\ref{main-thm-simple-2}\ref{item:t2-subdivision} via subdivisions, as the following example illustrates.
Consider two disjoint copies of the $K_{2, \aleph_0}$ together with a matching between the infinite sides, see Figure~\ref{fig:genK_4,kappa-intro}.
Now the vertices of the infinite side from one of the copies is a $4$-connected set in that graph, but the graph does not contain any subdivision of a $4$-typical graph, since it neither contains a path of length greater than~$13$ (and hence no subdivision of a $T_4(\mathcal{B})$ for some regular $k$-blueprint~$\mathcal{B}$), nor a subdivision of a $K_{4, \aleph_0}$.

To solve this problem we introduce \emph{generalised $k$-typical graphs}, where we `blow up' some of the vertices of our $k$-typical graph to some finite tree, e.g.~an edge in the previous example.
This then will allow us to obtain the desired subdivisions for our characterisation.

\begin{figure}[htbp]
    \includegraphics[scale=0.75]{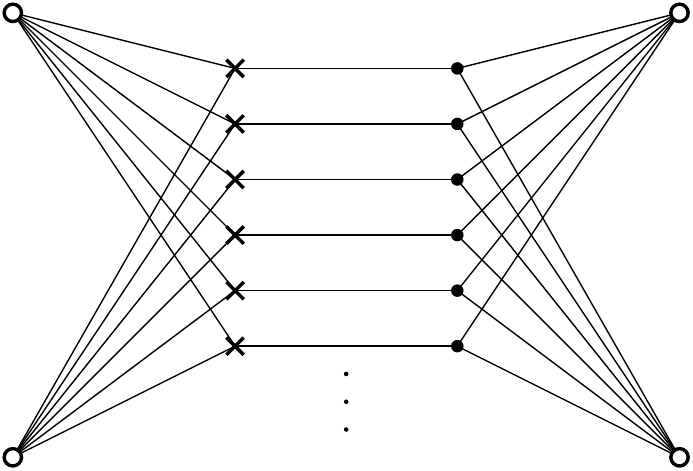}
    \caption{A graph with an infinite $4$-connected set (marked by the cross vertices) containing no subdivision of a $4$-typical graph.}
    \label{fig:genK_4,kappa-intro}
\end{figure}

Let~$G$ be a graph,~${v \in V(G)}$ be a vertex,~$T$ be a finite tree and~${\gamma:\ N(v) \to V(T)}$ be a map.
We define the \emph{$(v,T,\gamma)$-blow-up of~$v$ in~$G$} as the operation where we delete $v$, add a vertex set~${\{v\} \times V(T)}$ disjointly and for each~${w \in N(v)}$ add the edge between~$w$ and~${(v,\gamma(w))}$.
We call the resulting graph~$G(v,T,\gamma)$.

Given blow-ups~${(v, T_v, \gamma_v)}$ and ${(w,T_w,\gamma_w)}$ in~$G$, we can apply the blow-up of~$w$ in $G(v, T_v, \gamma_v)$ by replacing~$v$ in the preimage of~$\gamma_w$ by~${(v, \gamma_v(w))}$.
We call this graph ${G(v, T_v, \gamma_v)(w, T_w, \gamma_w)}$.
Note that no matter in which order we apply the blow-ups we obtain the same graph, that is ${G(v, T_v, \gamma_v)(w, T_w, \gamma_w) = G(w, T_w, \gamma_w)(v, T_v, \gamma_v)}$.
We analogously define for a set ${O = \{ (v, T_v,\gamma_v)\ |\ v \in W \}}$ of blow-ups for some~${W \subseteq V(G)}$ the graph~${G(O)}$ obtained by successively applying all the blow-ups in~$O$.
Note that if~$W$ is infinite, then~${G(O)}$ is still well-defined, since each edge gets each of its endvertices modified at most once.

\vspace{0.2cm}

A \emph{type-$1$ $k$-template}~$\mathcal{T}_1$ is a triple~${(T,\gamma,c)}$ 
consisting of a finite tree~$T$, a map \linebreak
${\gamma: [0,k) \to V(T)}$ and a node~${c \in V(T)}$ such that each node of degree~$1$ or~$2$ in~$T$ is either~$c$ or in the image of~$\gamma$.
Note that for each~$k$ there are only finitely many type-1 $k$-templates up to isomorphisms of the trees, since their trees have order at most~${2k+1}$.

Let~${\mathcal{T}_1 = (T,\gamma,c)}$ be a type-1 $k$-template and 
let~${O_1 := \{ (z, T, \gamma)\ |\ z \in C(K_{k,\kappa}) \}}$.
We call the graph~${K_{k,\kappa}(\mathcal{T}_1) := K_{k,\kappa}(O_1)}$ a \emph{generalised~$K_{k,\kappa}$}.
The \emph{core}~$C({K_{k,\kappa}(\mathcal{T}_1)})$ is the set ${C(K_{k,\kappa}) \times \{c\}}$, 
see Figure~\ref{fig:genK4kappa} for an example.
Note that Figure~\ref{fig:genK_4,kappa-intro} is also an example.

\begin{figure}[htbp]
    \includegraphics[scale=1.5]{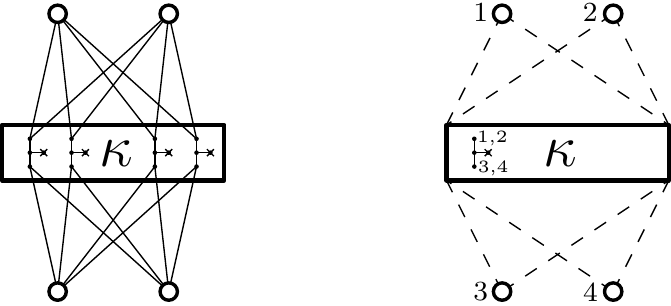}
    \caption{Image of a generalised $K_{4,\kappa}$ on the left. The crosses represent the core.
    On the right is how we represent the same graph in a simplified way 
    by labelling the vertices according to their adjacencies.}
    \label{fig:genK4kappa}
\end{figure}

Similarly, with~$\mathcal{T}_1$ as above, let~${O_1' := \{ (z, T, \gamma_\alpha)\ |\ \alpha \in \cf\kappa,  z \in Z^\alpha \}}$, where~$\gamma_\alpha$ denotes the map defined by~${y_i^\alpha \mapsto \gamma(i)}$.
The graph~${\lK{\mathcal{K}}(\mathcal{T}_1) := \lK{\mathcal{K}}(O_1')}$ is a \emph{generalised~$\lK{\mathcal{K}}$}.
We call the vertex set~${\bigcup \mathcal{Z} \times \{c\}}$ the \emph{precore} of that graph.

Analogously, we obtain a \emph{generalised~$\lFK(\mathcal{K})$} for any good $\kappa$-sequence~$\mathcal{K}$ as \linebreak ${\lFK(\mathcal{K})(\mathcal{T}_1) := \lFK(O_1')}$
with \emph{core}~${C({\lFK(\mathcal{K})(\mathcal{T}_1)}) := \bigcup \mathcal{Z} \times \{c\}}$.

\vspace{0.2cm}

A \emph{type-$2$ $k$-template}~$\mathcal{T}_2$ for a $k$-blueprint~${(B,D)}$ is a set 
${\{ (b, P_b, \gamma_b)\ |\ b \in V(B) \!\setminus\! D \}}$ 
of blow-ups in~$B$ such that for all $b \in V(B) \!\setminus\! D$
\begin{itemize}
    \item $P_b$ is a path of length at most $k+2$;
    \item the endnodes of $P_b$ are called $v_0^b$ and $v_1^b$;
    \item $P_b$ contains nodes $v_\bot^b$ and $v_\top^b$;
    \item the nodes $v_0^b, v_\bot^b, v_\top^b, v_1^b$ need not be distinct;
    \item if $v_0^n \neq v_\bot^b$, then $v_0^b v_\bot^b \in E(P_b)$ and  if $v_1^n \neq v_\top^b$, then $v_1^b v_\top^b \in E(P_b)$;
    \item $\gamma_b(N(b)) \subseteq v_\bot^b P_b v_\top^b$;
\end{itemize}
We say $\mathcal{T}_2$ is \emph{simple} if $v_0^b = v_\bot^b$ and $v_1^b = v_\top^b$. 
Note that for each~$k$ there are only finitely many type-2 $k$-templates, up to isomorphisms of the trees in the $k$-blueprints and the paths for the blow-ups.

Let ${\mathcal{T}_2 = \{ (b, T_b, \gamma_b)\ |\ b \in V(B) \!\setminus\! D \}}$ be a type-2 $k$-template for a $k$-blueprint $(B,D)$.
Then ${O_2 := \{ (b_n, T_b, \gamma_b^n)\ |\ n \in \mathbb{N}, b \in V(B) \!\setminus\! D \}}$ is a set of blow-ups in $\mathfrak{N}(B/D)$, where $\gamma_b^n$ is defined via
\[
    \gamma_b^n( v ) =
    \begin{cases}
    \begin{aligned}
        & \gamma_b(b')  && \text{ if } v = b'_n \text{ for } b' \in N(b); \\
        & v^b_\top && \text{ if } v = b_{n+1}; \\
        & v^b_\bot && \text{ if } n \geq 1 \text{ and } v = b_{n-1}.
    \end{aligned}
    \end{cases}
\]
Then $\mathfrak{N}(B/D)(\mathcal{T}_2) := \mathfrak{N}(B/D)(O_2)$ is a \emph{generalised~$\mathfrak{N}(B/D)$}.

Let~${\mathcal{B} = (B, D, c)}$ be a regular $k$-blueprint and let
${\mathcal{T}_2 = \{ (b, T_b, \gamma_b)\ |\ b \in V(B) {\setminus} D \}}$ be a type-2 $k$-template for~$(B,D)$.
We call $T_k(\mathcal{B})(\mathcal{T}_2) := T_k(\mathcal{B})(O_2)$ a \emph{generalised~$T_k(\mathcal{B})$} with \emph{core}~${C(T_k(\mathcal{B})(\mathcal{T}_2)) := V(\mathfrak{N}_c) \times \{v_1^c\}}$.
For an example that generalises the graph of Figure~\ref{fig:regBP-small}, see Figure~\ref{fig:genregBP-small}.

\begin{figure}[htbp]
    \includegraphics[scale=1]{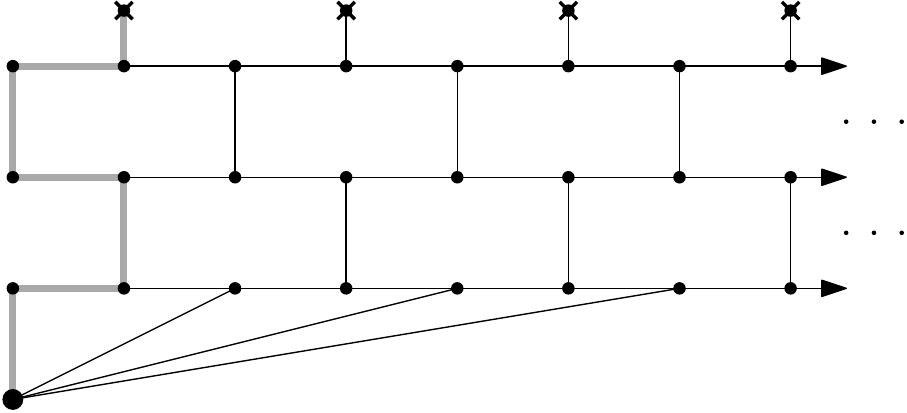}
    \caption{Image of a generalised $T_4(P,\{d\},c)$ for $P$, $\{d\}$, $c$ as in Figure~\ref{fig:regBP-small}.
    In grey we represent the blow-up of $P$ as given by some type-2 $k$-template.
    The crosses represent the core.}
    \label{fig:genregBP-small}
\end{figure}

\vspace{0.2cm}

A \emph{type-$3$ $k$-template} $\mathcal{T}_3$ for a singular $k$-blueprint ${\mathcal{B} = (\ell,f,B,D,\sigma)}$ is a tuple 
${(\mathcal{T}_1,\mathcal{T}_2)}$ consisting of 
a type-1 $(\ell+f)$-template~$\mathcal{T}_1$ and a type-2 
$(k-\ell-f)$-template~$\mathcal{T}_2$.
Note that for each~$k$ there are only finitely many type-3 $k$-templates up to isomorphisms as discussed above for~$\mathcal{T}_1$ and~$\mathcal{T}_2$.

Let $\mathcal{T}_3 = (\mathcal{T}_1, \mathcal{T}_2)$ be a type-3 $k$-template with ${\mathcal{T}_1 = (T,\gamma,c_1)}$ for a singular $k$-blueprint ${\mathcal{B} = (\ell,f,B,D,\sigma)}$.
Then for $(b_n, T_b, \gamma_b^n) \in O_2$ we extend $\gamma_b^n$ to $\hat{\gamma}_b^n$ via 
\[
    \hat{\gamma}_b^n( v ) =
    \begin{cases}
    \begin{aligned}
        & v_1^b && \text{ if } v \in \{ y_i^n\ |\ i \in [\ell+f,k) \} \text{ and } n \text{ even};\\
        & v_0^b && \text{ if } v \in \{ y_i^n\ |\ i \in [\ell+f,k) \} \text{ and } n \text{ odd}; \\
        &\gamma_b^n(v) && \text{ otherwise}. 
    \end{aligned}
    \end{cases}
\]
Let ${O_2' := \{ (b_n, T_b, \hat{\gamma}_b^n)\ |\ (b_n, T_b, \gamma_b^n) \in O_2 \}}$ denote the corresponding set of blow-ups in~$T_k(\mathcal{B})$ and let~$O_1'$ be for $\mathcal{T}_1$ as above.
The graph 
${T_k(\mathcal{B})(\mathcal{T}_3) := T_k(\mathcal{B})(O_1' \cup O_2')}$ is a \emph{generalised~$T_k(\mathcal{B})$} with \emph{core}~${C(T_k(\mathcal{B})(\mathcal{T}_3)) := \bigcup \mathcal{Z} \times \{c_1\}}$.
For an example that generalises the graph of Figure~\ref{fig:singBP-small} see Figure~\ref{fig:gensingBP-small}.

\begin{figure}[htbp]
    \includegraphics[scale=1.2]{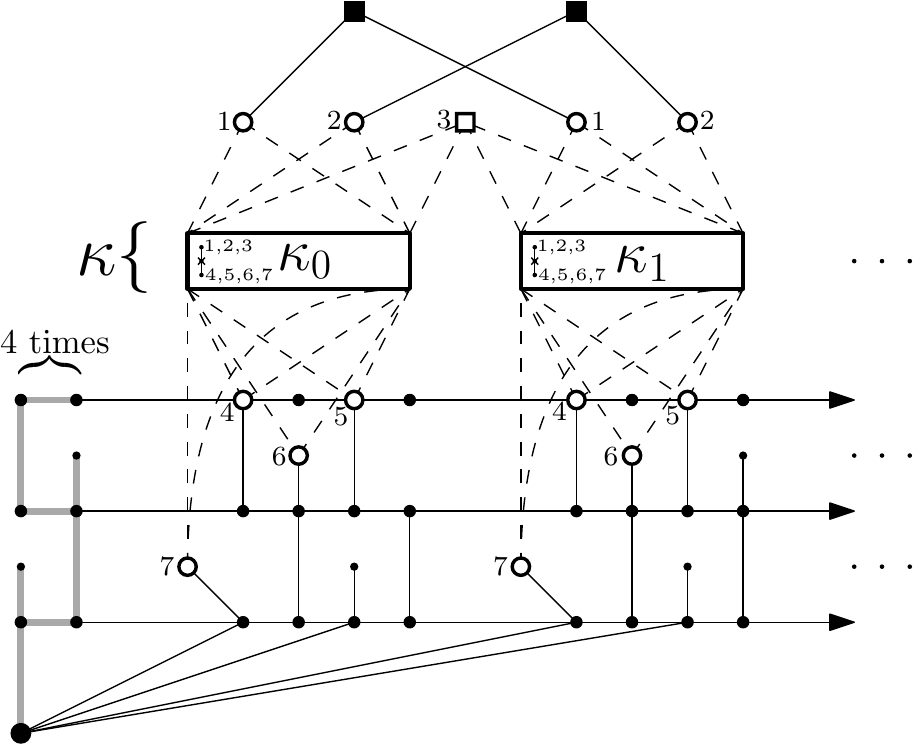}
    \caption{Image of a generalised $T_7(1,2,P,\{d\},\sigma)$ for $P$, $\{d\}$, $\sigma$ as in Figure~\ref{fig:singBP-small}.}.
    \label{fig:gensingBP-small}
\end{figure}

\vspace{0.2cm}

We call the graph from which a generalised graph is obtained via this process its~\emph{parent}.
As before, Remark~\ref{rem:cof1} and its extensions extend to generalised $k$-typical graphs as well.

\begin{remark}\label{rem:cof2}
    Every~${\lFK(\mathcal{T}_1)(\mathcal{K})}$ or~${T_k(\mathcal{B})(\mathcal{T}_3)(\mathcal{K})}$ for a singular $k$-blueprint~$\mathcal{B}$, a type-1 $k$-template~$\mathcal{T}_2$, a type-2 $k$-template~$\mathcal{T}_3$ and a good $\kappa$-sequence~$\mathcal{K}$, contains a subdivision of ${\lFK(\mathcal{T}_1)}$ or ${T_k(\mathcal{B})(\mathcal{T}_3)}$ respectively.
\end{remark}

A \emph{generalised $k$-typical graph} is either~${K_{k,\kappa}(\mathcal{T}_1)}$,~${\lFK(\mathcal{T}_1)}$,~${T_k(\mathcal{B})(\mathcal{T}_2)}$ or ${T_k(\mathcal{B}')(\mathcal{T}_3)}$ for any type-1 $k$-template~$\mathcal{T}_1$, any~${\ell \in [0,k)}$, any regular $k$-blueprint~$\mathcal{B}$, any type-2 $k$-template~$\mathcal{T}_2$ for~$\mathcal{B}$, any singular $k$-blueprint~$\mathcal{B}'$ and any type-3 $k$-template~$\mathcal{T}_3$ for~$\mathcal{B}'$.
As with the $k$-typical graphs we obtain that this list is finite.

\begin{corollary}\label{cor:gen-typ-k-conn}
    The core of a generalised $k$-typical graph is $k$-connected in that graph. \qed
\end{corollary}

\subsection{Statement of the Main Theorem}
\label{subsec:main-thm}
\ \newline \indent
Now that we introduced all $k$-typical and generalised $k$-typical graphs, let us give the full statement of our main theorem.

\begin{thm-intro}\label{main-thm}
    Let~$G$ be an infinite graph, let~${k \in \mathbb{N}}$, 
    let~${A \subseteq V(G)}$ be infinite 
    and let~${\kappa \leq \abs{A}}$ be an infinite cardinal. 
    Then the following statements are equivalent.
    \begin{enumerate}[label=(\alph*)]
        \item\label{item:t3-set} There is a subset~${A_1 \subseteq A}$ with~${\abs{A_1} = \kappa}$ such that~$A_1$ is $k$-connected in~$G$.
        \item\label{item:t3-minor} There is a subset~${A_2 \subseteq A}$ with~${\abs{A_2} = \kappa}$ such that
        there is a $k$-typical graph 
        which is a minor of~$G$ 
        with finite branch sets and 
        with~$A_2$ along its core.
        \item\label{item:t3-subdivision} There is a subset~${A_3 \subseteq A}$ with~${\abs{A_3} = \kappa}$ 
        such that~$G$ contains a subdivided generalised $k$-typical graph 
        with~$A_3$ as its core.
        \item\label{item:t3-duality} There is no tree-decomposition of~$G$ of adhesion less than~$k$ such that every part can be separated from~$A$ by less than~$\kappa$ vertices. 
    \end{enumerate}
    Moreover, if these statements hold, we can choose~${A_1 = A_2 = A_3}$.
\end{thm-intro}

Note that for $A = V(G)$ we obtain the simple version  as in Theorems~\ref{main-thm-simple-1} and~\ref{main-thm-simple-2} by forgetting the extra information about the core.

\section{\emph{k}-connected sets, minors and topological minors}
\label{sec:k-connSets}

In this section we will collect a few basic remarks and lemmas on $k$-connected sets and how they interact with minors and topological minors for future references.
We omit some of the trivial proofs.

\begin{remark}\label{rem:k-connected-subset}
    If $A \subseteq V(G)$ is $k$-connected in $G$, then any $A' \subseteq A$ with $\abs{A'} \geq k$ is $k$-connected in $G$ as well. \qed
\end{remark}

\begin{lemma}\label{lem:k-con-minor}
    If~$M$ is a minor of~$G$ and~${A \subseteq V(M)}$ is $k$-connected in~$M$ for some~${k \in \mathbb{N}}$, 
    then any set~${A' \subseteq V(G)}$ with~${\abs{A'} \geq k}$ consisting of at most one vertex of each branch set for the vertices of~$A$ is $k$-connected in~$G$. \qed
\end{lemma}

\begin{lemma}\label{lem:gen-sub-fbs}
    For $k \in \mathbb{N}$, if $G$ contains the subdivision of a generalised $k$-typical graph $T$ with core~$A$, then the parent of $T$ is an fbs-minor with $A$ along its core.
    \qed
\end{lemma}

A helpful statement for the upcoming inductive constructions would be that for every vertex~$v$ of~$G$, every large $k$-connected set in~$G$ contains a large subset which is $(k-1)$-connected in~${G - v}$.
But while this is a true statement (cf.~Corollary~\ref{cor:(k-1)-connected}), an elementary proof of it seems to be elusive if $v$ is not itself contained in the original $k$-connected set.
The following lemma is a simplified version of that statement and has an elementary proof.

\begin{lemma}\label{lem:delete<k}
    Let ${k \in \mathbb{N}}$ and let~${A \subseteq V(G)}$ be infinite and $k$-connected in~$G$.
    Then for any finite set~${S \subseteq V(G)}$ with~${\abs{S} < k}$ there is a subset~${A' \subseteq A}$ with~${\abs{A'} = \abs{A}}$ such that~$A'$ is ${1}$-connected in~${G-S}$.
\end{lemma}

\begin{proof}
    Without loss of generality we may assume that~$A$ and~$S$ are disjoint.
    Take a sequence ${(B_\alpha\ |\ \alpha \in \abs{A})}$ of disjoint subsets of~$A$ with~${\abs{B_\alpha} = k}$.
    For  every ${\alpha \in \abs{A} \setminus \{0\}}$ there is at least one path from~$B_0$ to~$B_\alpha$ disjoint from~$S$.
    By the pigeonhole principle there is some~${v \in B_0}$ such that~$\abs{A}$ many of these paths start in~$v$.
    Now let~$A'$ be the set of endvertices of these paths.
\end{proof}

\section{Structure within ends}
\label{sec:end-structure}

This section studies the structure within an end of a graph.

In Subsection~\ref{subsec:def-seq} we will extend to arbitrary infinite graphs a well-known result for locally finite graphs relating end degree with a certain sequence of minimal separators, making use of the combined end degree.

Subsection~\ref{subsec:T-conn-rays} is dedicated to the construction of a uniformly connecting structure between disjoint rays in a common end and vertices dominating that end.

\subsection{End defining sequences and combined end degree}
\label{subsec:def-seq}
\ \newline \indent
For an end~${\omega \in \Omega(G)}$ and a finite set~${S \subseteq V(G)}$ let~$C(S, \omega)$ denote the unique component of~${G - S}$ that contains $\omega$-rays.
A sequence ${(S_n\ |\ n \in \mathbb{N})}$ of finite vertex sets of~$G$ is called an \emph{$\omega$-defining sequence} if for all ${n, m \in \mathbb{N}}$ with ${n \neq m}$ the following hold:
\begin{itemize}
    \item ${C(S_{n+1},\omega) \subseteq C(S_n, \omega)}$; 
    \item $S_n \cap S_m \subseteq \Dom(\omega)$; and
    \item ${\bigcap \{ C(S_n,\omega)\ |\ n \in \mathbb{N} \} = \emptyset}$.
\end{itemize}
Note that for every $\omega$-defining sequence ${(S_n\ |\ n \in \mathbb{N})}$ and every finite set ${X \subseteq V(G)}$ we can find an ${N \in \mathbb{N}}$ such that ${X \subseteq G - C(S_{N}, \omega)}$.
Hence we shall also refer to the sets $S_n$ in such a sequence as \emph{separators}.
Given $n, m \in \mathbb{N}$ with $n < m$, let $G[S_n, S_m]$ denote ${G[ (S_n \cup C(S_n, \omega)) \!\setminus\! C(S_m, \omega) ]}$, the \emph{graph between the separators}.

For ends of locally finite graphs there is a characterisation of the end degree given by the existence of certain $\omega$-defining sequences.
The degree of an end $\omega$ is equal to~$k \in \mathbb{N}$, if and only if $k$ is the smallest integer such that there is an $\omega$-defining sequence of sets of size $k$, cf.\ \cite{Stein:ext-survey}*{Lemma~3.4.2}.
In this subsection we extend this characterisation to arbitrary graphs with respect to the combined degree.
Recall the definition of the combined degree,~${\Delta(\omega) := \deg(\omega) + \dom(\omega)}$.

In arbitrary graphs $\omega$-defining sequences need not necessarily exist, e.g.\ in $K_{\aleph_1}$.
We start by characterising the ends admitting such a sequence.

\begin{lemma}\label{lem:def-seq}
    Let $\omega \in \Omega(G)$ be an end.
    Then there is an $\omega$-defining sequence ${(S_n\ |\ n \in \mathbb{N})}$ if and only if ${\Delta(\omega) \leq \aleph_0}$.
\end{lemma}

\begin{proof}
    Note that for all finite ${S \subseteq V(G)}$, no ${d \in \Dom(\omega)}$ can lie in a component ${C \neq C(S, \omega)}$ of~$G - S$.
    Hence for every $\omega$-defining sequence~${(S_n\ |\ n \in \mathbb{N})}$ 
    and every ${d \in \Dom(\omega)}$ 
    there is an~$N \in \mathbb{N}$ such that~${d \in S_m}$ for all~${m \geq N}$.
    Therefore, if ${\dom(\omega) > \aleph_0}$, no $\omega$-defining sequence can exist, since 
    the union of the separators is at most countable.
    Moreover, note that for every ${\omega\text{-defining}}$ sequence 
    every $\omega$-ray meets infinitely many distinct separators. 
    It follows that~$\deg(\omega)$ is at most countable as well if an $\omega$-defining sequence exist.
    
    For the converse, suppose ${\Delta(\omega) \leq \aleph_0}$.
    Let ${\{ d_n\ |\ n < \dom(\omega) \}}$ be an enumeration of~$\Dom(\omega)$. 
    Let $R = r_0 r_1 \hdots$ be an $\omega$-devouring ray, which exists by Lemma~\ref{lem:end_devouring}.
    We build our desired $\omega$-defining sequence~${(S_n\ |\ n \in \mathbb{N})}$ inductively.
    Set~$S_0 := \{ r_0 \}$.
    For~${n \in \mathbb{N}}$ suppose~$S_n$ is already constructed as desired.
    Take a maximal set $\mathcal{P}_n$ of pairwise disjoint $N(S_n \setminus \Dom(\omega))$\,--\,$R$ paths in~$C(S_n, \omega)$.
    Note that $\mathcal{P}_n$ is finite since otherwise by the pigeonhole principle we would get a vertex ${v \in S_n \setminus \Dom(\omega)}$ dominating~$\omega$.
    Furthermore, ~$\mathcal{P}_n$ is not empty as~$C(S_n, \omega)$ is connected.
    Define\
    \[
        \begin{aligned}
            S_{n+1} := &\quad\ 
            (S_n \cap \Dom(\omega)) 
            \cup \bigcup \mathcal{P}_n \\
            &\cup \big\{ r_m\ |\ m \text{ is minimal with } r_m \in C(S_n, \omega) \big\} \\
            &\cup \big\{ d_{m}\ |\ m \text{ is minimal with } d_m \in C(S_n, \omega) \big\}
            .
        \end{aligned}
    \]
    By construction, $S_{n+1} \cap S_i$ contains only vertices dominating $\omega$ for~${i \leq n}$.
    Let~$P$ be any $S_n$\,--\,$C(S_{n+1}, \omega)$ path.
    We can extend~$P$ in ${C(S_{n+1}, \omega)}$ to an $S_n$\,--\,$R$ path.
    And since ${C(S_{n+1}, \omega) \cap S_{n+1}}$ is empty, we obtain~${P \cap S_{n+1} \neq \emptyset}$ by construction of~$S_{n+1}$.
    Hence any $S_n$\,--\,$C(S_{n+1}, \omega)$ path meets~$S_{n+1}$.
    Since for any vertex ${v \in C(S_{n+1}, \omega) \setminus C(S_n, \omega)}$ there is a path to $C(S_{n}, \omega) \cap C(S_{n+1}, \omega)$ in $C(S_{n+1}, \omega)$, this path would meet a vertex~$w \in S_n$.
    This vertex would be a trivial $S_n$\,--\,$C(S_{n+1}, \omega)$ path avoiding $S_{n+1}$, and hence contradicting the existence of such~$v$.
    Hence~${C(S_{n+1}, \omega) \subseteq C(S_n, \omega)}$.
    
    Suppose there is a vertex ${v \in \bigcap \{ C(S_n, \omega)\ |\ n \in \mathbb{N} \}}$.
    By construction~$v$ is neither dominating~$\omega$ nor is a vertex on~$R$.
    Note that every~$v$\,--\,$R$ path has to contain vertices from infinitely many~$S_n$, hence it has to contain a vertex dominating~$\omega$.
    For each ${d \in \Dom(\omega)}$ let~$P_d$ be either the vertex set of a~$v$\,--\,$\Dom(\omega)$ path containing~$d$ if it exists, or ${P_d = \emptyset}$ otherwise.
    If ${X := \bigcup \{ P_d\ |\ d \in \Dom(\omega) \}}$ is finite, 
    we can find an $N \in \mathbb{N}$ such that ${v \in X \subseteq G - C(S_N, \omega)}$, a contradiction. 
    Otherwise apply Lemma~\ref{star-comb} to ${X \cap \Dom(\omega)}$ in~$G[X]$.
    Note that in~$G[X]$ all vertices of $X \cap \Dom(\omega)$ have degree~$1$ in~$G[X]$.
    Furthermore, we know that $V(R) \cap X \subseteq \Dom(\omega)$, since no $P_d$ contains a vertex of $R$ as an internal vertex.
    But then the centre of a star would be a vertex dominating~$\omega$ in~${X \setminus \Dom(\omega)}$ and the spine of a comb would contain an $\omega$-ray disjoint to~$R$ as a tail, again a contradiction.
\end{proof}

In the proof of the end-degree characterisation via $\omega$-defining sequences we shall need the following fact regarding the relationship of $\deg(\omega)$ and $\dom(\omega)$.

\begin{lemma}\label{lem:uncount-end-deg}
    If $\deg(\omega)$ is uncountable for $\omega \in \Omega(G)$, then $\dom(\omega)$ is infinite.
\end{lemma}

\begin{proof}
    Suppose for a contradiction that~${\dom(\omega) < \aleph_0}$.
    For $G' := G - \Dom(\omega)$ let~$\mathcal{R}$ be a set of disjoint ${\omega\restricted G'}$-rays of size~$\aleph_1$, 
    which exist by Remark~\ref{rem:end_subgraph}.
    Let~$T$ be a transversal of ${\{ V(R)\ |\ R \in \mathcal{R} \}}$.
    Applying Lemma~\ref{star-comb} to $T$ yields a subdivided star with centre $d$ and uncountably many leaves in~$T$.
    Now $d \notin \Dom(\omega)$ dominates $\omega\restricted G'$ in $G'$ and hence $\omega$ in $G$ by Remark~\ref{rem:end_subgraph}, a contradiction.
\end{proof}

Let~${\omega \in \Omega(G)}$ be an end with~${\dom(\omega) = 0}$, 
$(S_n\ |\ n \in \mathbb{N})$ be an $\omega$-defining sequence  
and~$\mathcal{R}$ be a set of disjoint $\omega$-rays.
We call ${((S_n\ |\ n \in \mathbb{N}), \mathcal{R})}$ a \emph{degree witnessing pair} for~$\omega$, 
if for all $n \in \mathbb{N}$ and for each $s \in S_n$ there is a ray $R \in \mathcal{R}$ containing $s$ and every ray $R \in \mathcal{R}$ meets $S_n$ at most once for every $n \in \mathbb{N}$.
Note that this definition only makes sense for undominated ends, 
since a ray that contains a dominating vertex meets eventually all separators not only in that vertex.

\begin{lemma}\label{lem:def-seq-deg-wit}
    Let~${\omega \in \Omega(G)}$ be an end with~${\dom(\omega) = 0}$.
    Then there  is a degree witnessing pair~${((S_n\ |\ n \in \mathbb{N}), \mathcal{R})}$.
\end{lemma}

\begin{proof}
    By Lemma~\ref{lem:def-seq} and Lemma~\ref{lem:uncount-end-deg} there is an $\omega$-defining sequence ${(S'_n\ |\ n \in \mathbb{N})}$.
    Since $\omega$ is undominated, the separators are pairwise disjoint.
    
    We want to construct an $\omega$-defining sequence ${(S_n\ |\ n \in \mathbb{N})}$ with the property, that for all~${n \in \mathbb{N}}$ and for all~${m > n}$ there are~$\abs{S_n}$ many $S_n$\,--\,$S_{m}$ paths in~${G[S_n, S_{m}]}$.
    
    Let~$S_0$ be an $S_0'$\,--\,$S_{f(0)}'$ separator for some~${f(0) \in \mathbb{N}}$ which is of minimum order among all candidates separating~$S_0'$ from~$S_m'$ for any~${m \in \mathbb{N}}$.
    Suppose we already constructed the sequence up to~$S_n$.
    Let~$S_{n + 1}$ be an $S_{f(n)+1}'$\,--\,$S_{f(n+1)}'$ separator for some~${f(n+1) > f(n) + 1}$ which is of minimum order among all candidates separating~$S_{f(n) + 1}'$ and~$S_m'$ for \linebreak any~${m > f(n) + 1}$.
    
    Note that~$S_m$ and~$S_n$ are disjoint for all~${m,n \in \mathbb{N}}$ with~${n \neq m}$ and that \linebreak
    ${C(S_{n+1},\omega) \subseteq C(S_{f(n)+1}',\omega)}$ for all~${n \in \mathbb{N}}$. 
    Hence ${(S_n\ |\ n \in \mathbb{N})}$ is an $\omega$-defining sequence.
    Moreover, note that ${\abs{S_n} \leq \abs{S_{n+1}}}$ for all~${n \in \mathbb{N}}$, since $S_{n+1}$ would have been a candidate for~$S_n$ as well.
    In particular, there is no $S_n$\,--\,$S_{n+1}$ separator $S$ of order less than $\abs{S_n}$ for every $n \in \mathbb{N}$, since this would also have been a candidate for~$S_n$.
    Hence by Theorem~\ref{finite-parameter-menger} there is a set of~$\abs{S_n}$ many disjoint $S_n$\,--\,$S_{n+1}$ paths~$\mathcal{P}_n$ in~$G[S_n, S_{n+1}]$.
    
    Now the union~${\bigcup \{ \bigcup \mathcal{P}_n\ |\ n \in \mathbb{N} \}}$ is by construction a union of a set $\mathcal{R}$ of rays, 
    since the union of the paths in $\mathcal{P}_n$ intersect the union of the paths in $\mathcal{P}_m$ in precisely $S_{n+1}$ if $m = n+1$ and are disjoint if $m > n+1$.
    These rays are necessarily $\omega$-rays, meet every separator at most once and every $s \in S_n$ is contained in one of them,
    proving that ${((S_n\ |\ n \in \mathbb{N}), \mathcal{R})}$ is a degree witnessing pair for~$\omega$.
\end{proof}

\begin{corollary}\label{cor:end-deg}
    Let $k \in \mathbb{N}$ and let $\omega \in \Omega(G)$ with ${\dom(\omega) = 0}$.
    Then ${\deg(\omega) \geq k}$ if and only if for every $\omega$-defining sequence ${(S_n\ |\ n \in \mathbb{N})}$ the sets~$S_n$ eventually have size at least~$k$.
\end{corollary}

\begin{proof}
    Suppose ${\deg(\omega) \geq k}$.
    Let ${(S_n\ |\ n \in \mathbb{N})}$ be any $\omega$-defining sequence.
    Then each ray out of a set of~$k$ disjoint~$\omega$-rays has to go through eventually all $S_n$.
    For the other direction take a degree witnessing pair ${((S_n\ |\ n \in \mathbb{N}), \mathcal{R})}$.
    Now~${\abs{\mathcal{R}} \geq k}$, since eventually all~$S_n$ have size at least~$k$.
\end{proof}

\begin{corollary}\label{cor:end-deg1}
    Let ${k \in \mathbb{N}}$ and let ${\omega \in \Omega(G)}$ with ${\dom(\omega) = 0}$.
    Then ${\deg(\omega) = k}$ if and only if~$k$ is the smallest integer such that there is an $\omega$-defining sequence ${(S_n\ |\ n \in \mathbb{N})}$ with ${\abs{S_n} = k}$ for all ${n \in \mathbb{N}}$. \qed
\end{corollary}

We can easily lift these results to ends dominated by finitely many vertices with the following observation based on Remark~\ref{rem:end_subgraph}.

\begin{remarks}\label{rem:end_deg}
    Suppose ${\dom(\omega) < \aleph_0}$.
    Let~$G'$ denote ${G - \Dom(\omega)}$.
    \begin{enumerate}[label=(\alph*)]
        \item \label{item:end-deg-a}
            For every $\omega{\upharpoonright}G'$-defining sequence ${(S'_n\ |\ n \in \mathbb{N})}$ of $G'$ there is an $\omega$-defining sequence ${(S_n\ |\ n \in \mathbb{N})}$ of $G$ with $S'_n = S_n \setminus \Dom(\omega)$ for all $n \in \mathbb{N}$.
        \item \label{item:end-deg-b} 
            For every $\omega$-defining sequence ${(S_n\ |\ n \in \mathbb{N})}$ of $G$ there is an $\omega{\upharpoonright}G'$-defining sequence ${(S'_n\ |\ n \in \mathbb{N})}$ of $G'$ with $S'_n = S_n \setminus \Dom(\omega)$ for all $n \in \mathbb{N}$.
        \qed
    \end{enumerate}
\end{remarks}

\begin{corollary}\label{cor:end-deg2}
    Let $k \in \mathbb{N}$ and let $\omega \in \Omega(G)$. 
    Then $\Delta(\omega) \geq k$ if and only if for every $\omega$-defining sequence ${(S_n\ |\ n \in \mathbb{N})}$ the sets $S_n$ eventually have size at least $k$.
\end{corollary}

\begin{proof}
    As noted before, each vertex dominating $\omega$ has to be in eventually all sets of an $\omega$-defining sequence.
    
    Suppose $\Delta(\omega) \geq k$.
    If $\dom(\omega) \geq \aleph_1$, then there is no $\omega$-defining sequence and there is nothing to show.
    If $\dom(\omega) = \aleph_0$, then the sets of any $\omega$-defining sequence eventually have all size at least $k$.
    If $\dom(\omega) < \aleph_0$, we can delete $\Dom(\omega)$ and apply Corollary~\ref{cor:end-deg} to $G - \Dom(\omega)$ with $k' = \deg(\omega)$.
    With Remark~\ref{rem:end_deg}\ref{item:end-deg-b} the claim follows.
    
    If $\Delta(\omega) < k$, we can delete $\Dom(\omega)$ and apply Corollary~\ref{cor:end-deg} with ${k' = \deg(\omega)}$.
    With Remark~\ref{rem:end_deg}\ref{item:end-deg-a} the claim follows.
\end{proof}

\begin{corollary}\label{cor:end-deg3}
    Let $k \in \mathbb{N}$ and let $\omega \in \Omega(G)$. 
    Then $\Delta(\omega) = k$ if and only if $k$ is the smallest integer such that there is an $\omega$-defining sequence ${(S_n\ |\ n \in \mathbb{N})}$ with $\abs{S_n} = k$ for all $n \in \mathbb{N}$.
\end{corollary}

\begin{proof}
    As before, we delete $\Dom(\omega)$ and apply Corollary~\ref{cor:end-deg1} with $k' = \deg(\omega)$ and Remark~\ref{rem:end_deg}.
\end{proof}

Finally, we state more remarks on the relationship between $\deg(\omega)$ and $\dom(\omega)$ similar to Lemma~\ref{lem:uncount-end-deg} without giving the proof.

\begin{remarks}\label{rem:comb-end-deg}
    Let $\kappa_1$, $\kappa_2$ be infinite cardinals and let $k_1$, $k_2 \in \mathbb{N}$.
    \begin{itemize}
        \item[(1)] If $\dom(\omega)$ is infinite, then so is $\deg(\omega)$ 
            for every $\omega \in \Omega(G)$.
        \item[(2)] If $\Delta(\omega)$ is uncountable, then both $\deg(\omega)$ and $\dom(\omega)$ are infinite 
            for every $\omega \in \Omega(G)$.
        \item[(3)] There is a graph with an end $\omega'$ such that $\deg(\omega') = \kappa_1$ and $\dom(\omega') = \kappa_2$.
        %Construction:
            %Take the union of $K(\{c\},\kappa_1) \times \mathfrak{N}$ with $K(\{c\} \times \mathbb{N}, \kappa_2)$$.
            %Then each ray in this graph with its unique end $\omega$
            %  either contains a tail inside $\{\alpha\} \times \mathfrak{N}$ for $\alpha \in \kappa_1$
            %  or shares infinitely many vertices with $\{c\} \times \mathfrak{N}$.
            %Hence $\deg(\omega) = \kappa_1$.
            %Finally, $\Dom(\omega) = \kappa_2 \cup \{c\} \times \mathbb{N}$ and therefore $\dom(\omega) = \kappa_2$.
        \item[(4)] There is a graph with an end $\omega'$ such that $\deg(\omega') = k_1$ and $\dom(\omega') = k_2$.
        \item[(5)] There is a graph with an end $\omega'$ such that $\deg(\omega') = \aleph_0$ and $\dom(\omega') = k_2$.
    \end{itemize}
\end{remarks}

\subsection{Constructing uniformly connected rays}
\label{subsec:T-conn-rays}
\ \newline \indent
Let~${\omega \in \Omega(G)}$ be an end of~$G$ and 
let~$I, J$ be disjoint finite sets 
with~${1 \leq \abs{I} \leq \deg(\omega)}$
and ${0 \leq \abs{J} \leq \dom(\omega)}$.
Let~${\mathcal{R} = (R_i\ |\ i \in I)}$ be a family of disjoint $\omega$-rays
and let \linebreak
${\mathcal{D} = (d_j \in \Dom(\omega)\ |\ j \in J)}$ be a family of distinct vertices disjoint from $\bigcup \mathcal{R}$.
Let~$T$ be a tree on~$I \cup J$ such that~$J$ is a set of leaves of~$T$.
Let~${W := \bigcup\mathcal{R} \cup \mathcal{D}}$ 
and~${k := \abs{I \cup J}}$.
We call a finite subgraph $\Gamma \subseteq G$ a \emph{$(T,\mathcal{T}_2)$-connection}, if 
if $\mathcal{T}_2$ is a simple type-2 $k$-template for~${(T, J)}$, and
there is a set~$\mathcal{P}$ of internally disjoint $W$\,--\,$W$ paths 
such that ${\Gamma \subseteq \bigcup \mathcal{R} \cup \bigcup \mathcal{P}}$ and~$\Gamma$ is isomorphic to a subdivision of $T(\mathcal{T}_2)$.
Moreover, the subdivision of ${v^i_\bot P_i v^i_\top}$ is the segment ${R_i \cap \Gamma}$ for all~${i \in I}$ such that~$v^i_\top$ corresponds to the top vertex of that segment. 
Then~${(\mathcal{R},\mathcal{D})}$ is called \emph{$(T,\mathcal{T}_2)$-connected} 
if for every finite~${X \subseteq V(G) \setminus \mathcal{D}}$ 
there is a $(T,\mathcal{T}_2)$-connection avoiding~$X$.

\begin{lemma}\label{lem:T-connected}
    Let~${\omega \in \Omega(G)}$, let ${\mathcal{R} = (R_i\ |\ i \in I)}$ be a finite family of disjoint $\omega$-rays with $\abs{I} \geq 1$ 
    and let $\mathcal{D} = (d_j \in \Dom(\omega)\ |\ j \in J)$ be a finite family of distinct vertices disjoint from~${\bigcup \mathcal{R}}$ 
    with~${I \cap J = \emptyset}$. 
    Then there is a tree~$T$ on $I \cup J$ and a simple type-2 $\abs{I \cup J}$-template $\mathcal{T}$ for $(T,J)$ such that~${(\mathcal{R},\mathcal{D})}$ is $(T,\mathcal{T})$-connected.
\end{lemma}

\begin{proof}
    Let~$X \subseteq V(G) \setminus \mathcal{D}$ be any finite set.
    We extend~$X$ to a finite superset~$X'$ such that~${R_i \cap X'}$ is an initial segment of~$R_i$ for each~${i \in I}$, and such that~${\mathcal{D} \subseteq X'}$.
    As all rays in~$\mathcal{R}$ are $\omega$-rays, we can find finitely many $\bigcup\mathcal{R}$\,--\,$\bigcup\mathcal{R}$ paths avoiding~$X'$ which are internally disjoint such that their union with~$\bigcup\mathcal{R}$ is a connected subgraph of~$G$.
    Moreover it is possible to do this with a set $\mathcal{P}$ of~${\abs{I} - 1}$ many such paths in a tree-like way, 
    i.e.~contracting a large enough finite segment avoiding $X'$ of each ray in~$\mathcal{R}$ and deleting the rest yields a subdivision~$\Gamma_X'$ of a tree on $I$ whose edges correspond to the paths in $\mathcal{P}$.
    For each vertex~$d_j$ we can moreover find a $d_j$\,--\,$\bigcup \mathcal{R}$ path avoiding~${V(\Gamma_X') \cup X' \setminus \{d_j\}}$ and all paths we fixed so far.
    This yields a tree~$T_X$ on~${I \cup J}$ and a simple type-2 $k$-template~$\mathcal{T}_X$ for~$(T_X,J)$ such that~$J$ is a set of leaves and a~${(T_X, \mathcal{T}_X)}$-connection~$\Gamma_{X}$ avoiding~$X$.
    
    Now we iteratively apply this construction to find a family ${(\Gamma_i\ |\ i \in \mathbb{N})}$ of $(T_i,\mathcal{T}_i)$-connections such that~${\Gamma_m - \mathcal{D}}$ and~${\Gamma_n - \mathcal{D}}$ are disjoint for all~${m, n \in \mathbb{N}}$ with~${m \neq n}$.
    By the pigeonhole principle we now find a tree~$T$ on~${I \cup J}$, a type-2 $\abs{I \cup J}$-template~$\mathcal{T}$ and an infinite subset ${N \subseteq \mathbb{N}}$ such that ${(T_n, \mathcal{T}_n) = (T, \mathcal{T})}$ for all~${n \in N}$.
    
    Now for each finite set~${X \subseteq V(G) \setminus \mathcal{D}}$ there is an~${n \in N}$ such that~$\Gamma_n$ and~$X$ are disjoint, hence~${(\mathcal{R},\mathcal{D})}$ is $(T,\mathcal{T})$-connected.
\end{proof}

\begin{corollary}\label{cor:BP-sub}
    Let~${\omega \in \Omega(G)}$, let~${\mathcal{R} = (R_i\ |\ i \in I)}$ be a finite family of disjoint $\omega$-rays with $\abs{I} \geq 1$ 
    and let~$\mathcal{D} = (d_j \in \Dom(\omega) \ |\ j \in J)$ be a finite family of distinct vertices disjoint from~${\bigcup \mathcal{R}}$ 
    with~${I \cap J = \emptyset}$.
    Then there is a tree~$T$ such that~$G$ contains a subdivision of a generalised $\mathfrak{N}(T/J)$.
\end{corollary}

\begin{proof}
    By Lemma~\ref{lem:T-connected} there is a tree~$T$ and a simple type-2 ${\abs{I \cup J}}$-template~$\mathcal{T}$ such that $(\mathcal{R},\mathcal{D})$ is $(T, \mathcal{T})$-connected. 
    
    For a~${(T, \mathcal{T})}$-connection~$\Gamma$ let~$r^i_\Gamma$ denote the top vertex of~${R_i \cap \Gamma}$. 
    Then by~$\overline{\Gamma}$ we denote the union of~$\Gamma$ with the initial segments~${R_i r^i_\Gamma}$ for all~${i \in I}$. 
    
    Let ${(\Gamma_i\ |\ i \in \mathbb{N})}$ be a family of $(T,\mathcal{T})$-connections such that for all~${m < n}$ the graphs~${\Gamma_n}$ and~${\overline{\Gamma}_m}$ are disjoint but for the vertices in~$\mathcal{D}$. 
    Then ${H = \bigcup \{ \Gamma_n\ |\ n \in \mathbb{N} \} \cup \bigcup \mathcal{R}}$ is the desired subdivision. 
\end{proof}

Finally, this result can be lifted to the minor setting by Lemma~\ref{lem:gen-sub-fbs}.

\begin{corollary}\label{cor:BP-minor}
    Let~${\omega \in \Omega(G)}$, let~${\mathcal{R} = (R_i\ |\ i \in I)}$ be a finite family of disjoint $\omega$-rays with $\abs{I} \geq 1$ 
    and let~$\mathcal{D} = (d_j \in \Dom(\omega) \ |\ j \in J)$ be a finite family of distinct vertices disjoint from~${\bigcup \mathcal{R}}$ 
    with~${I \cap J = \emptyset}$.
    Then there is a tree~$T$ such that~$G$ contains $\mathfrak{N}(T/J)$ as an fbs-minor.\qed
\end{corollary}

\section{Minors for regular cardinalities}
\label{sec:regular-minors}

This section is dedicated to prove the equivalence of~\ref{item:t3-set},~\ref{item:t3-minor} and~\ref{item:t3-subdivision} of Theorem~\ref{main-thm} for regular cardinals~$\kappa$.

\subsection{Complete bipartite minors}
\label{subsec:r-K_k,kappa}
\ \newline \indent
In this subsection we construct the complete bipartite graph~$K_{k, \kappa}$ as the desired minor (and a generalised version as the desired subdivision), if possible.
The ideas of this construction differ significantly from Halin's construction \cite{Halin:simplicial-decomb}*{Thm.~9.1} of a subdivision of~$K_{k,\kappa}$ in a $k$-connected graph of uncountable and regular order~$\kappa$.

\begin{lemma}\label{lem:K_k,kappa}
    Let~${k \in \mathbb{N}}$, let~${A \subseteq V(G)}$ be infinite and $k$-connected in~$G$
    and let ${\kappa \leq \abs{A}}$ be a regular cardinal.
    If 
    \begin{itemize}
        \item either~$\kappa$ is uncountable; 
        \item or there is no end in the closure of~$A$; 
        \item or  there is an end $\omega$ in the closure of~$A$ with $\dom(\omega) \geq k$;
    \end{itemize}
    then there is a subset~${A' \subseteq A}$ with~${\abs{A'} = \kappa}$ 
    such that~$K_{k, \kappa}$ is an fbs-minor of~$G$ with~$A'$ along its core.
    
    Moreover, the branch sets for the vertices of the finite side of $K_{k,\kappa}$ are singletons.
\end{lemma}

\begin{proof}
    We iteratively construct a sequence of subgraphs~$H_i$ for~${i \in [0,k]}$ witnessing that~${K_{i, \kappa}}$ is a minor of~$G$.
    Furthermore, we incorporate that the branch sets for the vertices of the finite side of ${K_{i, \kappa}}$ are singletons ${\{v_j\ |\ j \in [0,i) \}}$ and the branch sets for the vertices of the infinite side induce finite trees on $H_i$ each containing a vertex of~$A$.
    Moreover, for~${i < k-1}$ we will guarantee the existence of a subset ${A_i \subseteq A}$ with ${\abs{A_i} = \abs{A}}$ which is ${1}$-connected in~${G_i := G - \{v_j\ |\ j \in [0,i) \}}$ and such that each vertex of~$A_i$ is contained in a branch set of~$H_i$ and each branch set of~$H_i$ contains precisely one vertex of~$A_i$.
    
    Set ${G_0 := G}$, ${A_0 := A}$ and~${H_0 = G[A]}$.
    For any~${i \in [0, k)}$ we inductively apply Lemma~\ref{star-comb} 
    (and in the third case also Remark~\ref{rem:star-from-comb}) 
    to~$A_i$ in~$G_i$ to find a subdivided star~$S_i$ with centre~$v_i$ and~$\kappa$ many leaves~${L_i \subseteq A_i}$.
    Without loss of generality we can assume~${v_i \notin V(H_i)}$, since otherwise we could just remove the branch set containing~$v_i$ and from~$A_i$ the vertex contained in that branch set.
    If ${i < k-1}$, then by Lemma~\ref{lem:delete<k} we find a subset~${L_i' \subseteq L_i}$ with ${\abs{L_i'} = \kappa}$ which is $1$-connected in~$G_{i+1}$. 
    In the case that ${i = k-1}$ let $L_{k-1}' = L_{k-1}$. 
    
    Again for any~${i \in [0,k)}$, we first remove from~$H_i$ every branch set 
    which corresponds to a vertex of the infinite side of~$K_{i,\kappa}$ and does not contain a vertex of~$L_i'$.
    Now each path in~$S_i$ from a neighbour of~$v_i$ to~$L_{i}'$ eventually hits a vertex of one of the finite trees induced by one of the remaining branch sets of~$H_i$.
    Since all these paths are disjoint, only finitely many of them meet the same branch set first.
    Thus~$\kappa$ many different of the remaining branch sets are met by those paths first.
    To get~$H_{i+1}$ we do the following.
    First we add~${\{v_i\}}$ as a new branch set.
    Then each of the~$\kappa$ many branch sets reached first as described above we extend by the path segment between~$v_i$ and that branch set of precisely one of those paths.
    Finally, we delete all remaining branch sets not connected to~${\{v_i\}}$.
    With~${A_{i+1} := L_i' \cap V(H_{i+1})}$ we now have all the desired properties.

    Finally, setting ${H := H_k}$ and~${A' := A_k}$ finishes the construction. 
\end{proof}

Let~$H$ be an inflated subgraph witnessing that~${K_{k,\kappa} = K([0,k), Z)}$ is an fbs-minor of~$G$ with~$A$ along~$Z$ for some~${A \subseteq V(G)}$ where each branch set of~${x \in [0,k)}$ is a singleton. 
Given a type-1 $k$-template~${\mathcal{T}_1 = (T, \gamma, c)}$ we say~$H$ is \emph{$\mathcal{T}_1$-regular} if for each~${z \in Z}$: 
\begin{itemize}
    \item there is an isomorphism ${\varphi_z: T'_z \to T_z}$  between a subdivision~$T'_z$ of~$T$ and the finite tree~${T_z = H[\mathfrak{B}(z)]}$;
    \item $x \varphi_z(\gamma(x)) \in E(H)$ for each ${x \in [0,k)}$; and
    \item $A \cap \mathfrak{B}(z) = \{ \varphi_z(c) \}$.
\end{itemize}
We say~$G$ contains~$K_{k,\kappa}$ as a \emph{$\mathcal{T}_1$-regular fbs-minor with~$A$ along~$Z$} if there is such a $\mathcal{T}_1$-regular~$H$.

\begin{lemma}\label{lem:gen-K_k,kappa}
    Let $k \in \mathbb{N}$ and $\kappa$ be a regular cardinal.
    If~${K_{k,\kappa}}$ is an fbs-minor of $G$ with~$A'$ along its core 
    where each branch set of ${x \in [0,k)}$ is a singleton, 
    then there is type-1 $k$-template~$\mathcal{T}_1$ 
    and~${A'' \subseteq A'}$ with~${\abs{A''} = \kappa}$ 
    such that~$G$ contains~$K_{k, \kappa}$ as a $\mathcal{T}_1$-regular fbs-minor with~$A''$ along its core.
\end{lemma}

\begin{proof}
    Let $H$ be the inflated subgraph witnessing that $K_{k,\kappa}$ is an fbs-minor as in the statement.
    Let~$x$ also denote the vertex of~$G$ in the branch set~$\mathfrak{B}(x)$ of~${x \in [0,k)}$.
    Let $v^z_x \in \mathfrak{B}(z)$ denote the unique endvertex in $\mathfrak{B}(z)$ of the edge corresponding to $xz \in E(M)$ (cf.~Section~\ref{sec:prelims}).
    Let~$T_z$ denote a subtree of~${H[\mathfrak{B}(z)]}$ containing~${B_z = \{v^z_x\ |\ x \in [0,k)\} \cup \{a_z\}}$ for the unique vertex~${a_z \in A \cap \mathfrak{B}(z)}$.
    Without loss of generality assume that each leaf of~$T_z$ is in~$B_z$.
    By suppressing each degree~$2$ node of~${T_z}$ that is not in~$B_z$, we obtain a tree suitable for a type-1 $k$-template where~$a_z$ is the node in the third component of the template.
    
    By applying the pigeonhole principle multiple times there is a tree~$T$ 
    such that there exist an isomorphism ${\varphi_z: T'_z \to T_z}$ for a subdivision~$T'_z$ of~$T$ 
    for all~${z \in Z'}$ for some~${Z' \subseteq Z}$ with~${\abs{Z'} = \kappa}$, 
    such that ${\{ \varphi_z(v^z_x)\ |\ z \in Z' \}}$ is a singleton~${\{ t_x \}}$ for all $x \in [0,k)$ as well as ${ \{ \varphi_z(a_z)\ |\ z \in Z'\} }$ is a singleton~${\{ c \}}$.
    
    Therefore with $\gamma: [0,k) \to V(T)$ defined by $x\mapsto t_x$ and $c$ defined as above, we obtain a type-1 $k$-template ${\mathcal{T}_1 := (T, \gamma, c)}$ such that the subgraph~$H'$ of~$H$ where we delete each branch set for $z \in Z \setminus Z'$ is $\mathcal{T}_1$-regular. 
\end{proof}

Hence, we also obtain a subdivision of a generalised $K_{k,\kappa}$.

\begin{corollary}\label{cor:gen-K_k,kappa}
    In the situation of Lemma~\ref{lem:K_k,kappa}, there is~${A'' \subseteq A'}$ with~${\abs{A''} = \kappa}$ such that~$G$ contains a subdivision of a generalised~$K_{k,\kappa}$ with core~$A''$.
    \qed
\end{corollary}

\subsection{Minors for regular \emph{k}-blueprints}
\label{subsec:r-BP}
\ \newline \indent
In this subsection we construct the $k$-typical minors for regular $k$-blueprints, if possible.
While these graphs are essentially the same minors given by Oporowski, Oxley and Thomas \cite{OOT:3&4-conn}*{Thm.~5.2}, we give our own independent proof based on the existence of an infinite $k$-connected set instead of the graph being essentially $k$-connected.

The first lemma constructs such a graph along some end of high combined degree.
\begin{lemma}\label{lem:rBP-minor}
    Let~${\omega \in \Omega(G)}$ be an end of~$G$ with~${\Delta(\omega) \geq k \in \mathbb{N}}$.
    Let~${A \subseteq V(G)}$ be a set with~$\omega$ in its closure.
    Then there is 
    a countable subset~$A' \subseteq A$ 
    and a regular 
    $k$-blueprint~$\mathcal{B}$
    such that 
    $G$ contains a subdivision of a generalised~$T_k(\mathcal{B})$ with core~$A'$.
\end{lemma}

\begin{proof}
    Let~${I, J}$ be disjoint sets with ${\abs{ I \cup J } = k}$, ${1 \leq \abs{I} \leq \deg(\omega)}$~and ${\abs{J} \leq \dom(\omega)}$.
    Let~${\mathcal{R} = ( R_i\ |\ i \in I)}$ be a family of disjoint $\omega$-rays and~${\mathcal{D} = ( d_j \in \Dom(\omega)\ |\ j \in J)}$ be a family of distinct vertices disjoint from $\bigcup \mathcal{R}$. 
    Applying Lemma~\ref{lem:T-connected} yields a tree~$B$ on~${I \cup J}$ and a type-2 $k$-template~$\mathcal{T}$ for~${(B,J)}$ such that ${(\mathcal{R}, \mathcal{D})}$ is ${(B,\mathcal{T})}$-connected.
    Let ${(\Gamma_i\ |\ i \in \mathbb{N})}$ denote the family of ${(B,\mathcal{T})}$-connections as in the proof of Lemma~\ref{lem:T-connected}.
    Moreover, there is an infinite set of disjoint~$A$\,--\,${\bigcup \mathcal{R}}$ paths by Theorem~\ref{cardinality-menger} since~$\omega$ is in the closure of~$A$.
    Now any infinite set of disjoint~$A$\,--\,$\bigcup \mathcal{R}$ paths has infinitely many endvertices on one ray~$R_c$ for some~${c \in I}$.
    Let~$A''$ denote the endvertices in~$A$ of such an infinite path system.
    Next we extend for infinitely many~$\Gamma_i$ the segment of~$R_c$ that it contains so that it has the endvertex of such an $A''$\,--\,$R_c$ path as its top vertex and add that segment together with the path to $\Gamma_i$, while keeping them disjoint but for~$\mathcal{D}$.
    Let~$A'$ denote the set of those endvertices of the paths in~$A''$ we used to extend $\Gamma_i$ for those infinitely many $i \in \mathbb{N}$.
    Finally, we modify the type-2 $k$-template accordingly.
    We obtain the subdivision of the generalised $T_k(B,J,c)$ as in the proof of Corollary~\ref{cor:BP-sub}.
\end{proof}

The following lemma allows us to apply Lemma~\ref{lem:rBP-minor} when Lemma~\ref{lem:K_k,kappa} is not applicable.

\begin{lemma}\label{lem:Delta>=k}
    Let $k \in \mathbb{N}$, let $A \subseteq V(G)$ be infinite and $k$-connected in $G$ and let~${\omega \in \Omega(G)}$ be an end in the closure of $A$. 
    Then $\Delta(\omega) \geq k$.
\end{lemma}

\begin{proof}
    We may assume that~$\Delta(\omega)$ is finite since otherwise there is nothing to show. 
    Hence without loss of generality~$A$ does not contain any vertices dominating~$\omega$.
    Let ${(S_n\ |\ n \in \mathbb{N})}$ be an $\omega$-defining sequence, which exists by Lemma~\ref{lem:def-seq}.
    Take ${N \in \mathbb{N}}$ such that there is a set ${B \subseteq A \setminus C(S_N, \omega)}$ of size~$k$.
    For every ${n > N}$ let ${C_n \subseteq A \cap C(S_n, \omega)}$ be a set of size~$k$, which exists since~$\omega$ is in the closure of~$A$.
    Since~$A$ is $k$-connected in~$G$, there are~$k$ disjoint $B$\,--\,$C_n$ paths in~$G$, each of which contains at least one vertex of~$S_n$.
    Hence for all ${n > N}$ we have ${\abs{S_n} \geq k}$ and by Corollary~\ref{cor:end-deg2} we have~${\Delta(\omega) \geq k}$. 
\end{proof}

We close this subsection with a corollary that is not needed in this paper, but provides a converse for Lemma~\ref{lem:Delta>=k} as an interesting observation.

\begin{corollary}\label{cor:Delta>=k_k-conn}
    Let ${\omega \in \Omega(G)}$ be an end of $G$ with~${\Delta(\omega) \geq k \in \mathbb{N}}$.
    Then every subset ${A \subseteq V(G)}$ with~$\omega$ in the closure of~$A$ contains a countable subset ${A' \subseteq A}$ which is $k$-connected in~$G$.
\end{corollary}

\begin{proof}
    By Lemma~\ref{lem:rBP-minor} we obtain a subdivision of a generalised $T_k(\mathcal{B})$ with core ${A'}$ for some ${A' \subseteq A}$ in~$G$ for a regular $k$-blueprint~$\mathcal{B}$.
    Corollary~\ref{cor:gen-typ-k-conn} yields the claim.
\end{proof}

\subsection{Characterisation for regular cardinals}\label{subsec:regular-thm}
\ \newline \indent
Now we have developed all the necessary tools to prove the minor and topological minor part of the characterisation in Theorem~\ref{main-thm} for regular cardinals.

\begin{theorem}\label{thm:main-regular}
    Let~$G$ be a graph, let~${k \in \mathbb{N}}$, 
    let~${A \subseteq V(G)}$ be infinite 
    and let ${\kappa \leq \abs{A}}$ be a regular cardinal.
    Then the following statements are equivalent.
    \begin{enumerate}[label=(\alph*)]
        \item \label{item:t6-set}
            There is a subset~${A_1 \subseteq A}$ with~${\abs{A_1} = \kappa}$ such that~$A_1$ is $k$-connected in~$G$.
        \item \label{item:t6-minor}
            There is a subset~${A_2 \subseteq A}$ with~${\abs{A_2} = \kappa}$ such that
            \begin{itemize}
                \item[$\bullet$] either~$K_{k, \kappa}$ is an fbs-minor of~$G$ 
                    with~$A_2$ along its core; 
                \item[$\bullet$] or~$T_k(\mathcal{B})$ is an fbs-minor of~$G$  
                    with~$A_2$ along its core 
                    for some regular 
                    $k$-blueprint~$\mathcal{B}$.
            \end{itemize}
        \item \label{item:t6-subdivision}
            There is a subset~${A_3 \subseteq A}$ with~${\abs{A_3} = \kappa}$ such that
            \begin{itemize}
                \item[$\bullet$] either $G$ contains a subdivision of a generalised~$K_{k, \kappa}$ 
                    with core~$A_3$; 
                \item[$\bullet$] or $G$ contains the subdivision of a generalised~$T_k(\mathcal{B})$
                    with core~$A_3$ 
                    for some regular 
                    $k$-blueprint~$\mathcal{B}$.
            \end{itemize}
    \end{enumerate}
    Moreover, if these statements hold, we can choose~${A_1 = A_2 = A_3}$.\end{theorem}

\begin{proof}
    If~\ref{item:t6-minor} holds, 
    then~$A_2$ is 
    $k$-connected in~$G$ by 
    Lemma~\ref{lem:k-con-minor} with 
    Lemma~\ref{lem:typical-k-conn}.
    
    If~\ref{item:t6-set} holds, 
    then we can find a subset ${A_3 \subseteq A_1}$ with ${\abs{A_3} = \kappa}$ yielding~\ref{item:t6-subdivision} by either
    Lemma~\ref{lem:K_k,kappa} and Corollary~\ref{cor:gen-K_k,kappa} 
    or by Lemma~\ref{lem:Delta>=k} and Lemma~\ref{lem:rBP-minor}.
    
    If~\ref{item:t6-subdivision} holds, then so does~\ref{item:t6-minor} by Lemma~\ref{lem:gen-sub-fbs} with~${A_2 := A_3}$.
    Moreover, $A_3$ is a candidate for both~$A_2$ and~$A_1$.
\end{proof}

\section{Minors for singular cardinalities}
\label{sec:singular-minors}

In this section we will prove the equivalence of~\ref{item:t3-set},~\ref{item:t3-minor} and~\ref{item:t3-subdivision} of Theorem~\ref{main-thm} for singular cardinals~$\kappa$.

\subsection{Cofinal sequence of regular bipartite minors with disjoint cores.}
\label{subsec:s-prep}
\ \newline \indent
In this subsection, given a $k$-connected set~$A$ of size $\kappa$, we will construct an ${\lK{\mathcal{K}}}$ minor in~$G$ for some suitable~${\ell \in [0,k]}$ and good $\kappa$-sequence~$\mathcal{K}$ with a suitable subset of $A$ along its precore.
This minor is needed as an ingredient for any of the possible $k$-typical graphs but the $K_{k, \kappa}$ (which we obtain from the following lemma if~${\ell = k}$).
Let ${\mathcal{A} = ( A^\alpha \subseteq A\ |\ \alpha \in \cf\kappa )}$ be a family of disjoint subsets of~$A$.
We say that~$G$ contains~${\lK{\mathcal{K}}}$ as an fbs-minor with~$\mathcal{A}$ along its precore~$\mathcal{Z}$ if the map mapping each vertex of the inflated subgraph to its branch set induces a bijection between~$A^\alpha$ and~$Z^\alpha$ for all~${\alpha \in \cf\kappa}$.

\begin{lemma}\label{lem:cofinal-minors}
    Let~${k \in \mathbb{N}}$, let~${A \subseteq V(G)}$ be infinite and $k$-connected in~$G$ and let ${\kappa \leq \abs{A}}$ be a singular cardinal.
    Then there is an $\ell \in [0,k]$, 
    a good $\kappa$-sequence ${\mathcal{K} = (\kappa_\alpha < \kappa\ |\ \alpha \in \cf\kappa)}$, 
    and a family~${\mathcal{A} = (A^\alpha \subseteq A\ |\ {\alpha \in \cf\kappa} )}$ of pairwise disjoint subsets of~$A$ with~${\abs{A^\alpha} = \kappa_\alpha}$ 
    such that~$G$ contains 
    ${\lK{\mathcal{K}}}$ as an fbs-minor 
    with~$\mathcal{A}$ along~$\mathcal{Z}$.
    Moreover, the branch sets for the vertices in $\bigcup \mathcal{Y}$ are singletons.
\end{lemma}

\begin{proof}
    We start with any good $\kappa$-sequence~${\mathcal{K} = (\kappa_\alpha < \kappa\ |\ \alpha \in \cf\kappa)}$.
    We construct the desired inflated subgraph by iteratively applying Lemma~\ref{lem:K_k,kappa}.
    
    For ${\alpha \in \cf\kappa}$ suppose we have already constructed for each~${\beta < \alpha}$ an inflated subgraph~$H^\beta$ witnessing that~$K_{k, \kappa_\beta}$ is an fbs-minor of~$G$ with some~${A^\beta \subseteq A}$ along its core.
    Furthermore, suppose that the branch sets of the vertices of the finite side are singletons and the branch sets of the vertices of the infinite side are disjoint to all branch sets of~$H^\gamma$ for all~${\gamma < \beta}$.
    We apply Lemma~\ref{lem:K_k,kappa} for $\kappa_\alpha$ to any set $A' \subseteq A \setminus \bigcup_{\beta < \alpha} A^\beta$ of size~$\kappa_\alpha$ 
    to obtain an inflated subgraph for $K_{k,\kappa_\alpha}$ with the properties as stated in that lemma.
    If any branch set for a vertex of the infinite side meets any branch set we have constructed so far, we delete it.
    Since $\kappa_\alpha$ is regular and $\kappa_\alpha > \cf\kappa$, the union of all inflated subgraphs we constructed so far has order less than~$\kappa_\alpha$.
    We obtain that the new inflated subgraph (after the deletions) still witnesses that~$K_{k,\kappa_\alpha}$ is an fbs-minor of~$G$ with some~${A^\alpha \subseteq A'}$ along its core.
    If a branch set for the finite side meets any branch set of a vertex for the infinite side for some~${\beta < \alpha}$, we delete that branch set and modify~$A^\beta$ accordingly.
    As the union of all branch sets for the finite side we will construct in this process has cardinality~$\cf\kappa$, each~$A^\beta$ will lose at most~${\cf\kappa < \kappa_\beta}$ many elements, hence will remain at size $\kappa_\beta$ for all $\beta \in \cf\kappa$.
    We denote the sequence~${( A^\alpha\ |\ \alpha \in \cf\kappa )}$ with~$\mathcal{A}$.
    
    By Lemma~\ref{delta-system-lemma} 
    there is an~${\ell \leq k}$ and an~${I \subseteq \cf\kappa}$ with~${\abs{I} = \cf\kappa}$ such that $H^\alpha$ and $H^\beta$ have precisely $\ell$ branch sets for the vertices of the finite side in common for all~${\alpha, \beta \in I}$.
    Hence relabelling the subsequences~$\mathcal{K}\restricted I$ and~${\mathcal{A}\restricted I}$ to~$\overline{\mathcal{K}\restricted I}$ and~$\overline{\mathcal{A}\restricted I}$ respectively as discussed in Section~\ref{sec:typical} yields the claim, where the union of the respective subgraphs $H^\alpha$ is the witnessing inflated subgraph.
\end{proof}

\subsection{Frayed complete bipartite minors}
\label{subsec:s-lFK} 
\ \newline \indent
In this subsection we will construct a frayed complete bipartite minor, if possible.
We shall use an increasing amount of fixed notation in this subsection based on Lemma~\ref{lem:cofinal-minors}, which we will fix as we continue our construction.

\begin{situation}\label{situation-lFK}
    Let~${k \in \mathbb{N}}$, let~${A \subseteq V(G)}$ be infinite and $k$-connected in~$G$ and let ${\kappa \leq \abs{A}}$ be a singular cardinal.
    Let~${\ell \leq k}$ and let
    \begin{itemize}
        \item ${\mathcal{K} = (\kappa_\alpha < \kappa\ |\ \alpha \in \cf\kappa)}$ be a good $\kappa$-sequence; and
        \item ${\mathcal{A} = (A^\alpha\ |\ {\alpha \in \cf\kappa})}$ be a family of pairwise disjoint subsets of~$A$ with ${\abs{A^\alpha} = \kappa_\alpha}$.
    \end{itemize}
    Let~$H$ be an inflated subgraph witnessing that~$G$ contains~${\lK{\mathcal{K}}}$ as an fbs-minor with~$\mathcal{A}$ along~$\mathcal{Z}$ as in Lemma~\ref{lem:cofinal-minors}.
    To simplify our notation, we 
    denote the unique vertex of~$H$ in a branch set of~$y_i^\alpha$ also by~$y_i^\alpha$ for all~${\alpha \in \cf\kappa}$ and~${i < k}$.
    Similarly, we denote the set ${\{ y_i^\alpha \in V(H)\ |\ i \in [0, k) \}}$ also with $Y^\alpha$ for all~${\alpha \in \cf\kappa}$, 
    and denote the family~${( Y^\alpha \subseteq V(H)\ |\ \alpha \in \cf\kappa )}$ with~$\mathcal{Y}$.
    Moreover, let~$H^\alpha$ denote the subgraph of~$H$ witnessing that~${K(Y^\alpha,Z^\alpha)}$ is an fbs-minor of~$G$ with~$A^\alpha$ along~$Z^\alpha$.
    Finally, let $D_\ell = \{ y_i\ |\ i \in [0,\ell) \} = \bigcap \{ V(H^\alpha)\ |\ \alpha \in \cf\kappa \}$ denote the set of degenerate vertices of $\lK{\mathcal{K}}$.
\end{situation}

For a set~${U \subseteq V(G)}$ and $\alpha \in \cf\kappa$, a \emph{$Y^\alpha$\,--\,$U$ bundle}~$P^\alpha$ is the union ${\bigcup \{ P_i^\alpha\ |\ i \in [0,k) \}}$ of~$k$ disjoint paths, 
where~$P_i^\alpha \subseteq G$ is a (possibly trivial) $Y^\alpha$\,--\,$U$ path starting in~${y_i^\alpha \in Y^\alpha}$ and ending in some~${u_i^\alpha \in U}$.
A family $\mathcal{P} = ( P^\alpha\ |\ \alpha \in \cf\kappa )$ of $Y^\alpha$\,--\,$U$ bundles is a \emph{$\mathcal{Y}$\,--\,$U$ bundle} if $P^\alpha - U$ and $P^\beta - U$ are disjoint for all $\alpha, \beta \in \cf\kappa$ with $\alpha \neq \beta$.
Note that if a $\mathcal{Y}$\,--\,$U$ bundle exists, then $U$ contains $D_\ell$.

A set~${U \subseteq V(G)}$ \emph{distinguishes}~$\mathcal{Y}$  
if whenever~${y_i^\alpha}$ and~${y_j^\beta}$ are in the same component of~${G - U}$ for $\alpha, \beta \in \cf\kappa$ and ${i, j \in [0, k)}$, then ${\alpha = \beta}$.

\begin{lemma}\label{lem:distinguish}
    If a set $U \subseteq V(G)$ distinguishes~$\mathcal{Y}$, then there is a $\mathcal{Y}$\,--\,$U$ bundle $\mathcal{P}$.
\end{lemma}

\begin{proof}
    Let~${U \subseteq V(G)}$ distinguish~$\mathcal{Y}$.
    By definition every finite set separating~$Y^\alpha$ from~$Y^\beta$ in~$G$ also has to separate~$A^\alpha$ from~$A^\beta$.
    Since $A$ is $k$-connected in $G$, there are also $k$ disjoint $Y^\alpha$\,--\,$Y^\beta$ paths in~$G$ by Theorem~\ref{finite-parameter-menger}.
    Hence we fix the initial $Y^\alpha$\,--\,$U$ segments of these paths for each~${\alpha \in \cf\kappa}$, which are disjoint outside of~$U$ by the assumption that~$U$ distinguishes~$\mathcal{Y}$.
    This yields the desired $\mathcal{Y}$\,--\,$U$ bundle.
\end{proof}

For a cardinal~$\lambda$, a set~${W \subseteq V(G)}$ is $\lambda$-\emph{linked} to a set~${U \subseteq V(G)}$, if for every~${w \in W}$ and every~${u \in U}$ there are~$\lambda$ many internally disjoint $w$\,--\,$u$ paths in~$G$.

The following lemma is the main part of the construction.

\begin{lemma}\label{lem:lFK}
    In Situation~\ref{situation-lFK}, 
    suppose there is 
    a set ${U \subseteq V(G)}$ such that 
    \begin{itemize}
        \item there is a $\mathcal{Y}$\,--\,$U$ bundle $\mathcal{P} = ( P^\alpha\ |\ \alpha \in \cf\kappa )$; and
        \item there is a set~${W \subseteq U}$ with~${\abs{W} = k}$ 
            such that~$W$ is~$\cf\kappa$-linked to~$U$.
    \end{itemize}
        
    Then there is an~${I_0 \subseteq \cf\kappa}$ with~${\abs{I_0} = \cf\kappa}$ 
    and a family~${\mathcal{A}_0 = ( A_0^\alpha \subseteq A^\alpha\ |\ \alpha \in I_0 )}$
    with ${\abs{A_0^\alpha} = \kappa_\alpha}$ for all~${\alpha \in I_0}$  
    such that 
     $\lFK(\overline{\mathcal{K}\restricted{I_0}})$ is an fbs-minor of~$G$ with~$\overline{\mathcal{A}_0}$ along~$\overline{\mathcal{Z}\restricted{I_0}}$.
\end{lemma}

\begin{proof}
    Let $U$, $\mathcal{P}$ and $W$ be as above.
    By Lemma~\ref{delta-system-lemma} there is a~${j \in [0 , k]}$ and a subset~${I' \subseteq \cf\kappa}$ with ${\abs{I'} = \cf\kappa}$ such that (after possibly relabelling the sets~$Y^\alpha$ for all~${\alpha \in I'}$ simultaneously) 
     for every~${\alpha, \beta \in I'}$ with~${\alpha \neq \beta}$
    \begin{itemize}
        \item $y_i = u_i^\alpha = u_i^\beta$ for all $i \in [0, \ell)$;
        \item $x_i := u_i^\alpha = u_i^\beta$ 
            for all $i \in [\ell, \ell+j)$; and 
        \item $u_{i_0}^\alpha \neq u_{i_1}^\beta$ for all $i_0,i_1 \in [\ell+j, k)$.
    \end{itemize}
    Furthermore, after deleting at most $j$ more elements from $I'$ we obtain $I''$ such that 
    \begin{itemize}
        \item $u_i^\alpha \neq y_i^\alpha$ for all $i \in [\ell, \ell+j)$ and all $\alpha \in I''$.
    \end{itemize}
    Note that if~${\abs{U} < \cf\kappa}$, then~${\ell + j = k}$ and we set $I_0 := I''$ and $L := \emptyset$.
    
    Otherwise we construct subdivided stars with distinct centres in~$W$. 
    We start with a~${k-\ell-j}$ element subset~${W' = \{ w_i\ |\ i \in [\ell+j, k) \} \subseteq W}$ disjoint from both~$D_\ell$ as well as~${\{x_i\ |\ i \in [\ell, \ell+j)\}}$.
    A subgraph~$L$ of~$G$ is a \emph{partial star-link} if there is a set ${I(L) \subseteq I''}$ such that~$L$ is the disjoint union of subdivided stars~$S_i$ for all~${i \in [\ell + j, k)}$ with centre~${w_i}$ and leaves~$u_i^\alpha$, 
    and~$L$ is disjoint to ${{P}^\alpha - \{ u_i^\alpha\ |\ i \in [\ell+j, k) \}}$  for all~${\alpha \in I(L)}$.
    A partial star-link~$L$ is a \emph{star-link} if $\abs{I(L)} = \cf\kappa$.
    Note that the union of a chain of partial star-links (ordered by the subgraph relation) yields another partial star-link.
    Hence by Zorn's Lemma there is a maximal partial star-link $M$.
    Assume for a contradiction that $M$ is not a star-link.
    Then the set $N = V(M) \cup \bigcup_{\alpha \in I(M)} V({P}^\alpha)$ has size less than $\cf\kappa$.
    Take some $\beta \in I' \setminus I(M)$ such that $M$ is disjoint to ${P}^\beta$.
    Since $W$ is $\cf\kappa$-linked to $U$, we can find $k-\ell-j$ disjoint $W'$\,--\,$\{ u_i^\beta\ |\ i \in [\ell+j, k) \}$ paths disjoint from $N \setminus W'$, contradicting the maximality of~$M$ (after possibly relabelling).
    Hence there is a star-link~$L$, and we set $I_0 := I(L)$.
    
    Let $H_{I_0}$ denote the subgraph of $H$ containing only the branch sets for vertices in~${Y^\alpha \cup Z^\alpha}$ for~${\alpha \in I_0}$. 
    Since $L \cup \bigcup_{\alpha \in I_0} {P^\alpha}$ has size $\cf\kappa < \kappa_\alpha$ for all $\alpha \in I_0$, 
    we can remove every branch set for some $z \in Z^\alpha$ meeting $L \cup \bigcup_{\alpha \in I_0} {P^\alpha}$ and obtain $A_0^\alpha \subseteq A^\alpha$ with $\abs{A_0^\alpha} = \kappa_\alpha$. 
    The union of the resulting subgraph with $L$ and $\bigcup_{\alpha \in I_0} P^\alpha$ witnesses that 
    $\lFK(\overline{\mathcal{K} \restricted I_0})$ is an fbs-minor of $G$ with~${\overline{\mathcal{A}_0} := \overline{( A_0^\alpha\ |\ \alpha \in I_0)}}$ along~${\overline{\mathcal{Z}\restricted I_0}}$.
\end{proof}

As before, the previous lemma can be translated to find a desired subdivision of a generalised~$\lFK$.

\begin{lemma}\label{lem:gen-lFK}
    In the situation of Lemma~\ref{lem:lFK}, 
    there is an $I_1 \subseteq I_0$ with $\abs{I_1} = \cf\kappa$ 
    and a family ${\mathcal{A}_1 = (A_1^\alpha \subseteq A_0^\alpha\ |\ \alpha \in I_1)}$ with~${\abs{A_1^\alpha} = \kappa_\alpha}$ for all~${\alpha \in I_1}$ such that $G$ contains a subdivision of a generalised ${\lFK(\overline{\mathcal{K}\restricted I_1})}$ with core~$\bigcup \mathcal{A}_1$.
\end{lemma}

\begin{proof}
    Let $H$ be the inflated subgraph witnessing that ${\lFK(\overline{\mathcal{K}\restricted I_0})}$ is an fbs-minor of $G$ with $\overline{\mathcal{A}_0}$ along its core.
    Let $H^\alpha \subseteq H$ be the subgraph corresponding to the subgraph $K(Y^\alpha,Z^\alpha)$ of $\lFK(\overline{\mathcal{K}\restricted I_0})$ for each $\alpha \in I_0$.
    For each $\alpha \in I_0$ we apply Lemma~\ref{lem:gen-K_k,kappa} to $H^\alpha$.
    By the pigeonhole principle there is a set $I_1 \subseteq I_0$ with $\abs{I_1} = \cf\kappa$ such that the type-1 $k$-template we got is the same for each $\alpha \in I_1$.
    This yields the desired subdivision as for Corollary~\ref{cor:gen-K_k,kappa}.
\end{proof}

The remainder of this subsection is dedicated to identify when we can apply Lemma~\ref{lem:lFK}.

\begin{lemma}\label{lem:lFK-no-end}
    In Situation~\ref{situation-lFK}, 
    if either~$\cf\kappa$ is uncountable
    or there is no end in the closure of some transversal~$T$ of~$\mathcal{A}$, 
    then there is a set~${U \subseteq V(G)}$ 
    with the properties needed for Lemma~\ref{lem:lFK}.
\end{lemma}

\begin{proof}
    We start with a transversal~$T$ of~$\mathcal{A}$ (whose closure does not contain any end if~${\cf\kappa}$ is countable).
    We apply Lemma~\ref{lem:K_k,kappa} to~$T$ to obtain an inflated subgraph witnessing that~$K_{k,\cf\kappa}$ is an fbs-minor of~$G$ with~${T_0 \subseteq T}$ along its core.
    We call the union of the singleton branch sets for the vertices of the finite side~${W =: U_0}$.
    By construction~$W$ is~$\cf\kappa$-linked to~$U_0$.
    Let~${I_0}$ denote the set~${\{ \alpha \in \cf\kappa\ |\ \abs{T_0 \cap A^\alpha} = 1 \}}$.
    We construct~$U$ inductively.
    
    For some ordinal~$\alpha$ 
    we assume we already constructed a strictly $\subseteq$-ascending sequence ${(U_\beta\ |\ \beta < \alpha)}$ such that~$W$ is~$\cf\kappa$-linked to~$U_\beta$ for all~${\beta < \alpha}$.
    If there is a subset~$I \subseteq I_0$ with~${\abs{I} = \cf\kappa}$ such that~${U' := \bigcup_{\beta < \alpha} U_\alpha}$ 
    distinguishes ${\overline{\mathcal{Y}\restricted{I}}}$, 
    then we are done by Lemma~\ref{lem:distinguish} since by construction $W$ is still $\cf\kappa$-linked to $U'$.
    Otherwise there is a component of $G - U'$ containing a transversal~$T_\alpha$ of~${\mathcal{Y}\restricted{I_\alpha}}$ 
    for some~${I_\alpha \subseteq I_0}$ with~${\abs{I_\alpha} = \cf\kappa}$.
    Applying Lemma~\ref{star-comb} to~$T_\alpha$ yields a subdivided star with centre $u_\alpha$ and~$\cf\kappa$ many leaves~$L_\alpha \subseteq T_\alpha$.
    We then set~${U_\alpha := U' \cup \{ u_\alpha \}}$.
    By Theorem~\ref{cardinality-menger} there are~$\cf\kappa$ many internally disjoint $w$\,--\,$u_\alpha$ paths for all~${w \in W}$, since no set of size less than~$\cf\kappa$ could 
    separate~$u_\alpha$ from~$L_\alpha$, $L_\alpha$ from~$T_0$, or any subset of size $\cf\kappa$ of~$T_0$ from~$w$.
    Hence~$W$ is $\cf\kappa$-linked to~$U_\alpha$ and we can continue the construction.
    This construction terminates at the latest if $U' = V(G)$.
\end{proof}

If $\cf\kappa$ is countable and there is an end in the closure of some transversal of $\mathcal{A}$, then there is still a chance to obtain an $\lFK$ minor. 
We just need to check whether $G$ contains a $\mathcal{Y}$\,--\,$\Dom(\omega)$ bundle, since we have the following lemma.

\begin{lemma}\label{lem:Dom-linked}
    For every end $\omega \in \Omega(G)$, the set $\Dom(\omega)$ is $\aleph_0$-linked to itself.
\end{lemma}

\begin{proof}
    Suppose there are ${u, v \in \Dom(\omega)}$ with only finitely many internally disjoint $u$\,--\,$v$ paths. 
    Hence there is a finite separator~${S \subseteq V(G)}$ such that~$u$ and~$v$ are in different components of~${G - S}$.
    Then at least one of them is in a different component than $C(S,\omega)$, a contradiction.
\end{proof}

Hence, we obtain the final corollary of this subsection.

\begin{corollary}\label{cor:lFK-end}
    In Situation~\ref{situation-lFK}, 
    suppose~${\cf\kappa}$ is countable and 
    there is an end~$\omega$ in the closure of some transversal of~${\mathcal{A}}$ 
    with~${\dom{\omega} \geq k}$ 
    such that ${\Dom(\omega)}$ distinguishes~${\mathcal{Y}}$.
    Then~${\Dom(\omega)}$ satisfies the properties needed for Lemma~\ref{lem:lFK}.
    \qed
\end{corollary}

\subsection{Minors for singular \emph{k}-blueprints}
\label{subsec:sBP}
\ \newline \indent
This subsection builds differently upon Situation~\ref{situation-lFK} in the case where we do not obtain the frayed complete bipartite minor.
We incorporate new assumptions and notation, establishing a new situation, which we will further modify according to some assumptions that we can make without loss of generality during this subsection.

\begin{situation}\label{situation-sBP}
    Building upon Situation~\ref{situation-lFK}, 
    suppose~${\cf\kappa}$ is countable 
    and there is an end~$\omega$ in the closure of some transversal of~$\mathcal{A}$, i.e.~an $\omega$-comb whose teeth are a transversal~$T$ of $\{ A^i\ |\ i \in J \}$ for some infinite $J \subseteq \mathbb{N}$.
    Suppose that 
    \begin{itemize}
         \item[($\ast$)] there is no $\overline{\mathcal{Y}\restricted I}$\,--\,$\Dom(\omega)$ bundle for any infinite $I \subseteq \mathbb{N}$.
    \end{itemize}
    In particular~$\Dom(\omega)$ does not distinguish~$\overline{\mathcal{Y}\restricted{J}}$ 
    by Lemma~\ref{lem:distinguish}.
    Hence there is a component~$C$ of~${G - \Dom(\omega)}$ containing a comb with teeth in $\mathcal{Y}\restricted{J}$, since a subdivided star would yield a vertex dominating $\omega$ outside $\Dom(\omega)$.
    This comb is an $\omega$-comb since its teeth cannot be separated from~$T$ by a finite vertex set.
    Without loss of generality we may assume that $J = \mathbb{N}$ by redefining $\mathcal{K}$, $\mathcal{Y}$ and $\mathcal{A}$ as $\overline{\mathcal{K}\restricted{J}}$, $\overline{\mathcal{Y}\restricted{J}}$ and $\overline{\mathcal{A}\restricted{J}}$ respectively.
    
    Let ${G' := G[C]}$ and let $\omega'$ be the end of $G'$ containing the spine of the aforementioned $\omega$-comb in $G'$.
    Let~${\mathcal{S} = (S^n\ |\ n \in \mathbb{N})}$ be an 
    $\omega'$-defining sequence in~$G'$ 
    and let~$\mathcal{R}$ 
    be a family of disjoint $\omega'$-rays in~$G'$ such that 
    ${\big( \mathcal{S}, \mathcal{R} \big)}$ witnesses the degree of the undominated end~$\omega'$ of~$G'$, which exist by Lemma~\ref{lem:def-seq-deg-wit}.
    Moreover, we will modify this situation with some assumptions that we can make without loss of generality.
    We will fix them in some of the following lemmas and corollaries.
\end{situation}

\begin{lemma}\label{lem:sBP:A1}
    In Situation~\ref{situation-sBP}, 
    we may assume without loss of generality that 
    for all~${n \in \mathbb{N}}$ the following hold:
    \begin{itemize}
        \item ${S^n \cap \bigcup \mathcal{Y} = \emptyset}$; and 
        \item $S^n$ is contained in a component of $G'[S^n, S^{n+1}]$.
    \end{itemize}
    Hence we include these assumptions into Situation~\ref{situation-sBP}.
\end{lemma}

\begin{proof}
    Given $x,y \in \mathbb{N}$ we can choose $n \in \mathbb{N}$ with $n \geq y$ and $m \in \mathbb{N}$ with $m > n$ such that $S^n$ is contained in a component of~$G'[S^n, S^{m}]$ and $Y^x$ is disjoint to ${S^n \cup S^m}$.
    Note that it is possible to incorporate the first property since $(\mathcal{S}, \mathcal{R})$ is degree-witnessing in~$G'$.
    Iteratively applying this observation yields subsequences of~$\mathcal{S}$ and~$\mathcal{Y}$.
    Taking the respective subsequences of~$\mathcal{K}$ and~$\mathcal{A}$ and relabelling all of them accordingly as before yields the claim.
\end{proof}

\begin{lemma}\label{lem:sBP:A2}
    In Situation~\ref{situation-sBP}, 
    we may assume without loss of generality that \linebreak
    ${\emptyset \neq Y^n \setminus \Dom(\omega) \subseteq V(G'[S^n, S^{n+1}])}$ 
    for all $n \in \mathbb{N}$.
    Hence we include this assumption into Situation~\ref{situation-sBP}.
\end{lemma}

\begin{proof}
    Note that ${Y^n \setminus \Dom(\omega) = \emptyset}$ for only finitely many $n \in \mathbb{N}$ by ($\ast$).
    Moreover, for all but finitely many $n \in \mathbb{N}$ there is an $x_n \in \mathbb{N}$ such that ${Y^{x_n} \setminus \Dom(\omega)}$ meets ${V(G'[S^n, S^{n+1}])}$ since~$\omega$ is in the closure of~$\mathcal{Y}$.
    Suppose that for some~${n \in \mathbb{N}}$ with setting~$x := x_n$ we have that~${Y^{x} \setminus \Dom(\omega)}$ is not contained in~${V(G'[S^n, S^{n+1}])}$. 
    Since for any~${i,j \in [0,k)}$ with~${i \neq j}$ there are~$\kappa_x$ many disjoint $y_i^x$\,--\,$y_j^x$ paths in~$H^x$, all but finitely many of them have to traverse~$\Dom(\omega)$.
    In particular, there is an $Y^x$\,--\,$\Dom(\omega)$ bundle in $H^x$.
    Such a bundle trivially also exists if~${Y^x \subseteq \Dom(\omega)}$.
    If this happens for all $x$ in some infinite $I \subseteq \mathbb{N}$, then there is a $\overline{\mathcal{Y}\restricted{I}}$\,--\,$\Dom(\omega)$ bundle in~$G$, contradicting~($\ast$). 
    Hence this happens at most finitely often.
    Again, relabelling and taking subsequences yields the claim.
\end{proof}

The following lemma allows some control on how we can find a set of disjoint paths from~$Y^n$ to the rays in~$\mathcal{R}$  
and has two important corollaries.

\begin{lemma}\label{lem:sBP:bundleP}
    In Situation~\ref{situation-sBP}, 
    let $\mathcal{R}' \subseteq \mathcal{R}$ with $\abs{\mathcal{R}'} = \min(\deg(\omega'), k)$.
    Then for all~${n > 2k}$ there is an $M > n$ such that for all ${m \geq M}$  
    there exists an
    $Y^n$\,--\,\linebreak${(\Dom(\omega) \cup (S^m \cap \bigcup \mathcal{R}'))}$ bundle~$P^{n,m}$
    with ${P^{n,m} - \Dom(\omega) \subseteq G'[S^{n-2k}, S^m]}$.
\end{lemma}

\begin{proof}
    Let $n > 2k$ be fixed.
    As in the proof of Lemma~\ref{lem:distinguish} for each $x > 0$ there are~$k$ disjoint $Y^n$\,--\,$Y^{x-1}$ paths in~$G$, whose union contains a $Y^n$\,--\,${(\Dom(\omega) \cup S^{x})}$ bundle~$Q^x$ in~${G[C \cup \Dom(\omega)]}$.
    Considering~$Q^{n-2k}$, let~$M \in \mathbb{N}$ be large enough such that ${Q^{n-2k} - \Dom(\omega) \subseteq G'[S^{n-2k}, S^M]}$, and let~${m \geq M}$.
    
    Suppose for a contradiction that there is a vertex set~$S$ of size less than $k$ separating~$Y^n$ from ${\Dom(\omega) \cup (S^m \cap \bigcup \mathcal{R}')}$ in $G[ V(G'[S^{n-2k}, S^{m}]) \cup \Dom(\omega) ]$.
    Then for at least one~${i \in [n-2k, n)}$ the graph $G'[S^i, S^{i+1}]$ does not contain a vertex of~$S$.
    We distinguish two cases.
    
    Suppose~${\deg(\omega') \geq k}$.
    Then~$S$ contains a vertex from every path of~$Q^{n-2k}$ ending in~$\Dom(\omega)$, but does not contain a vertex from every path of~$Q^{n-2k}$.
    Let~$Q$ be such a $Y^n$\,--\,$S^{n-2k}$ path avoiding~$S$.
    Now~$Q$ meets~$S^i$ by construction.
    There is at least one ray~${R \in \mathcal{R}'}$ that does not contain a vertex of~$S$.
    Since~$S^i$ is contained in a component of~${G'[S^i, S^{i+1}]}$ and~${R \cap S^i \neq \emptyset}$, we can connect~$Q$ with~$R$ and hence with~${S^{m} \cap R}$ in~$G'[S^i, S^{i+1}]$ avoiding~$S$, which contradicts the assumption.
    
    Suppose~${\deg(\omega') < k}$, then~${\mathcal{R}' = \mathcal{R}}$ and hence ${S^m \cap \bigcup \mathcal{R}' = S^m}$.
    As before, there is a $Y^n$\,--\,$S^m$ path~$Q$ in~$Q^m$ not containing a vertex of~$S$.
    This path being contained in~${G'[S^{n-2k}, S^{m}]}$ would contradict the assumption.
    Hence we may assume the path meets~$S^j$ for every $j \in [n-2k, m]$ and in particular~$S^i$.
    Let~${Q_1 \subseteq Q}$ denote $Y^n$\,--\,$S^i$ path in~${G'[S^i,S^m]}$, and let~$Q_2 \subseteq Q$ denote $S^i$\,--\,$S^m$ path in~${G'[S^i, S^m]}$.
    As before, we can connect~$Q_1$ and~$Q_2$ in~${G'[S^i, S^{i+1}]}$ avoiding $S$, which again contradicts the assumption.
\end{proof}

\begin{corollary}\label{cor:s-BP:bundleP}
    In Situation~\ref{situation-sBP}, 
    let $\mathcal{R}' \subseteq \mathcal{R}$ with $\abs{\mathcal{R}'} = \min(\deg(\omega'), k)$. 
    Without loss of generality for all $n \in \mathbb{N}$ there is a $Y^n$\,--\,${(\Dom(\omega) \cup (S^{n+1} \cap \bigcup \mathcal{R}'))}$ bundle~${P^n}$ such that 
    ${P^n - \Dom(\omega) \subseteq G'[S^{n}, S^{n+1}] - S^n}$.
    Hence we include this assumption into Situation~\ref{situation-sBP}.
\end{corollary}

\begin{proof}
    We successively apply Lemma~\ref{lem:sBP:bundleP} to obtain suitable subsequences. 
    Relabelling them yields the claim.
\end{proof}

\begin{corollary}\label{cor:s-BP:fin-dom}
    Situation~\ref{situation-sBP} implies~${\dom(\omega) < k}$.
\end{corollary}

\begin{proof}
    Suppose~${\dom(\omega) \geq k}$.
    Then for every~${n > 2k}$ there is no $Y^n$\,--\,$\Dom(\omega)$ separator~$S$ of size less than~$k$ by Lemma~\ref{lem:sBP:bundleP}, since~$m$ can be chosen such that ${S \cap C(S^m, \omega') = \emptyset}$.
    Hence we can extend a path of the bundle in~$C(S^m, \omega')$.
    Therefore, for each~${n > 2k}$ there is an~${m > n}$ such that we can find a $Y^n$\,--\,$\Dom(\omega)$ bundle~$P^{n,m}$ such that \linebreak
    ${P^{n,m} - \Dom(\omega) \subseteq G'[S^{n-2k}, S^{m}]}$, and consequently an infinite subset~${I' \subseteq \mathbb{N}}$ such that ${\overline{ ( P^{n,m}\ |\ n \in I' )}}$ is an $\overline{\mathcal{Y}\restricted I'}$\,--\,$\Dom(\omega)$ bundle, contradicting the assumption~($\ast$) in Situation~\ref{situation-sBP}.
\end{proof}

This last corollary is quite impactful.
From this point onwards, we know that $\omega' = \omega\restricted(G-\Dom(\omega))$ by Remark~\ref{rem:end_subgraph}.

\begin{lemma}\label{lem:sBP:A3}
    In Situation~\ref{situation-sBP}, we may assume without loss of generality that 
    for all $n \in \mathbb{N}$ the following hold:
    \begin{itemize}
        \item ${H^n - \Dom(\omega) \subseteq G'[S^n, S^{n+1}] - (S^n \cup S^{n+1})}$;
        \item ${H^n \cap \Dom(\omega) = D_\ell \subseteq Y^n}$.
    \end{itemize}
    Hence we include these assumptions into Situation~\ref{situation-sBP}.
\end{lemma}

\begin{proof}
    Note that $H^n \cap G'[S^n, S^{n+1}] - (S^n \cup S^{n+1}) \neq \emptyset$ by Lemmas~\ref{lem:sBP:A1} and~\ref{lem:sBP:A2}.
    We delete the finitely many branch sets of vertices corresponding to the infinite side of~$K_{k,\kappa_n}$ in~$H^n$ containing a vertex of $\Dom(\omega)$, $S^n$ or $S^{n+1}$.
    Since the remaining inflated subgraph is connected, no branch set of the infinite side meets a vertex outside of $G'[S^n, S^{n+1}]$.
    Moreover, for all but finitely many $n \in \mathbb{N}$ the branch sets of vertices corresponding to the finite side of $H^n$ that meet~$\Dom(\omega)$ are precisely the singletons of the elements in $D_\ell$ by Corollary~\ref{cor:s-BP:fin-dom}.
    Deleting the exceptions and relabelling accordingly yields the claims.
\end{proof}

The next lemma reroutes some rays to find a bundle from $Y^n$ to those new rays and dominating vertices with some specific properties.

\begin{lemma}\label{lem:sBP:bundleQ}
    In Situation~\ref{situation-sBP}, there is a set $\mathcal{R}''$ of $\abs{\mathcal{R}'}$ disjoint $\omega'$-rays in $G'$ and a~$Y^n$\,--\,$(\bigcup \mathcal{R}'' \cup \Dom(\omega))$ bundle $Q^n$ for each $n \in \mathbb{N}$
    such that for every $R'' \in \mathcal{R}''$
    \begin{itemize}
        \item there is an $R' \in \mathcal{R}'$ with $V(R'') \cap \bigcup \mathcal{S} = V(R') \cap \bigcup \mathcal{S}$; and
        \item ${\abs{Q^n \cap R''} \leq 2}$ for every~${n \in \mathbb{N}}$.
    \end{itemize}
    Hence we include references to these objects into Situation~\ref{situation-sBP}.
\end{lemma}

\begin{proof}
    Given $n \in \mathbb{N}$, let $P^n$ be as in Corollary~\ref{cor:s-BP:bundleP}.
    Let $\mathcal{P}$ be a set of $\abs{\mathcal{R}'}$ disjoint $S^n$\,--\,$S^{n+1}$ paths in $G'[S^n, S^{n+1}]$ each with end vertices $R' \cap (S^n \cup S^{n+1})$ for some~$R' \in \mathcal{R}'$.
    We call such a set $\mathcal{P}$ \emph{feasible}.
    For a feasible $\mathcal{P}$, let 
    $P^n(\mathcal{P})$ denote the $Y^n$\,--\,$(\Dom(\omega) \cup \bigcup \mathcal{P})$ bundle contained in~$P^n$ 
    and let ${p^n(\mathcal{P})}$ denote the finite parameter ${\abs{(P^n - P^n(\mathcal{P})) - \bigcup\mathcal{P}}}$.
    Note that $\{ R' \cap G'[S^{n}, S^{n+1}]\ |\ R' \in \mathcal{R}' \}$ is a feasible set.
    Now choose a feasible $\mathcal{P}^n$ such that~$p^n(\mathcal{P}^n)$ is minimal and let~${Q^n := P^n(\mathcal{P}^n)}$.
    
    Assume for a contradiction that there is a path $P \in \mathcal{P}^n$ with $\abs{Q^n \cap P} > 2$.
    Let $v_0$, $v_1$ and $v_2$ denote vertices in this intersection such that $v_1 \in V(v_0 P v_2)$.
    We replace the segment $v_0 P v_2$ by the path consisting of the paths 
    $Q_i^n$ and $Q_j^n$ that contain $v_0$ and $v_2$ respectively, 
    as well as any $y_i^n$\,--\,$y_j^n$ path in~$H^n$ avoiding the finite set 
    ${\Dom(\omega) \cup Q^n \cup S^n \cup S^{n+1}}$.
    The resulting set $\mathcal{P}$ is again feasible and the parameter $p^n(\mathcal{P})$ is strictly smaller than $p^n(\mathcal{P}^n)$, contradicting the choice of $\mathcal{P}^n$.
    
    Now let $\mathcal{R}''$ be the set of components in the union $\bigcup \{ \mathcal{P}^n\ |\ n \in \mathbb{N} \}$.
    Indeed, this is a set of $\omega'$-rays that together with the bundles $Q^n$ satisfy the desired properties.
\end{proof}

For~${m,n \in \mathbb{N}}$, we say~$Q^m$ and~$Q^n$ 
\emph{follow the same pattern}, if for all $i, j \in [0,k)$
\begin{itemize}
    \item $Q_i^m$ and~$Q_i^n$ either meet the same ray in~${\mathcal{R}''}$ or the same vertex in~${\Dom(\omega)}$;
    \item if~$Q_i^m$ and~$Q_j^m$ both meet some~${R \in \mathcal{R}''}$ and~$Q_i^m$ meets~$R$ closer to the start vertex of~$R$ than~$Q_j^m$, then~$Q_i^n$ meets~$R$ closer to the start vertex of~$R$ than~$Q_j^n$.
\end{itemize}

\begin{lemma}\label{lem:sBP:A4}
    In Situation~\ref{situation-sBP}, 
    we may assume without loss of generality that 
    \begin{itemize}
        \item there are integers ${k_0, k_1, f \in \mathbb{N}}$ 
    with ${1 \leq k_0 \leq \deg(\omega')}$, ${0 \leq \ell + f + k_1 \leq \dom(\omega)}$ and ${\ell + f + k_0 + k_1 = k}$;
        \item there is a subset $\mathcal{R}_0 \subseteq \mathcal{R}''$ with $\abs{\mathcal{R}_0} = k_0$; and
        \item there are disjoint $D_f, D_1 \subseteq \Dom(\omega) \!\setminus\! D_\ell$ with $\abs{D_f} = f$ and ${\abs{D_1} = k_1}$;
    \end{itemize}
    such that for all $m, n \in \mathbb{N}$ 
    \begin{enumerate}[label=(\alph*)]
        \item \label{item:sBP:A4-a}
            $Q^n$ is a $Y^n$\,--\,$(\bigcup\mathcal{R}_0 \cup D_\ell \cup D_f)$ bundle;
        \item \label{item:sBP:A4-b}
            $Q^n \cap \Dom(\omega) = D_f \cup D_\ell$; and
        \item \label{item:sBP:A4-c}
            $Q^m$ and $Q^n$ follow the same pattern.
    \end{enumerate}
    Hence we include these assumptions and references to the existing objects into Situation~\ref{situation-sBP}.
\end{lemma}

\begin{proof}
    Using the fact that $\Dom(\omega)$ is finite, we apply the pigeonhole principle to find a set ${D_f \subseteq \Dom(\omega) \setminus D_\ell}$ and an infinite subset~${I \subseteq \mathbb{N}}$ such that~\ref{item:sBP:A4-b} hold for all~${n \in I}$.
    Set~${f := \abs{D_f}}$.
    Applying it multiple times again, we find an infinite subset~${I' \subseteq I}$ such that~\ref{item:sBP:A4-c} holds for all~${m, n \in I'}$.
    If~${\abs{\mathcal{R}''} \geq k - \ell - f}$, then set~$\mathcal{R}_0$ to be any subset of~$\mathcal{R}''$ of size~${k - \ell - f}$ containing each ray that meets~$Q^n$ for any~${n \in I'}$. 
    Otherwise set~${\mathcal{R}_0 := \mathcal{R}''}$ and set ${k_0 := \abs{\mathcal{R}_0} = \deg(\omega') = \deg(\omega)}$.
    Now~\ref{item:sBP:A4-a} holds by the choices of $D_f$ and $\mathcal{R}_0$.
    Since~${\Delta(\omega) \geq k}$ by Lemma~\ref{lem:Delta>=k}, there is a set ${D_1 \subseteq \Dom(\omega) \!\setminus\! (D_\ell \cup D_f)}$ of size~${k_1 := k - \ell - f - k_0}$, completing the proof.
\end{proof}

Finally, we construct the subdivision of a generalised $k$-typical graph for some singular $k$-blueprint.

\begin{lemma}\label{lem:sBP-minor}
    In Situation~\ref{situation-sBP}, there is a singular $k$-blueprint~${\mathcal{B} = (\ell,f,B,D)}$ for a tree~$B$ of order~${k_0 + k_1}$ with~${\abs{D} = k_0}$, such that~$G$ contains a subdivision of a generalised ${T_k(\mathcal{B})(\mathcal{K})}$ with core~${\bigcup \mathcal{A}}$.
\end{lemma}

\begin{proof}
    We apply Lemma~\ref{lem:T-connected} to~$\mathcal{R}_0$ and~$D_1$ 
    to obtain a simple type-2 $k$-template~$\mathcal{T}_2$, a tree~$B$ of order~${k_0 + k_1}$ and a set~${D \subseteq V(B)}$ with~${\abs{D} = k_1}$ 
    such that $(\mathcal{R}_0, D_1)$ is $(B, \mathcal{T}_2)$-connected.
    For each ${n \in \mathbb{N}}$ let~$\Gamma_n$ denote a ${(B, \mathcal{T}_2)}$-connection avoiding~${S^n}$, ${\Dom(\omega) \!\setminus\! D_1}$ as well as for each $R \in R_0$ its initial segment~$Rs$ for~${s \in (S^n \cap V(R))}$. 
    Note that there is an~${m > n}$ such that ${\Gamma_n - D_1 \subseteq G'[S^{n}, S^{m}] - S^m}$.
    Hence~$\Gamma_n$ and~$\Gamma_{m+1}$ are disjoint to $G'[S^{m}, S^{m+1}] \supseteq Q^m$.
    For rays ${R \in \mathcal{R}_0}$ with ${\abs{Q^m \cap R} \geq 1}$, 
    we extend~$\Gamma_n$ on that ray to include precisely one vertex in the intersection as well as with the corresponding path in~$Q^m$ to~$Y^m$.
    If furthermore ${\abs{Q^m \cap R} = 2}$, we also extend~$\Gamma_{m+1}$ on that ray to include the other vertex of the intersection and with the corresponding path in~$Q^m$ to~$Y^m$.
    Since $Q^m$ and $Q^n$ follow the same pattern for all $m,n \in \mathbb{N}$ by Lemma~\ref{lem:sBP:A4}, we can modify~$\mathcal{T}_2$ to~$\mathcal{T}_2'$ accordingly to have infinitely many~${(B, \mathcal{T}_2')}$-connections which pairwise meet only in~$D_1$ and contain~$Y^n$ for each~${n \in I}$ for some infinite subset~${I \subseteq \mathbb{N}}$.
    After relabelling and setting $\mathcal{B} := (\ell,f,B,D)$, we obtain the subdivision of~${T_k(\mathcal{B})(\mathcal{T}_2')}$ as in the proof of Corollary~\ref{cor:BP-sub}.
\end{proof}

\subsection{Characterisation for singular cardinals}\label{subsec:singular-thm}
\ \newline \indent
Now we have developed all the necessary tools to prove the minor and topological minor part of the characterisation in Theorem~\ref{main-thm} for singular cardinals.

\begin{theorem}\label{thm:main-singular}
    Let~$G$ be a graph, let~${k \in \mathbb{N}}$, 
    let~${A \subseteq V(G)}$ be infinite 
    and let ${\kappa \leq \abs{A}}$ be a singular cardinal.
    Then the following statements are equivalent.
    \begin{enumerate}[label=(\alph*)]
        \item \label{item:t7-set}
            There is a subset~${A_1 \subseteq A}$ with~${\abs{A_1} = \kappa}$ such that~$A_1$ is $k$-connected in~$G$.
        \item \label{item:t7-minor}
            There is a subset~${A_2 \subseteq A}$ with~${\abs{A_2} = \kappa}$ such that
            \begin{itemize}
                \item[$\bullet$] either~$G$ contains 
                    an $\ell$-degenerate frayed~$K_{k,\kappa}$ as an fbs-minor 
                    with~$A_2$ along its core 
                    for some $0 \leq \ell \leq k$; 
                \item[$\bullet$] or~$T_k(\mathcal{B})$ is an fbs-minor of~$G$ 
                    with~$A_2$ along its core 
                    for a singular 
                    $k$-blueprint~$\mathcal{B}$.
            \end{itemize}
        \item \label{item:t7-subdivision}
            There is a subset~${A_3 \subseteq A}$ with~${\abs{A_3} = \kappa}$ such that
            \begin{itemize}
                \item[$\bullet$] either $G$ contains a subdivision of a generalised~$\lFK$ 
                    with core~$A_3$
                    for some $0 \leq \ell \leq k$;
                \item[$\bullet$] or $G$ contains the subdivision of a generalised~$T_k(\mathcal{B})$
                    with core~$A_3$ 
                    for some singular 
                    $k$-blueprint~$\mathcal{B}$.
            \end{itemize}
    \end{enumerate}
    Moreover, if these statements hold, we can choose~${A_1 = A_2 = A_3}$.
\end{theorem}

\begin{proof}
    If~\ref{item:t7-minor} holds, 
    then~$A_2$ is 
    $k$-connected in~$G$ by 
    Lemma~\ref{lem:k-con-minor} with 
    Lemma~\ref{lem:typical-k-conn}.
    
    Suppose~\ref{item:t7-set} holds.
    Either we can find a subset ${A_3 \subseteq A_1}$ with ${\abs{A_3} = \kappa}$ and a subdivision of $\lFK(\mathcal{K})$ with core $A_3$ for some good $\kappa$-sequence~$\mathcal{K}$ by Lemma~\ref{lem:lFK} and either Lemma~\ref{lem:lFK-no-end} or Corollary~\ref{cor:lFK-end}.
    Otherwise, we can apply Lemma~\ref{lem:sBP-minor} to obtain ${A_3 \subseteq A_1}$ with ${\abs{A_3} = \kappa}$ and a subdivision of ${T_k(\mathcal{B})(\mathcal{K})}$ with core~$A_3$
    for some singular $k$-blueprint~$\mathcal{B}$ and a good $\kappa$-sequence~$\mathcal{K}$.
    With Remark~\ref{rem:cof2} we obtain the subdivision of the respective generalised $k$-typical graph with respect to the fixed good $\kappa$-sequence.
    
    If~\ref{item:t7-subdivision} holds, then so does~\ref{item:t7-minor} by Lemma~\ref{lem:gen-sub-fbs} with~${A_2 := A_3}$.
    Moreover, $A_3$ is a candidate for both~$A_2$ and~$A_1$.
\end{proof}

\section{Applications of the minor-characterisation}
\label{sec:applications}

In this section we will present some applications of the minor-characterisation of $k$-connected sets.

As a first corollary we just restate the theorem for $k=1$, giving us a version of the Star-Comb Lemma for singular cardinalities.
For this, given a singular cardinal~$\kappa$, we call the graph~${FK_{1, \kappa}}$ a \emph{frayed star}, whose \emph{centre} is the vertex~$x_0$ of degree~${\cf\kappa}$ and whose \emph{leaves} are the vertices~$\bigcup \mathcal{Z}$. 
Moreover, we call a generalised $1$-typical graph obtained from the single vertex tree (i.e.~a generalised~${T_1(0, 0, (\{c\}, \emptyset), \emptyset, 0 \mapsto (c,0))}$) a \emph{frayed comb} with \emph{spine}~$\mathfrak{N}_c$ and \emph{teeth}~${\bigcup \mathcal{Z}}$. 

\begin{corollary}[Frayed-Star-Comb Lemma]\label{frayed-star-comb}
    Let~${U \subseteq V(G)}$ be infinite 
    and let~${\kappa \leq \abs{U}}$ be a singular cardinal.
    Then the following statements are equivalent.
    \begin{enumerate}[label=(\alph*)]
        \item \label{item:FSC-a} There is a subset $U_1 \subseteq U$ with $\abs{U_1} = \kappa$ such that $U_1$ is $1$-connected in~$G$.
        \item \label{item:FSC-b} There is a subset $U_2 \subseteq U$ with $\abs{U_2} = \kappa$ such that $G$ either contains a subdivided star or frayed star whose set of leaves is~$U_2$, or a subdivided frayed comb whose set of teeth is~$U_2$.
        
        (Note that if $\cf\kappa$ is uncountable, only one of the former two can exist.)
    \end{enumerate}
    Moreover, if these statements hold, we can choose~${U_1 = U_2}$.\qed
\end{corollary}

%Even though this Frayed-Star-Comb Lemma has a much more elementary proof, we state it here only as a corollary of our main theorem.

Even though the Frayed-Star-Comb Lemma is a direct corollary of our main theorem, we include a much more elementary proof here as well. 

\begin{proof}
    As usual, if~\ref{item:FSC-b} holds, then~$U_2$ is $1$-connected and we can set ${U_1 := U_2}$ to satisfy~\ref{item:FSC-a}. 
    
    If~\ref{item:FSC-a} holds, then we take a tree ${T \subseteq G}$ containing~$U_1$ such that each edge of~$T$ lies on a path between two vertices of~$U_1$ as in the proof of Lemma~\ref{star-comb}. 
    
    If~$T$ has a vertex~$c$ of degree~$\kappa$, then this yields a subdivided star with centre~$c$ and a set~${U_2 \subseteq U_1}$ of leaves with~${\abs{U_2} = \kappa}$ by extending each incident edge of~$c$ to a $c$\,--\,$U_1$ path. 
    Otherwise, there is a sequence of distinct vertices~${W := ( w_\alpha\ |\ \alpha \in \cf\kappa )}$ such that the degree of~$w_\alpha$ is at least~$\kappa_\alpha$ for any good $\kappa$-sequence~${( \kappa_\alpha\ |\ \alpha \in \cf\kappa )}$. 
    As above, $w_\alpha$ is the centre of a subdivided star~$S_\alpha$ with a set~$U^\alpha \subseteq U_1$ of leaves of size~$\kappa_\alpha$. 
    Since~$W$ has size~$\cf\kappa$, we may assume without loss of generality that each~$S_\alpha$ is disjoint from~$W$ and hence that~$S_\alpha$ and~$S_\beta$ are disjoint for distinct~$\alpha$ and~$\beta$. 
    Applying Lemma~\ref{star-comb} in~$T$ to~${W}$ yields a subdivided star with leaves in~$W$ or a subdivided comb with teeth in~$W$. 
    Since this subdivided star or comb has size~$\cf\kappa$ as before we may assume that each~$S_\alpha$ meets it only in~$w_\alpha$. 
    The union of all these graphs yields the desired subdivision. 
\end{proof}

\vspace{0.2cm}

Now Theorems~\ref{thm:main-regular} and~\ref{thm:main-singular} give us the tools to prove the statement we originally wanted to prove instead of Lemma~\ref{lem:delete<k}.

\begin{corollary}\label{cor:(k-1)-connected}
    Let~${k \in \mathbb{N}}$, let~${A \subseteq V(G)}$ be infinite and $k$-connected in~$G$ 
    and let ${\kappa \leq \abs{A}}$ be an infinite cardinal.
    Then for every $v \in V(G)$ there is a subset ${A' \subseteq A}$ with ${\abs{A'} = \kappa}$ such that~$A'$ is~${(k-1)}$-connected in~${G-v}$.
\end{corollary}

\begin{proof}
    First we apply Theorem~\ref{thm:main-regular} or Theorem~\ref{thm:main-singular} to~$A$ to get a $k$-typical graph~$T$ and an inflated subgraph~$H$ witnessing that~$T$ is an fbs-minor of~$G$ with some ${A'' \subseteq A}$ along its core such that ${\abs{A''} = \kappa}$.
    Let us call a vertex of $T$ \emph{essential}, if either
    \begin{itemize}
        \item it is a vertex of the finite side of~$K_{k,\kappa}$ if ${T = K_{k,\kappa}}$;
        \item it is a degenerate vertex or frayed centre of $\lFK$ if~${T = \lFK}$ for some ${\ell \in [0,k]}$; or
        \item it is a dominating vertex, a degenerate vertex or a frayed centre of~${T_k(\mathcal{B})}$ if~${T = T_k(\mathcal{B})}$ for some regular or singular $k$-blueprint~$\mathcal{B}$. 
    \end{itemize}
    We distinguish four cases.
    
    If~${v \notin V(H)}$, then ${H \subseteq G - v}$ still witnesses that~$T$ is an fbs-minor of ${G-v}$ with ${A' := A''}$ along its core.
    
    If~$v$ belongs to a branch set of a vertex~$c$ of the core, then the inflated subgraph obtained by deleting that branch set still yields a witness that~$T$ is an fbs-minor of~${G-v}$ with ${A' := A'' \setminus \{ c \}}$ along its core. 
    
    If~$v$ belongs to a branch set of an essential vertex~${w \in V(T)}$, 
    then the inflated subgraph where we delete this branch set from~$H$ witnesses that the obvious $(k-1)$-typical subgraph of ${T - w}$ is an fbs-minor of ${G-v}$ with ${A' := A''}$ along its core.
    
    If~$v$ belongs to a branch set of a vertex~${w \in V(\mathfrak{N}(B/D) - D)}$, then we delete the branch sets of the layers (not including $D$) up to the layer containing~$w$ and relabelling accordingly (and modifying the $\kappa$-sequence if necessary).
    This yields a supergraph of an inflated subgraph witnessing that~$T$ is an fbs-minor of~${G-v}$ with~$A'$ along its core for some~${A' \subseteq A''}$~with ${\abs{A'} = \kappa}$.
    Similar arguments yield the statement if~$v$ belongs to a branch set of a neighbour of a frayed centre.
    
    In any case, with the other direction of Theorem~\ref{thm:main-regular} or Theorem~\ref{thm:main-singular} we get that $A'$ is $(k-1)$-connected in~${G-v}$.
\end{proof}

As another corollary we prove that we are able to find $k$-connected sets of size~$\kappa$ in sets which cannot be separated by less than~$\kappa$ many vertices from another $k$-connected set.
This will be an important tool for our last part of the characterisation in the main theorem.

\begin{corollary}\label{cor:inseparableA}
    Let~${k \in \mathbb{N}}$, let~${A, B \subseteq V(G)}$ be infinite and let~${\kappa \leq \abs{A}}$ be an infinite cardinal.
    If~$B$ is $k$-connected in~$G$ and~$A$ cannot be separated from~$B$ by less than~$\kappa$ vertices, then there is an~${A' \subseteq A}$ with~${\abs{A'} = \kappa}$ which is $k$-connected in~$G$.
\end{corollary}

\begin{proof}
    Let~$\mathcal{P}$ be a set of~$\kappa$ many disjoint $A$\,--\,$B$ paths 
    as given by Theorem~\ref{cardinality-menger}.
    Let~${B'}$ denote ${B \cap \bigcup \mathcal{P}}$.
    Let~${H \subseteq G}$ be an inflated subgraph witnessing that a $k$-typical graph is an fbs-minor of~$G$ with~$B''$ along its core for some~${B'' \subseteq B'}$ with~${\abs{B''} = \kappa}$
    as given by Theorem~\ref{thm:main-regular} or Theorem~\ref{thm:main-singular}.
    Let~$\mathcal{P}'$ denote the set of the $A$\,--\,$H$ subpaths of the $A$\,--\,$B''$ paths in~$\mathcal{P}$.
    We distinguish two cases.
    
    If the $k$-typical graph is a $T_k(\mathcal{B})$ for some regular $k$-blueprint ${\mathcal{B} = (T,D,c)}$, then (since each branch set in~$H$ is finite) there is an infinite subset ${\mathcal{P}'' \subseteq \mathcal{P}'}$ and a node ${c' \in V(T \setminus D)}$ such that each branch set in~$H$ of vertices in~${V(\mathfrak{N}_{c'})}$ meets~${\bigcup\mathcal{P}''}$ at most once and no other branch set meets~${\bigcup\mathcal{P}''}$.
    Let~${A' := \bigcup \mathcal{P}'' \cap A}$.
    We extend each of these branch sets with the path from~$\mathcal{P}''$ meeting it.
    This yields a subgraph~$H'$ witnessing that~${T_k(T,D,c')}$ is an fbs-minor of~$G$ with some~$A''$ along its core with~${A' \subseteq A''}$.
    
    Otherwise, 
    since each branch set in~$H$ is finite, there is a subset ${\mathcal{P}'' \subseteq \mathcal{P}'}$ of size~$\kappa$ such that each branch set in~$H$ of vertices corresponding to the core meets~${\bigcup\mathcal{P}''}$ at most once and no other branch set meets~${\bigcup\mathcal{P}''}$.
    Let~${A' := \bigcup \mathcal{P}'' \cap A}$.
    Again, we extend each of these branch sets with the path from~$\mathcal{P}''$ meeting it.
    This yields a subgraph~$H'$ witnessing that the same $k$-typical graph is an fbs-minor of~$G$ with some~$A''$ along its core such that ${A' \subseteq A'' \subseteq A' \cup B''}$.
    
    Applying Theorem~\ref{thm:main-regular} or Theorem~\ref{thm:main-singular} again together with Remark~\ref{rem:k-connected-subset} yields the claim.
\end{proof}

\section{Duality}
\label{sec:duality}

This section will finish the proof of Theorem~\ref{main-thm} by proving the duality theorem and hence providing the last equivalence of the characterisation. 

Before we direct our attention to tree-decompositions, we need to first need to consider more general nested separation systems. 
Recall that a nested separation system ${N \subseteq S_k(G)}$ is called \emph{$k$-lean} if given any two (not necessarily distinct) parts~$P_{1}$, $P_{2}$ of~$N$ and vertex sets ${Z_1 \subseteq P_{1}}$, ${Z_2 \subseteq P_{2}}$ with ${\abs{Z_1} = \abs{Z_2} = \ell \leq k}$ there are either $\ell$ disjoint $Z_1$\,--\,$Z_2$ paths in~$G$ or there is a separation~$(A,B)$ in~$N$ with ${P_1 \subseteq A}$ and ${P_2 \subseteq B}$ of order less than~$\ell$.

For a subset~${X \subseteq V(G)}$ consider the induced subgraph~${G[X]}$.
Every separation of~${G[X]}$ is of the form ${(A \cap X, B \cap X)}$ for some separation~$(A,B)$ of~$G$.
We denote this separation also as~${(A,B)\restricted X}$.
Given a set~$S$ of separations of~$G$ we write~${S\restricted X}$ for the set consisting of all separations~${(A,B)\restricted X}$ for~${(A,B) \in S}$. 

Consider the directed partially ordered set~$\mathcal{F}$ of finite subsets of~${V(G)}$ ordered by inclusion, as well as the directed inverse system $( S_k(G[X])\ |\ X \in \mathcal{F} )$.
\begin{obs}\label{obs:seps}
    Every separation in $S_k(G)$ is determined by all its restrictions to finite subsets of~$V(G)$.
    
    More precisely, on the one hand for each element
    ${( (A_X, B_X) \in S_k(G[X])\ |\ X \in \mathcal{F} )}$ of the inverse limit the separation 
    ${( \bigcup\{ A_X\ |\ X \in \mathcal{F}\}, \bigcup\{ B_X\ |\ X \in \mathcal{F}\} )}$ 
    is the unique separation in~${S_k(G)}$ inducing $(A_X, B_X)$ on $G[X]$ for each $X \in \mathcal{F}$.
    On the other hand, ${( (A, B) \restricted X\ |\ X \in \mathcal{F} )}$ is an element of the inverse limit for each separation~${(A,B) \in S_k(G)}$.
    For more information on this approach, see \cites{Diestel:tree-sets, DK:profinite-ss}.
\end{obs}

The following theorem lifts the existence of $k$-lean nested separation systems for finite graphs as in Theorem~\ref{thm:fin-k-lean} to infinite graphs via the Generalised Infinity Lemma~\ref{gen-inf-lemma}.

\begin{lemma}\label{lem:k-lean-treeset}
    For every graph~$G$ and every ${k \in \mathbb{N}}$ there is a nested separation system ${N \subseteq S_k(G)}$ such that $N$ 
    is $k$-lean. 
\end{lemma}

\begin{proof}
    As above, consider the directed partially ordered set~$\mathcal{F}$ of finite subsets of~${V(G)}$ ordered by inclusion.
    For every~${X \in \mathcal{F}}$ let~${\mathcal{N}(X)}$ denote the set of nested separation systems~${N\restricted X}$ of~${G[X]}$ such that there is a nested separation system~$N \subseteq S_k(G[Z])$ that is $k$-lean for a~${Z \in \mathcal{F}}$ containing~$X$.
    Note that~${\mathcal{N}(X)}$ is not empty by Theorem~\ref{thm:fin-k-lean} for every ${X \in \mathcal{F}}$.
    Moreover, for every ${Y \subseteq X \in \mathcal{F}}$ there is a natural map ${f_{X,Y}: \mathcal{N}(X) \to \mathcal{N}(Y)}$ defined by ${f_{X,Y}(N) := N \restricted Y}$.
    It is easy to check that this defines a directed inverse system of finite sets.
    By the Generalised Infinity Lemma~\ref{gen-inf-lemma} the inverse limit of that system is non-empty and contains an element ${( N_X\ |\ X \in \mathcal{F} )}$.
    
    Let ${N  := \{ (A,B) \in S_k(G)\ |\ (A,B) \restricted X \in N_X \text{ for all } X \in \mathcal{F} \}}$.
    Note that $N$ is non-empty and contains for each $(A,B) \in N_X$ at least one separation inducing $(A,B)$ on $G[X]$ by Observation~\ref{obs:seps}.
    It is easy to check that~$N$ is a nested separation system since the fact that two separations are crossing is witnessed in a finite set of vertices.
    
    For two (not necessarily distinct) parts~$P_1$,~$P_2$ of~${N}$ and vertex sets~${Z_1 \subseteq P_1}$ and~${Z_2 \subseteq P_2}$ with~${\abs{Z_1} = \abs{Z_2} = \ell \leq k}$, 
    we consider a largest possible set~$\mathcal{P}$ of disjoint $Z_1$\,--\,$Z_2$ paths in~$G$.
    We may assume that ${\abs{\mathcal{P}} < \ell}$, since otherwise there is nothing to show.
    For every ${X \in \mathcal{F}}$ containing ${Z := Z_1 \cup Z_2 \cup V(\bigcup \mathcal{P})}$, 
    note that $N_X$ is the restriction of some $k$-lean tree set of a finite supergraph $G[X']$ of $G[X]$ to $X$.
    Hence there is a separation $(A,B)$ of $G[X']$ of order~$\abs{\mathcal{P}}$ 
    separating~$Z_1$ and~$Z_2$, 
    whose restriction $(A,B)\restricted X$ is non-trivial and in~${N_X}$ by the choice of $X$.
    For each finite $X, Y \subseteq V(G)$ with $Z \subseteq Y \subseteq X$, each such separation~$(C,D)$ induces a separation ${(C,D) \restricted Y \in N_Y}$ of order~$\abs{\mathcal{P}}$ separating $Z_1$ and $Z_2$.
    Applying the Generalised Infinity Lemma~\ref{gen-inf-lemma} again yields an element~${( (A_X, B_X) \in N_X\ |\ X \in \mathcal{F})}$ from the inverse limit.
    By Observation~\ref{obs:seps}, this element corresponds to a separation of order $\abs{\mathcal{P}}$ of $G$ which by construction separates $Z_1$ and $Z_2$ and is an element of $N$.
    Hence $N$ is $k$-lean.
\end{proof}

Given a nested separation system~$N$ of a graph~$G$, we define the \emph{clean-up of~$N$} as the separation system consisting of those separations~${(A',B')}$ of~$G$ for which there is a separation~${(A,B) \in N}$ such that
\begin{itemize}
    \item either~$A' \setminus B'$ or~$B' \setminus A'$ is a component~$C$ of~${G-(A \cap B)}$; and
    \item ${A' \cap B' = N(C)}$.
\end{itemize}

\begin{lemma}
    \label{rem:treeset-clean-up}
    Let~$G$ be a graph, let~$k \in \mathbb{N}$, and let~${N \subseteq S_k(G)}$ be nested separation system. 
    Then the following statements are true. 
    \begin{enumerate}
        [label=(\alph*)]
        \item\label{item:ts-cleanup1} The cleanup~$\hat{N}$ of~$N$ is a nested separation system contained in~${S_k(G)}$. 
        \item\label{item:ts-cleanup4} No separation in~$\hat{N}$ is of the form~${(A,V(G))}$. 
        \item\label{item:ts-cleanup2} Each part of~$\hat{N}$ is contained in a part of~$N$. 
        \item\label{item:ts-cleanup3} No three separations in~$\hat{N}$ that share the same separator are in a common chain. 
    \end{enumerate}
\end{lemma}

\begin{proof}
    Clearly, each separator of a separation in~$\hat{N}$ is contained in a separator of a separation in~$N$ and hence~${\hat{N} \subseteq S_k(G)}$. 
    Let~${(A'_0,B'_0), (A'_1,B'_1) \in \hat{N}}$ such that~${A'_i \setminus B'_i}$ is a component~$C_i$ of~${G-(A_i \cap B_i)}$ for~${i \in \{ 0,1 \}}$ and respective separations~${(A_i,B_i) \in N}$. 
    Let~${S_i := A_i \cap B_i}$ and let~${S'_i := N(C_i)}$. 
    Note that since~$N$ is nested, neither~$S'_0$ separates~$S'_1$ nor vice versa, 
    and hence for each~${i \in \{0,1\}}$ we obtain that~$S'_i$ is contained in either~${A'_{1-i}}$ or~${B'_{1-i}}$. 
    In this case, to check whether~${(A'_0,B'_0)}$ and~${(A'_1,B'_1)}$ are nested it is enough to verify that one side of one of them is a subset of one side of the other. 
    Moreover, note that if~$S'_{1}$ is contained in~$B'_{0}$, then since~$G[A'_{0}]$ is connected, we conclude that~${A'_{0} \setminus S'_{1}}$ is contained in a unique component of~${G - S'_{1}}$, and hence~$A'_{0}$ is a subset of either~$A'_{1}$ or~$B'_{1}$. 
    So suppose that~${A'_{0}}$ is not contained in either~${A'_{1}}$ or~${B'_{1}}$, and hence~${S'_{1} \subseteq A'_{0}}$. 
    But now each component of~$G - S'_{0}$ other than~$C_{0}$ are contained in the same side of~${(A'_1,B'_1)}$ as~$S'_{0}$, and hence~$B'_0$ is a subset of that side. 
    We conclude that~$\hat{N}$ is nested and hence~\ref{item:ts-cleanup1} holds. 
    
    Note that \ref{item:ts-cleanup4} immediately follows from the definition since the empty set is never a component of a non-empty graph. 
    
    For~\ref{item:ts-cleanup2}, let~$P$ be a part of~$\hat{N}$. 
    Then no separation in~$\hat{N}$ separates~$P$, and by definition of~$\hat{N}$ no separation in~$N$ separates~$P$. 
    Hence~$P$ is contained in a part of~$N$. 
    
    Lastly, assume for a contraction that there are separations~${(A'_i,B'_i) \in \hat{N}}$ for~${i \in \{ 0, 1, 2 \}}$ that share the same separator~$S$ for which~${(A'_0,B'_0) < (A'_1,B'_1) < (A'_2,B'_2)}$. 
    Since no separation in~$\hat{N}$ is comparable with its inverse, by definition of~$\hat{N}$ there are three distinct components~$C_0$, $C_1$ and $C_2$ of~${G - S}$ for which~${C_i \in \{ A'_i \setminus B'_i, B'_i \setminus A'_i \}}$ for all~${i \in \{ 0, 1, 2 \}}$. 
    Now it is easy to see that no possible choices for the~$C_i$ satisfy~${(A'_0,B'_0) < (A'_1,B'_1) < (A'_2,B'_2)}$, proving~\ref{item:ts-cleanup3}. 
\end{proof}

\begin{lemma}\label{lem:k-lean-ts2td}
    Let~$G$ be a graph, let~$k \in \mathbb{N}$, and let~${N \subseteq S_k(G)}$ be a~$k$-lean nested separation system.
    Then the clean-up~$\hat{N}$ of~$N$ is~$k$-lean and contains no chains of order type~${\omega+1}$. 
\end{lemma}

\begin{proof}
    Suppose there are parts~$P_{1}$,~$P_{2}$ of~$\hat{N}$ and vertex sets ${Z_1 \subseteq P_{1}}$, ${Z_2 \subseteq P_{2}}$ with ${\abs{Z_1} = \abs{Z_2} = \ell \leq k}$ for which there are no~$\ell$ disjoint $Z_1$--$Z_2$ paths in~$G$.  
    By Lemma~\ref{rem:treeset-clean-up} and the $k$-leanness of~$N$ there is a separation~${(A,B) \in N}$ with~$Z_1 \subseteq A$ and~$Z_2 \subseteq B$ and~$\abs{A \cap B} < \ell$. 
    Let~${S := A \cap B}$. 
    Since~$P_1$ and~$P_2$ are parts of~$\hat{N}$, there are distinct components~$C_1$ and~$C_2$ of~${G - S}$ with~$P_1 \subseteq C_1 \cup S$ and~$P_2 \subseteq C_2 \cup S$. 
    But then~${(C_1 \cup S, V(G) \setminus C_1) \in \hat{N}}$ is as desired. 
    Hence~$\hat{N}$ is indeed $k$-lean.
    
    We show that between any two separations~${(A,B), (A',B') \in \hat{N}}$ there are only finitely many separations in~$\hat{N}$. 
    Suppose for a contradiction that~${(A,B)}$ and~${(A',B')}$ form a counterexample for which the number~$m$ of disjoint ${(A \cap B)}$--${(A' \cap B')}$ paths is maximal.
    Let~$\mathcal{P}$ be a set containing $m$ such paths.
    By this condition, infinitely many elements of~$\hat{N}$ between~${(A,B)}$ and~${(A',B')}$ have order~$m$. 
    But since each of them contains a vertex from each path in~$\mathcal{P}$, infinitely many of them share the same separator by the pigeon hole principle, contradicting that no three of them are in a common chain. 
\end{proof}

Lemma~\ref{lem:k-lean-treeset} and Lemma~\ref{lem:k-lean-ts2td} yield together with Theorem~\ref{thm:ts2td} the existence of $k$-lean tree-decompositions for any graph. 

\begin{theorem}\label{thm:k-lean}
    For every graph~$G$ and every~${k \in \mathbb{N}}$ there is a $k$-lean tree-decomposition of~$G$. \qed
\end{theorem}

Now we are able to prove the duality theorem and hence the remaining equivalence of our main theorem.

\begin{theorem}\label{thm:duality}
    Let~$G$ be an infinite graph, let~${k \in \mathbb{N}}$, 
    let~${A \subseteq V(G)}$ be infinite 
    and let~${\kappa \leq \abs{A}}$ be an infinite cardinal.
    Then the following statements are equivalent.
    \begin{enumerate}[label=(\alph*)]
        \item \label{item:t9-set}
            There is a subset~${A_1 \subseteq A}$ with~${\abs{A_1} = \kappa}$ such that~$A_1$ is $k$-connected in~$G$.
        \addtocounter{enumi}{2}
        \item \label{item:t9-duality} 
            There is no tree-decomposition of~$G$ of adhesion less than~$k$ such that every part can be separated from~$A$ by less than~$\kappa$ vertices. 
    \end{enumerate}
\end{theorem}

\begin{proof}
    Assume that~\ref{item:t9-set} does not hold.
    Let~$\mathcal{T}$ be a $k$-lean tree-decomposition system as obtained from Theorem~\ref{thm:k-lean}.
    Suppose for a contradiction that there exists a part~$P$ of~$\mathcal{T}$ that cannot be separated from $A$ by less than~$\kappa$ vertices.
    Then~$P$ is $k$-connected in~$G$ and has size at least~$\kappa$.
    By Corollary~\ref{cor:inseparableA}, there is a subset ${A_1 \subseteq A}$ of size~$\kappa$ which is $k$-connected in~$G$, a contradiction.
    Hence every part of~$\mathcal{T}$ can be separated from~$A$ by less than~$\kappa$ vertices, so~\ref{item:t9-duality} does not hold.
    
    If~\ref{item:t9-set} holds, let $\mathcal{T}$ be any tree-decomposition of~$G$ of adhesion less than~$k$ and let~$H$ be an inflated subgraph witnessing that a $k$-typical graph~$T$ is an fbs-minor of~$G$ with some~${A' \subseteq A}$ along its core for~${\abs{A'} = \kappa}$ as in Theorem~\ref{thm:main-regular} or Theorem~\ref{thm:main-singular}. 
    
    If~${T = T_k(B,D,c)}$ for some regular $k$-blueprint $(B,D,c)$, then since~$T$ contains~$k$ disjoint paths between~${B_i \cup D}$ and~${B_j \cup D}$ for all~${i, j \in \mathbb{N}}$, no separation of~$G$ of order less than~$k$ can separate the unions of the branch sets corresponding to the vertices of the layers~${B_i \cup D}$ and~${B_j \cup D}$. 
    Hence there is a part of~$\mathcal{T}$ containing at least one vertex in a branch set corresponding to some vertex of every layer of~$T$.
    
    In every other case~$T$ contains~$k$ internally disjoint paths between any two core vertices. 
    Hence there cannot exist a separation of~$G$ of order less than~$k$ that separates two distinct branch sets containing vertices of the core, and therefore there is a part of~$\mathcal{T}$ containing at least one vertex from each branch set corresponding to the core of~$T$. 
    
    In any case, this part has to have size at least~$\kappa$, and the disjoint paths in each branch set from a vertex of~$A'$ to the part witness by Theorem~\ref{cardinality-menger} that~$A$ cannot be separated by less than~$\kappa$ vertices from that part. 
    Since~$\mathcal{T}$ was arbitrarily chosen,~\ref{item:t9-duality} holds. 
\end{proof}

\section*{Acknowlegement}

J.~Pascal Gollin was supported by the Institute for Basic Science (IBS-R029-C1).

Karl Heuer was supported by the European Research Council (ERC) under the European Union's Horizon 2020 research and innovation programme (ERC consolidator grant DISTRUCT, agreement No.\ 648527).

\end{document}